\documentclass[final,onefignum,onetabnum]{siamart190516}

\usepackage{verbatim}
\usepackage{lipsum}
\usepackage{amsfonts}
\usepackage{graphicx}
\usepackage{subfigure}
\usepackage{epstopdf}
\usepackage{algorithmic}
\ifpdf
  \DeclareGraphicsExtensions{.eps,.pdf,.png,.jpg}
\else
  \DeclareGraphicsExtensions{.eps}
\fi

\renewcommand\theequation{\thesection.\arabic{equation}}
\theoremstyle{plain}

\newtheorem{remark}{\textbf{Remark}}[section]
\newcommand{\eps}{\epsilon}

\newcommand{\bm}{\boldsymbol}

\newcommand{\Grad}[1]{\nabla #1}

\newcommand{\be}{\begin{equation}}
\newcommand{\ee}{\end{equation}}

\newcommand{\bse}{\begin{subequations}}
\newcommand{\ese}{\end{subequations}}
\def\benl{\begin{eqnarray*}}
\def\eenl{\end{eqnarray*}}

\def\be{\bm{e}}

\def\bx{\bm{x}}

\def\bmu1{\bm{\mu_1}}

\newcommand{\ben}{\begin{eqnarray}}
\newcommand{\een}{\end{eqnarray}}
\newcommand{\beq}{\begin{equation}}
\newcommand{\eeq}{\end{equation}}
\newcommand{\bea}{\begin{array}}
\newcommand{\eea}{\end{array}}
\newcommand{\bef}{\begin{figure}[H]}
\newcommand{\eef}{\end{figure}}

\newsiamremark{hypothesis}{Hypothesis}
\crefname{hypothesis}{Hypothesis}{Hypotheses}
\newsiamthm{claim}{Claim}

\title{Modeling and simulation of nuclear architecture  reorganization  process using a phase field  approach\thanks{ The work  is partially supported by NSF Grant DMS-1720442 and AFOSR Grant FA9550-16-1-0102.}}

\author{Qing Cheng\thanks{Department of  Mathematics,Purdue University, West Lafayette, IN 47907, USA  (cheng573$@$purdue.edu, shen7$@$purdue.edu).}\and Pourya Delafrouz\thanks{Department of Bioengineering, University of Illinois at Chicago, SEO, MC-063, Chicago, IL, 60607-7052, USA (pdelaf2$@$uic.edu, jliang$@$uic.edu).}
\and Jie Liang${}^\ddag$\and Chun Liu\thanks{Department of Applied Mathematics, Illinois Institute of Technology, Chicago, IL 60616, USA (cliu124$@$iit.edu).} \and  Jie Shen${}^\dag$.}

\usepackage{amsopn}

\begin{document}
\bibliographystyle{plain}
\graphicspath{ {Figures/} }
\maketitle

% REQUIRED
\begin{abstract}
 We develop a  special phase field/diffusive interface method to  model the nuclear architecture reorganization process. In particular, we use  a
 Lagrange multiplier approach in the phase field model  to preserve  the specific physical and geometrical constraints  for the biological events.  We develop several
 efficient and robust  linear and weakly nonlinear schemes  for this  new  model. 
 To validate the model and numerical methods, we present ample numerical simulations which in particular  reproduce several processes of nuclear architecture reorganization from the experiment literature.
\end{abstract}

\section{Introduction}
Genetic information is transmitted through genome and exhibits a hierarchical structure in the nucleus. The cell nucleus and the genome are  organized into spatially separated sub-compartments \cite{erdel2018formation}. 
 For eukaryotes, genetic information is stored in  DNA which  is also integrated into  chromosomes of a cell nucleus \cite{erdel2018formation, lee2017new,solovei2013lbr,cremer2006chromosome}. 
 In chromosome territory, chromatin fibers are formed when DNA molecule is wrapped around the histones, and form transcriptionally inactivated and condensed region form as heterochromatin and transcriptionally activated region form euchromatin.  
 The position, structure and configuration of heterochromatin and euchromatin regions are  closely related to the gene expression.
 %%%%%%%%%%%%%
For instance, it has been observed that the volume of cell nucleus is a main determinant of the overall landscape of chromatin folding~\cite{gursoy2014spatial,gursoy2017spatial,gursoy2017computational}.
 %%%%%%%%%%%% 
Different kinds of nuclear architecture can  attribute to different distribution and configuration of heterochromatin and euchromatin regions \cite{strom2017phase,cremer2001chromosome},
%%%%
with additional protein-mediated specific interactions among genomic elements~\cite{perez2020chromatix,sun2021high}.
%%%
Moreover, the  inverted  architecture occurs  where heterochromatin is located at the center of the nucleus and euchromatin is enriched at the periphery. 
In contrast,  when the heterochromatin is enriched at the nuclear periphery and around nucleoli, this is referred as the  conventional architecture.

In   \cite{solovei2009nuclear,solovei2013lbr},  Solovei  et al.  demonstrated that different types of nuclear architecture were 
associated with different mammalian lifestyles, such as  diurnal versus nocturnal. The inverted nuclear  architecture can be transformed from the convention architecture in 
mouse retina rod cells \cite{solovei2009nuclear,solovei2013lbr}. The reorganization process is accompanied by the relocation of chromosomes from positions  enriched at nuclear periphery, 
and the recreation of a single heterochromatin cluster into the inverted architecture. The difference of nuclear structure is partially attributed to the activity of lamin B,  lamin A and envelope proteins \cite{paulsen2017chrom3d,solovei2009nuclear,solovei2013lbr,kinney2018chromosome,van2017lamina}.  Moreover, the rate of conversion of heterochromatin to euchromatin can  also be controlled by volume constraints of the nucleus.

Recently several mathematical models using the phase field approach  \cite{laghmach2020mesoscale,lee2017new,seirin2017role} have been introduced to study the 
mechanism of generation  of different nuclear architectures, including the size  and shape of the nucleus,  the rate of conversion of heterochromatin to euchromatin.  
These models usually include the minimization of total energy with various relevant geometric constraints.  
A common method to preserve such constraints is through   extra penalty terms introduced in the energy functionals
in models \cite{laghmach2020mesoscale,lee2017new}.  The drawback of such methods is the presence of   the large penalty parameters which  results in a stiff nuclear 
architecture reorganization systems, leading to significant challenge in the simulation and analysis. This is especially true in situations when the volume of chromosome must be enforced  
during the reduction of nuclear size and reorganization  of nuclear architectures.

The Lagrange multiplier approach is commonly used  for constrained   gradient dynamic systems \cite{du2006simulating,du2009energetic,du2008numerical,MR4049377,
yang2017efficient,yang2021numerical,wang2016efficient,cheng2018multiple}. 
In this paper, we introduce a new Lagrange multiplier approach \cite{cheng2020new,cheng2019global,sun2020structure} to enforce the geometric constraints such as the  volume constraints for both chromosome and heterochromatin.  
When the volumes of  each chromosome and heterochromatin are preserved as constants,   the  reaction-diffusion system with the Lagrange multipliers  leads to a constrained 
 gradient flow dynamics  which  satisfies an energy dissipation law.  
 We  develop several numerical schemes for nuclear architecture system with Lagrange multipliers. 
One is  a weakly nonlinear scheme which  preserves the volume constraints but requires solving a set of $2\times 2$ nonlinear algebraic systems for the Lagrange multipliers, the second is a purely linear scheme 
which  approximates the volume constraints to second order and only requires solving linear systems with constant coefficient.  These two schemes, while being numerically efficient, do not satisfy a discrete energy law. Hence, we construct the third scheme 
which  is also weakly nonlinear but is unconditionally energy stable. However, this scheme requires solving 
a nonlinear algebraic system of $2N+1$ (where $N$ is the number of chromosomes in the nucleus) equations which may require smaller time steps to be well posed. One can choose to use one of these schemes in different scenarios. In our numerical simulations, we use the first scheme to control the volumes to targeted values exactly, then we switch to the second scheme which is more efficient. We can use the third scheme if we want to make sure that the scheme is energy dissipative.

For validation purpose, we present several simulations results which are  consistent with those observed in the  experiments.
We also demonstrate that our model and schemes are efficient and robust for investigating various nuclear architecture reorganization processes.   

The paper is organized as follows: in Section 2, we present our phase field  model for nuclear architecture reorganization by using a Lagrange multiplier approach. 
We chose suitable energy functionals to capture the most important interactions and constraints of various biological elements, and introduce Lagrange multipliers to  
 capture the specific geometric constraints for the biological process.
In Section 3, we develop  efficient linear and weakly nonlinear time discretization schemes for the phase field model developed earlier. We present  numerical results using the proposed schemes  in Section 4, and
compare them  with the existing experimental literature and previous works. Finally we conclude the paper
with more discussions of our methods and results in Section 5.

\section{A phase field model for nuclear architecture reorganization (NAR)}
\label{sec:main}
To study the nuclear architecture reorganization (NAR) process, we employ a phase field/diffusive interface method. To start with this approach,
%  which is defined by variables  $\phi_0$, $\phi_m$, $m=1,2,\cdots,N$ and $\psi$ in phase field approach.  
the total nucleus is defined by a phase (labeling) function $\phi_0$ such that $\phi_0=0$ and $\phi_0=1$ respectively for the interior and exterior regions of the nucleus (Fig.\;\ref{ref}).  For a
ellipsoid shape domain, we can define
\begin{equation}
\phi_0(x,y,t)=\frac 12(1-\tanh(\frac{\sqrt{\frac{x^2}{r^2_x(t)}+\frac{y^2}{r^2_y(t)}}-1}{\sqrt{2}\eps})),\label{nuclear:3}
\end{equation}
where $\eps$ is the interfacial (transitional domain)  width,  ${r}_x(t)$ and ${r}_y(t)$ describe  the ellipse shape of the nucleus.
Similarly,  we will introduce phase functions    $\psi$  to describe the  heterochromatin region and $\bm \phi=(\phi_1,\phi_2,\cdots,\phi_N)$ to describe each individual chromosome region, 
where $N$ represents the total number of chromosomes in the nucleus ( Fig.\;\ref{ref}).  
In particular, we will choose   $N=8$ chromosomes for drosophila and $N=46$ chromosomes for human.  In addition to the heterochromatin region, the rest of chromosome region is the euchromatin region.
These are the order parameter/phase field functions that will be used to  determined the final nuclear architecture.

For the general phase field approaches, the configuration and distribution of various regions are the consequence of minimizing a specific energy functional
in terms of the above phase field functions, which takes into all considerations of the coupling and competition between different domains, as well as the relevant geometry constraints.

In the  NAR models \cite{lee2017new,laghmach2020mesoscale},  the free energy  for chromosome and  heterochromatin  had been chosen to include							
\begin{equation}\label{orienergy}
 E_0(\bm \phi,\psi)=\sum\limits_{m=1}^N\int_\Omega \frac{\eps_{\phi}^2}{2}|\Grad\phi_m|^2 +g(\phi_m) d\bx+\int_{\Omega} \frac{\eps_{\psi}^2}{2}|\Grad\psi|^2 +g(\psi)d\bx,
\end{equation}
where   $\eps_{\phi}$
and $\eps_{\psi}$ measure the interfacial thickness corresponding to $\phi$ and $\psi$, $\Omega$ is the computational domain, and  $g(\phi)=\frac 14\phi^2(1-\phi)^2$ is the   double well potential which possess the local minima at $\phi=0$ and $\phi=1$ (see \cite{provatas2011phase}).
This part of free energy represents the competition and coupling between various chromosome and heterochromatin regions.

Next we will consider the following geometric constraints for all these regions that are biologically relevant to our application  \cite{alberts2015essential,cremer2010chromosome}:
\begin{enumerate}
\item All chromosomes are restricted  within the cell nucleus region;
\item Heterochromosome of each chromosome stays within the chromosome;
\item Between all chromosomes, due to the excluded volume effects, do not self-cross or cross each other.
\end{enumerate}
These  constraints could be incorporated into the model by introducing  three extra terms in the free energy:
\begin{equation}
\begin{split}
 E_1(\bm \phi,\psi)&=\underbrace{\beta_0\sum\limits_{m=1}^N\int_\Omega  h(\phi_0)h(\phi_m) d\bx}_1+\underbrace{\beta_{\psi}\int_{\Omega}  [1-\sum\limits_{m=1}^Nh(\phi_m)]h(\psi)  d\bx}_2
   \\&+ \underbrace{\beta_{\phi}\sum\limits_{m\neq n}\int_\Omega  h(\phi_n)h(\phi_m) d\bx}_3,
\end{split}
\end{equation}
where $\beta_0,\beta_{\psi}$ and $\beta_{\phi}$ are three positive constants which  indicate the intensities of domain territories, and  $h(\phi)$ is used for the induction of driving  interface between $\phi=0$ and $\phi=1$ while keeping the local minima $0$ and $1$ fixed during dynamic process. The required conditions for $h(\phi)$ are
\begin{equation}\label{six}
h(0)=0,\;h(1)=1,\;h'(0)=h'(1)=h''(0)=h''(1)=0.
\end{equation}
The  lowest degree polynomials satisfying the above conditions is  $h(\phi)=\phi^3(10-15\phi+6\phi^2)$, see Fig.\;\ref{hphi}.  
\begin{figure}[htbp]
\centering
\includegraphics[width=0.5\textwidth,clip==]{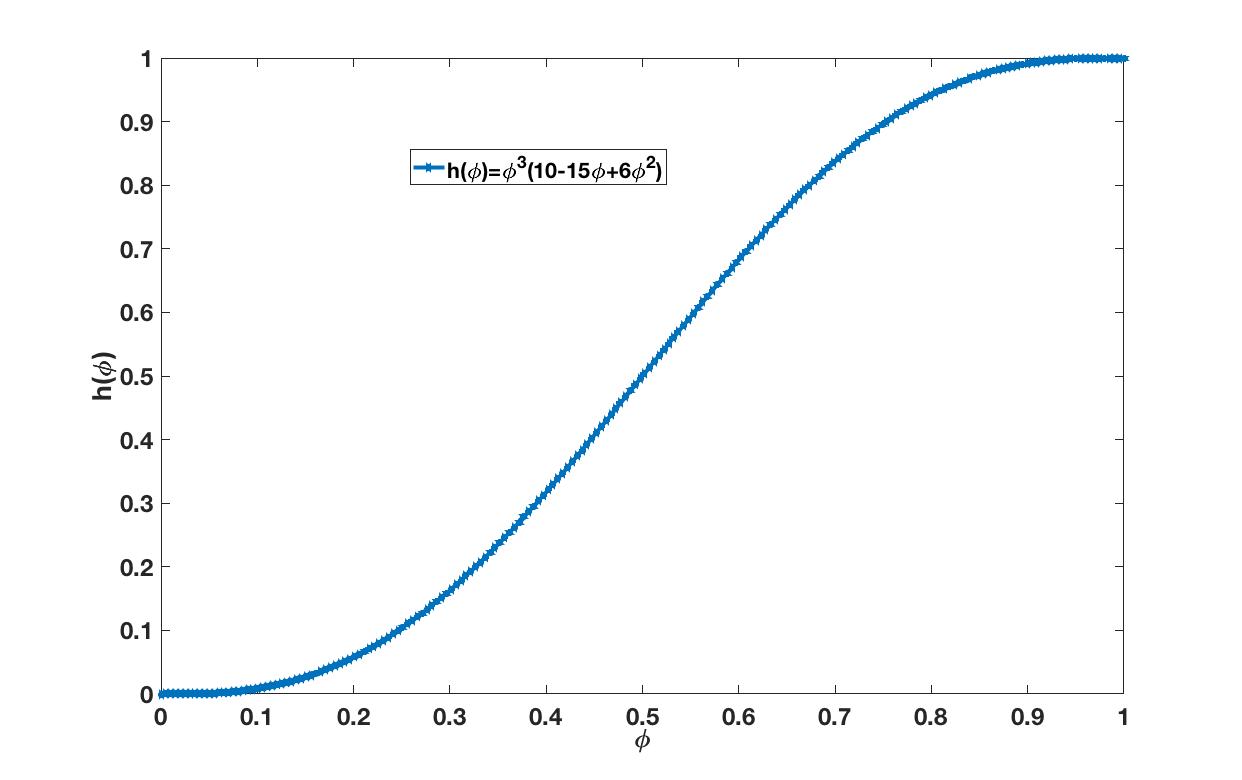}
\caption{The function $h(\phi)=\phi^3(10-15\phi+6\phi^2)$.}\label{hphi}
\end{figure}

Due to the expression of  LBR and lamin A/C in the nuclear envelope,  interactions between heterochromatin and the nuclear envelope play an important role in nuclear architecture reorganization process.    
For this purpose,  another term was introduced in the  free energy  to describe the  interactions  \cite{solovei2013lbr}:
\begin{equation}
E_2(\phi_0,\psi)=\gamma\int_{\Omega} \Grad h(\phi_0)\cdot \Grad h(\psi)d\bx,
\end{equation}
 where $\gamma$ is the affinity constant, and  $\gamma >0$ implies  the heterochromatin will locate at the nuclear periphery, while $\gamma =0$ leads to the lack of heterochromatin and nuclear envelope interactions due to the absence of LBR or lamin A/C.  More precisely,  $E_2$  represents the intensity of affinity between nuclear function $\phi_0$ and heterochromatin region function $\psi$.

\begin{figure}[htbp]
\centering
\includegraphics[width=0.8\textwidth,clip==]{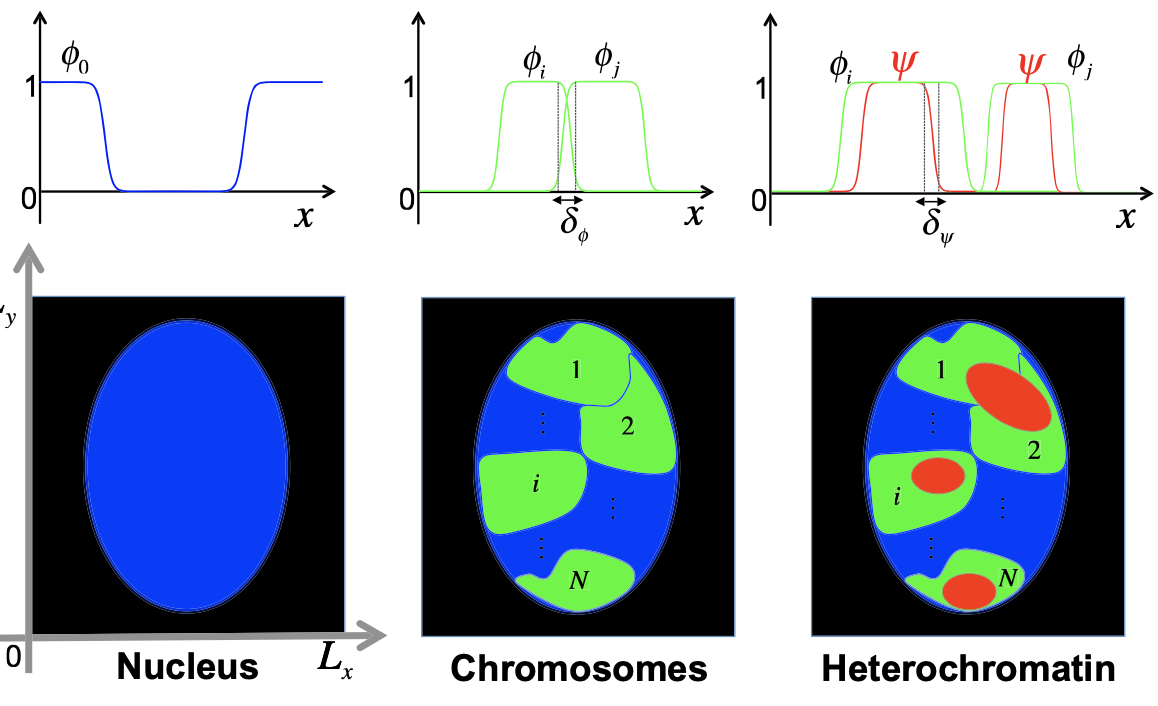}
\caption{Phase functions $\phi_m$ and $\psi$ on 1D and domain diagram on 2D. Blue region is the nuclear domain formulated by  $\phi_0$. Green regions are chromosome territories for $m=1,2,\cdots,N$ defined by $\phi_m$. Red region is the heterochromatin domain defined by $\psi$.  }\label{ref}
\end{figure}

\subsection{NAR model with Lagrange multipliers}
In the general nuclear reorganization process, often one need to take into account more geometric constraints.
In particular,  we shall include the following constraints in our model:
\begin{enumerate}
\setcounter{enumi}{3}
\item   The nuclear space is fully occupied by chromosomes;
%\item  5. The chromosome can be contracted or expanded to a given volume;
\item   Heterochromatin is converted from/to euchromatin within each chromosome.
\end{enumerate}
Notice that our approach could be extended to more general situations, especially related to those of item 4.

Our approach here is to introduce  Lagrange multipliers to guarantee the  constraints $4,5$ in the following NAR model with the total free energy 
\begin{equation}\label{totalenergy}
E(\phi_0,\bm \phi,\psi)=E_0+E_1+E_2.
\end{equation}
%One can use the gradient decent method to obtain the equilibrium of the above functional. 
 The corresponding Allen-Cahn type gradient flow \cite{liu2003phase,cheng2020new,feng2003numerical,chen1998applications,
 guan2014second,guan2017convergence,yang2017efficient}  with respect to the above energy and the constraints  4 and 5   take the form:
\begin{eqnarray}
&&\frac{\partial \phi_m}{\partial t}=-M\Big\{\frac{\delta E}{\delta \phi_m} -\lambda_m h'(\phi_m)-\eta_m h'(\phi_m)h(\psi)\Big\}, \quad m=1,2,\cdots,N,\label{nuclear:1}\\
&&\frac{\partial \psi}{\partial t}=-M\Big\{\frac{\delta E}{\delta \psi}-\sum\limits_{m=1}^N\eta_mh(\phi_m)h'(\psi)\Big\},\label{nuclear:2}\\
&&V_m(t)=\int_{\Omega}h(\phi_m)d\bx,\qquad v_m(t)=\int_{\Omega}h(\phi_m)h(\psi)d\bx,  \quad m=1,2,\cdots,N, \label{nuclear:4}\\
&&\sum\limits_{m=1}^NV_m(t)=\int_{\Omega}h(\phi_0) d\bx,\label{nuclear:5}
\end{eqnarray}
where $M$ is mobility constant, $V_m(t)$ and $v_m(t)$  $(m=1,2,\cdots,N)$  represent, respectively.  From \eqref{nuclear:4},  the  volume of chromosome can be contracted or expanded to a given volume by using our model. The volumes of  $m$-th chromosome and heterochromatin in the $m$-th chromosome at time $t$ during  nuclear reorganization (growth or inversion) stage are enforced by the Lagrange multipliers  $\lambda_m(t)$, $\eta_m(t)$ $(m=1,2,\cdots,N)$  \eqref{nuclear:4}.
The boundary conditions can be  one of the following two types
\begin{eqnarray}\label{bc:1}
&&(i)\mbox{ periodic; or } (ii)\,\,\partial_{\bf n} \phi_m|_{\partial\Omega}=\partial_{\bf n} \psi|_{\partial\Omega}=0,
\end{eqnarray}
where $\bf n$ is the unit outward normal on the boundary $\partial\Omega$.

Let $\bar{V}_m$ and $\bar{v}_m$ be, respectively, the target volumes for each chromosome  and  heterochromatin in each chromosome,  $\rho_m(t)= v_m(t)/V_m(t)$ can be interpreted as the heterochromatin  conversion rate during nuclear architecture reorganization process. We assume that they will reach the target values at time $t=t_0$ and then stay there according to:
 \begin{equation}\label{rate2}
 \begin{split}
  &V_{m}(t) =\begin{cases}V_m(0) +\frac{(\bar{V}_m+\delta_{1m})t}{t + \alpha_1 e^{-\alpha_2 t}},& 0\le t\le t_0\\
  \bar{V}_m,& t\ge t_0\end{cases};\\
&   v_{m}(t) =\begin{cases} v_m(0) +\frac{(\bar{v}_m+\delta_{2m})t}{t + \alpha_1 e^{-\alpha_2 t}},& 0\le t\le t_0\\
  \bar{v}_m,& t\ge t_0\end{cases},
  \end{split}
  \end{equation}
where $\delta_{1m}$ and  $\delta_{2m}$ are determined by $V_{m}(t_0)=\bar V_m$ and  $v_{m}(t_0)=\bar v_m$.
%$$V_m(0) +\frac{(\bar{V}_m+\delta_1)t_0}{t_0 + \alpha_1 e^{-\alpha_2 t_0}}=\bar V_m,\quad 
%v_m(0) +\frac{(\bar{v}_m+\delta_2)t_0}{t_0 + \alpha_1 e^{-\alpha_2 t_0}}=\bar v_m,$$
 and $\alpha_i$  $(i=1,2)$ are suitable positive constants related to the time scale. 
 Similarly, we assume that $r_x(t)$ and $r_y(t)$  evolve according to
  \begin{equation}\label{rate3}
 \begin{split}
  &r_{x}(t) =\begin{cases}r_x(0) +\frac{(\bar{r}_x+\delta_3)t}{t + \alpha_1 e^{-\alpha_2 t}},& 0\le t\le t_0\\
  \bar{V}_m,& t\ge t_0\end{cases};\\
&   r_{y}(t) =\begin{cases} r_y(0) +\frac{(\bar{r}_y+\delta_4)t}{t + \alpha_1 e^{-\alpha_2 t}},& 0\le t\le t_0\\
  \bar{v}_m,& t\ge t_0\end{cases},
  \end{split}
  \end{equation}
where $\delta_3$ and  $\delta_4$ are determined by $r_{x}(t_0)=\bar r_x$ and  $r_{y}(t_0)=\bar r_y$
%$$r_x(0) +\frac{(\bar{r}_x+\delta_3)t_0}{t_0 + \alpha_1 e^{-\alpha_2 t_0}}=\bar r_x,\quad 
%r_y(0) +\frac{(\bar{r}_y+\delta_4)t_0}{t_0 + \alpha_1 e^{-\alpha_2 t_0}}=\bar r_y,$$
%  To control  decreasing rate of  nucleus shape in time,   $x$-radius and $y$-radius of nuclear shape are  represented  as  two  sigmod function s which are  formulated as 
with $\bar r_x$ and $\bar r_y$ being the targeted axis lengths for the  ellipse enclosing the nucleus.

\begin{remark}
In \cite{lee2017new}, a penalty approach is introduced to satisfy the  physical  constraints 4 and 5 by adding the following to the free energy:
\begin{equation}
\begin{split}
E_2(\bm \phi,\psi)&=\underbrace{\alpha_0[\int_\Omega [ 1-h(\phi_0)] dx-\sum\limits_{m=1}^NV_m(t) ]^2}_4 + 
\underbrace{\alpha_V\sum\limits_{m=1}^N[V_m(t)-\bar{V}_m(t)]^2}_5\\&+
\underbrace{\alpha_v\sum\limits_{m=1}^N[v_m(t)-\bar{v}_m(t)]^2}_{6},
\end{split}
\end{equation}
where   $\alpha_0,\alpha_V,\alpha_v$ are three  positive penalty parameters.  The volume of $m$-th chromosome $V_m(t)$ and volume of heterochromatin in $m$-th chromosome are defined in \eqref{nuclear:4}.
A disadvantage of the penalty approach  is large penalty parameters   are needed for  accurate approximation of the physical constraints, and may lead  to very stiff systems that are difficult to solve numerically.  The  Lagrange multiplier approach that we use here can  enforce these non-local constraints exactly and is free of penalty parameters. Furthermore, the  NAR model \eqref{nuclear:1}-\eqref{nuclear:5}  based on the Lagrange multiplier approach can control  exactly the growth rate of  volume for different compartments  during the nuclear reorganization (growth or inversion).
\end{remark} 
 
\begin{remark}
In the phase field approaches, there are many ways to represent the volume of each individual domain.
For instance, the volume of each chromosome denoted by  $V_m(t)$ and their corresponding heterochromatin 
domain volume $v_m(t)$ could be  computed  by the integrals  $\int_{\Omega}\phi_m d\bx$ and $\int_{\Omega}\psi d\bx$. 
However this representation may have disadvantages in the minimizing procedure, especially for the penalty methods used in \cite{laghmach2020mesoscale,lee2017new,seirin2017role}, due to
its linearity with respect to the phase functions.
One way to overcome this is to use the polynomial function $h(\phi)=\phi^3(10-15\phi+6\phi^2)$ (see Fig.\;\ref{hphi})  for the computation
of  the volumes for different chromosome regions.
Since $h(\phi)$ is an increasing function with respect to $\phi$ in the interval $[0,1]$ with  $h(0)=0$ and $h(1)=1$. 
One can adapt $V_m=\int_{\Omega}h(\phi)d\bx$ and $\int_{\Omega}h(\phi_m)h(\psi) d\bx$ for the volume of $m$-th chromosome and heterochromatin in $m$-th chromosome.
\end{remark}

%\cite{laghmach2020mesoscale,lee2017new}. 

Let $(\cdot,\cdot)$ be the inner product in $L^2(\Omega)$, and $\|\cdot\|$ be the associated norm in $L^2(\Omega)$. 
 The constrained NAR model  \eqref{nuclear:1}-\eqref{nuclear:5} with \eqref{rate2}  can be interpreted as   a $L^2$ gradient system  which implies phase separations will happen  for $t\ge t_0$. 

\begin{theorem}
The constrained NAR  model  \eqref{nuclear:1}-\eqref{nuclear:5}  with \eqref{rate2} satisfies  the following energy dissipation law 
\begin{equation}\label{law:1}
 \frac{d }{d t} E(\phi_0,\bm \phi,\psi)=-\frac{1}{M}(\sum\limits_{m=1}^N\|\partial_t\phi_m\|^2+\|\partial_t \psi\|^2), \quad\forall t\ge t_0.
\end{equation}
\end{theorem}
\begin{proof}
 
Taking the inner product of  \eqref{nuclear:1} with $\partial_t \phi_m$, we obtain
\begin{equation}\label{law:eq1}
\begin{split}
-\frac{1}{M}\|\frac{\partial \phi_m}{\partial t}\|^2&=(\frac{\delta E}{\delta \phi_m},\partial_t\phi_m)-\lambda_m( h'(\phi_m),\partial_t\phi_m)\\&-\eta_m (h'(\phi_m)h(\psi),\partial_t\phi_m).
\end{split}
\end{equation}
Taking the inner product of  \eqref{nuclear:2} with $\partial_t \psi$, we obtain
\begin{equation}\label{law:eq2}
\begin{split}
-\frac{1}{M}\|\frac{\partial \psi}{\partial t}\|^2=(\frac{\delta E}{\delta \psi},\partial_t\psi)-\sum\limits_{m=1}^N\eta_m(h(\phi_m)h'(\psi),\partial_t\psi).
\end{split}
\end{equation}
We derive  from \eqref{nuclear:4} that
\begin{equation}\label{law:eq3}
(h'(\phi_m), \partial_t\phi_m)=\frac{d}{dt}\int_{\Omega} h(\phi_m)d\bx=\frac{d}{dt}\bar{V}_m=0.
\end{equation}
Similarly, we obtain
\begin{equation}\label{law:eq4}
(h'(\phi_m)h(\psi),\partial_t\phi_m)+(h(\phi_m)h'(\psi),\partial_t\psi)
=\frac{d}{dt}\int_{\Omega}h(\phi_m)h(\psi)d\bx=\frac{d}{dt}\bar{v}_m=0.
\end{equation}
Summing up  \eqref{law:eq1} for $m=1,2,\cdots,N$,  
combing it with \eqref{law:eq2}-\eqref{law:eq4} and using equality
\begin{equation}
\sum\limits_{m=1}^N(\frac{\delta E}{\delta \phi_m},\partial_t\phi_m) + (\frac{\delta E}{\delta \psi},\partial_t\psi)=\frac{d }{d t} E(\phi_0,\bm\phi,\psi),
\end{equation} 
we obtain the desired  energy dissipation law \eqref{law:1}.
\end{proof}
\begin{comment}
In order to facilitate our next section construction algorithm, we explicitly  calculate the variational derivative of the total free energy $E$  for each term,  NAR model \eqref{nuclear:1}-\eqref{nuclear:5}  is rewritten into
\begin{eqnarray}
&&\hskip 2cm\frac{\partial \phi_m}{\partial t}=-M\mu_{\phi_m},\label{nu:eq:1}\\
&&\hskip 2cm\mu_{\phi_m}=-\eps^2_\phi\Delta \phi_m+ g'(\phi_m) +\Big\{\beta_0h(\phi_0)-\beta_{\psi}h(\psi)\nonumber\\&&\hskip 2cm+\beta_\phi\{\sum\limits_{m=1}^Nh(\phi_m)-h(\phi_m)\}\Big\}h'(\phi_m)
 -\lambda_m h'(\phi_m)-\eta_m h'(\phi_m)h(\psi), \quad m=1,2,\cdots,N\label{nu:eq:2}\\
&&\hskip 2cm\frac{\partial \psi}{\partial t}=-M\nu,\label{nu:eq:3}\\
&&\hskip 2cm\nu=-\eps^2_\psi\Delta \psi + g'(\psi) +\beta_\psi\Big\{1-\sum\limits_{m=1}^Nh(\phi_m)\Big\}h'(\psi)\nonumber\\&&\hskip 2cm-\sum\limits_{m=1}^N\eta_mh(\phi_m)h'(\psi) -\gamma\Delta\phi_0h'(\psi),\label{nu:eq:4}\\
&&\hskip 2cm V_m(t)=\int_{\Omega}h(\phi_m)d\bx,\quad v_m(t)=\int_{\Omega}h(\phi_m)h(\psi)d\bx, \quad m=1,2,\cdots,N\label{nu:eq:5}\\
&& \hskip 2cm \sum\limits_{m=1}^NV_m(t)=\int_{\Omega}h(\phi_0) d\bx, \label{nu:eq:6}
\end{eqnarray}
with boundary conditions  \eqref{bc:1}.  For system \eqref{nu:eq:1}-\eqref{nu:eq:6},  numerical  simulations are performed on a square domain with an elliptic nucleus in the center of computational domain in which the chromosomes and heterochromatin domains are manually  generated initially  by $\tanh$-like functions  in Fig.\,\ref{ini}.
\end{comment}

\section{Numerical Schemes}
In this section,  we construct several efficient time discretization schemes based on the  Lagrange multiplier approach \cite{cheng2020new,cheng2019global} for the phase field NAR model \eqref{nuclear:1}-\eqref{nuclear:5}.  For the sake of simplicity,   for any function $f$,   we denote  $f^{n,\dagger}=2f^n-f^{n-1}$, $f^{n,\star}=\frac 32f^n-\frac 12f^{n-1}$and $f^{n+\frac 12}=\frac{f^{n+1}+f^n}{2}$.

We split the total energy into a quadratic part and the remaining part  as follows:
$$E(\phi_0,\bm\phi,\psi)=\big(\sum\limits_{m=1}^N\int_\Omega \frac{\eps_{\phi}^2}{2}|\Grad\phi_m|^2  d\bx+\int_{\Omega} \frac{\eps_{\psi}^2}{2}|\Grad\psi|^2 d\bx\big) +\tilde{E}(\phi_0,\bm\phi,\psi),$$ where $\tilde{E}$ is 
\begin{equation}
\begin{split}
\tilde{E}(\phi_0,\bm\phi,\psi)=&\beta_0\sum\limits_{m=1}^N\int_\Omega  h(\phi_0)h(\phi_m) d\bx+\beta_{\psi}\int_{\Omega}  [1-\sum\limits_{m=1}^Nh(\phi_m)]h(\psi)  d\bx
   \\&+ \beta_{\phi}\sum\limits_{m\neq n}\int_\Omega  h(\phi_n)h(\phi_m) d\bx + 
   \sum\limits_{m=1}^N\int_\Omega g(\phi_m)d\bx + \int_{\Omega} g(\psi)d\bx\\&
   +  \gamma\int_{\Omega} \Grad h(\phi_0)\cdot \Grad h(\psi)d\bx.
\end{split}
\end{equation}
Once the volume of  nucleus $\int_{\Omega}h(\phi_0) d\bx$  is given,  volumes of each chromosome territory can be set up  accordingly so that the constraint \eqref{nuclear:5} can be satisfied automatically. 
\subsection{A weakly nonlinear volume preserving scheme}
\begin{comment}

 We consider the original NAR model \eqref{nu:eq:1}-\eqref{nu:eq:6} 

\begin{eqnarray}
&&\hskip 2cm\frac{\partial \phi_m}{\partial t}=-M\mu_{m},\label{non:new:eq:1}\\
&&\hskip 2cm\mu_{m}=-\eps^2_\phi\Delta \phi_m+\frac{\delta \tilde{E}}{\delta \phi_m}-\lambda_m h'(\phi_m)-\eta_m h'(\phi_m)h(\psi),\label{non:new:eq:2}\\
&&\hskip 2cm\frac{\partial \psi}{\partial t}=-M\nu,\label{non:new:eq:3}\\
&&\hskip 2cm\nu=-\eps^2_\psi\Delta \psi +\frac{\delta \tilde{E}}{\delta \psi}-\sum\limits_{m=1}^N\eta_mh(\phi_m)h'(\psi),\label{non:new:eq:4}\\
&&\hskip 2cm V_m(t)=\int_{\Omega}h(\phi_m)d\bx,\quad v_m(t)=\int_{\Omega}h(\phi_m)h(\psi)d\bx,\quad m=1,2,\cdots,N.\label{non:new:eq:5}
\end{eqnarray}
\end{comment}

Note that in the first stage, we need to increase the volumes $V_m(t)$ and $v_m(t)$ to the targeted values $\bar V_m$ and $\bar v_m$ according to \eqref{rate2}, respectively. Hence, we shall first construct below a volume preserving scheme which allows us to achieve this goal. More precisely,  
 a second order scheme based on the Lagrange multiplier approach is as follows: 
 \begin{align}
\frac{ \phi^{n+1}_m-\phi_m^n}{\delta t}&=-M(-\eps^2_\phi\Delta \phi^{n+\frac 12}_m\nonumber\\
&+(\frac{\delta \tilde{E}}{\delta \phi_m})^{n,\star}
 -\lambda^{n+\frac 12}_m h'(\phi^{n,\star}_m)-\eta^{n+\frac 12}_m h'(\phi^{n,\star}_m)h(\psi^{n,\star})),\;m=1,\cdots,N,\label{scheme:1}\\
\frac{ \psi^{n+1}-\psi^n}{\delta  t}&=-M(-\eps^2_\psi\Delta \psi^{n+\frac 12} +(\frac{\delta \tilde{E}}{\delta \psi})^{n,\star} -\sum\limits_{m=1}^N\eta_m^{n+\frac 12}h(\phi^{n,\star}_m)h'(\psi^{n,\star})),\label{scheme:3}\\
V_m(t^{n+1})&=\int_{\Omega}h(\phi^{n+1}_m)d\bx,\;m=1,\cdots,N,\label{scheme:4}\\
 v_m(t^{n+1})&=\int_{\Omega}h(\phi^{n+1}_m)h(\psi^{n+1})d\bx,\;m=1,\cdots,N.\label{scheme:5}
\end{align}
Below we show how to efficiently  solve the coupled scheme  \eqref{scheme:1}-\eqref{scheme:5}. 
Writing 
\begin{equation}\label{sol:update:1}
\begin{split}
&\phi_m^{n+1} = \phi_{1,m}^{n+1} + \lambda_m^{n+\frac 12}\phi_{2,m}^{n+1} +\eta_m^{n+\frac 12}\phi_{3,m}^{n+1},
%&\mu_{m}^{n+1}= \mu_{1,m}^{n+1} + \lambda_m^{n+\frac 12}\mu_{2,m}^{n+1} +\eta_m^{n+\frac 12}\mu_{3,m}^{n+1},
\end{split}
\end{equation}
in  \eqref{scheme:1}, collecting all terms without  $(\lambda_m^{n+1},\eta_m^{n+1})$, with $\lambda_m^{n+1}$ or with $\eta_m^{n+1}$, we find that, for $m=1,2,\cdots,N$, 
 $(\phi_{i,m}^{n+1}, \,i=1,2,3)$ can be determined from the following decoupled linear systems:
\begin{equation}\label{part:1}
\begin{split}
&\frac{ \phi^{n+1}_{1,m}-\phi_m^n}{\delta t}=-M(-\eps^2_\phi\Delta \phi^{n+\frac 12}_{1,m}+(\frac{\delta \tilde{E}}{\delta \phi_m})^{n,\star});
\end{split}
\end{equation}
\begin{equation}\label{part:1:1}
\begin{split}
&\frac{ \phi^{n+1}_{2,m}}{\delta t}=-M(-\eps^2_\phi\Delta \phi^{n+\frac 12}_{2,m}-
h'(\phi^{n,\star}_m));
\end{split}
\end{equation}
\begin{equation}\label{part:1:2}
\begin{split}
&\frac{ \phi^{n+1}_{3,m}}{\delta t}=-M(-\eps^2_\phi\Delta \phi^{n+\frac 12}_{3,m}-
h'(\phi^{n,\star}_m)h(\psi^{n,\star})).
\end{split}
\end{equation}
Then, writing
\begin{equation}\label{sol:update:2}
\begin{split}
&\psi^{n+1}=\psi_1^{n+1} +\sum\limits_{m=1}^N\eta_m^{n+\frac 12}\psi_{2,m}^{n+1},
%&\nu^{n+1} =\nu_{1}^{n+1} +\sum\limits_{m=1}^N\eta_m^{n+\frac 12}\nu_{2,m}^{n+1},
\end{split}
\end{equation}
in \eqref{scheme:3}, we find that $\psi_1^{n+1}$ and
$ (\psi_{2,m}^{n+1}, \,m=1,2,\cdots,N)$ can be determined from the following decoupled linear systems:
\begin{equation}\label{part:2}
\begin{split}
&\frac{ \psi^{n+1}_1-\psi^n}{\delta  t}=-M(-\eps^2_\psi\Delta \psi_1^{n+\frac 12} +(\frac{\delta \tilde{E}}{\delta \psi})^{n,\star});
\end{split}
\end{equation}
\begin{equation}\label{part:2:2}
\begin{split}
&\frac{ \psi^{n+1}_{2,m}}{\delta  t}=-M(-\eps^2_\psi\Delta \psi_{2,m}^{n+\frac 12} +
h(\phi^{n,\star}_m)h'(\psi^{n,\star})).
\end{split}
\end{equation}
We observe that the above systems are all linear  Poisson-type equation with constant coefficients so they can be efficiently solved. 

Once we have obtained $(\phi_{i,m}^{n+1}, \,i=1,2,3)$ and $ \psi_{i,m}^{n+1}, \,i=1,2)$,  we  plug  \eqref{sol:update:1}-\eqref{sol:update:2} into \eqref{scheme:4}-\eqref{scheme:5} to obtain a $2\times 2$ nonlinear algebraic system for $(\lambda_m^{n+1},\eta_m^{n+1})$. For $\delta t$ sufficiently small, this nonlinear algebraic system admits real solutions that  be solved with an iterative method at negligible cost.

In summary,  the scheme \eqref{scheme:1}-\eqref{scheme:5} can be efficiently implemented as follows.

\begin{itemize}
\item Solve $\psi_1^{n+1}$ from  \eqref{part:2}.
 \item For $m=1,\cdots, N$:
\begin{itemize}
\item   solve $(\phi_{i,m}^{n+1}, \,i=1,2,3)$  from \eqref{part:1}-\eqref{part:1:2} and $\psi_{2,m}^{n+1}$   from \eqref{part:2:2};
  \item  determine the Lagrange multipliers  $(\lambda_m^{n+1},\eta_m^{n+1})$ from the coupled  nonlinear algebraic system \eqref{scheme:4}-\eqref{scheme:5};
\item  update $\phi_m^{n+1}$  using \eqref{sol:update:1}.
\end{itemize}
\item Update $\psi^{n+1}$ using  \eqref{sol:update:2}.
\end{itemize}

\subsection{A linear scheme}
 In practice,  the scheme \eqref{scheme:1}-\eqref{scheme:5} should be used if we want to exactly preserve the volume dynamics of  chromosome and heterochromatin. A  disadvantage  of the scheme \eqref{scheme:1}-\eqref{scheme:5} is that we need to solve a nonlinear algebraic system which may require small time steps. 
 To accelerate the simulation,  
we  construct below a linear scheme for  system  \eqref{nuclear:1}-\eqref{nuclear:5} which is more efficient but only  approximately preserve the volume dynamics.  %The boundary conditions are still  \eqref{bc:1}.
% $\frac{\delta \tilde{E}}{\delta \phi_m}$ and $\frac{\delta \tilde{E}}{\delta \psi}$ can be obtained  from \eqref{nu:eq:2} and \eqref{nu:eq:4}. Given initial volume $V_m(0)=V_m^0$, $v_m(0)=v_m^0$, taking integration of equation \eqref{new:eq:5}-\eqref{new:eq:6} with respect to time, we obtain the  original NAR model \eqref{nu:eq:1}-\eqref{nu:eq:6}. 

To this end, we reformulate   \eqref{nuclear:1}-\eqref{nuclear:5} into the following equivalent system:
\begin{eqnarray}
&&\hskip 2cm\frac{\partial \phi_m}{\partial t}=-M\big(-\eps^2_\phi\Delta \phi_m+\frac{\delta \tilde{E}}{\delta \phi_m}-\lambda_m h'(\phi_m)-\eta_m h'(\phi_m)h(\psi)\big), \quad m=1,\cdots,N,\label{new:eq:1}\\
&&\hskip 2cm\frac{\partial \psi}{\partial t}=-M\big(-\eps^2_\psi\Delta \psi +\frac{\delta \tilde{E}}{\delta \psi}-\sum\limits_{m=1}^N\eta_mh(\phi_m)h'(\psi)\big),\label{new:eq:3}\\
&&\hskip 2cm V'_m(t)=\int_{\Omega}h'(\phi_m)\partial_t\phi_md\bx, \quad m=1,\cdots,N,\label{new:eq:5}\\
&&\hskip 2cm v'_m(t)=\int_{\Omega}h'(\phi_m)h(\psi)\partial_t\phi_{m}+h(\phi_m)h'(\psi)\partial_t\psi d\bx,\quad m=1,\cdots,N.\label{new:eq:6}
\end{eqnarray}
Note that the last two relations are obtained by taking the time derivative of $V_m$ and $v_m$ in  \eqref{nuclear:4}. 

A second-order linear scheme  for the new system
 \eqref{new:eq:1}-\eqref{new:eq:6} is as follows: 
  \begin{align}
\frac{ \phi^{n+1}_m-\phi_m^n}{\delta t}&=-M\big(-\eps^2_\phi\Delta \phi^{n+\frac 12}_m+(\frac{\delta \tilde{E}}{\delta \phi_m})^{n,\star}\nonumber\\
 &-\lambda^{n+\frac 12}_m h'(\phi^{n,\star}_m)-\eta^{n+\frac 12}_m h'(\phi^{n,\star}_m)h(\psi^{n,\star})\big),\; m=1,\cdots,N,\label{linear:scheme:1}\\
\frac{ \psi^{n+1}-\psi^n}{\delta  t}&=-M\big(-\eps^2_\psi\Delta \psi^{n+\frac 12} +(\frac{\delta \tilde{E}}{\delta \psi})^{n,\star} -\sum\limits_{m=1}^N\eta_m^{n+\frac 12}h(\phi^{n,\star}_m)h'(\psi^{n,\star})\big),\label{linear:scheme:3}\\
 V_m(t^{n+1})-V_m(t^n)&=\int_{\Omega}h'(\phi^{n,\star}_m)(\phi_m^{n+1}-\phi_m^n)d\bx,\; m=1,\cdots,N,\label{linear:scheme:5}\\
 v_m(t^{n+1})-v_m(t^n)&=\int_{\Omega}h'(\phi^{n,\star}_m)h(\psi^{n,\star})(\phi_{m}^{n+1}-\phi_m^n)\nonumber\\ &+h(\phi^{n,\star}_m)h'(\psi^{n,\star})(\psi^{n+1}-\psi^n)d\bx,\; m=1,\cdots,N.\label{linear:scheme:6}
\end{align}
The above coupled scheme can be solved in essentially the same fashion as the scheme \eqref{scheme:1}-\eqref{scheme:5}. In fact, setting
\begin{equation}\label{sol:update:1b}
\begin{split}
&\phi_m^{n+1} = \phi_{1,m}^{n+1} + \lambda_m^{n+\frac 12}\phi_{2,m}^{n+1} +\eta_m^{n+\frac 12}\phi_{3,m}^{n+1},
%&\mu_{m}^{n+1}= \mu_{1,m}^{n+1} + \lambda_m^{n+\frac 12}\mu_{2,m}^{n+1} +\eta_m^{n+\frac 12}\mu_{3,m}^{n+1},
\end{split}
\end{equation}
in  \eqref{linear:scheme:1}, we find that for $m=1,2,\cdots,N$, 
 $(\phi_{i,m}^{n+1}, \,i=1,2,3)$ are still determined from \eqref{part:1}- \eqref{part:1:2}.
Then, writing
\begin{equation}\label{sol:update:2b}
\begin{split}
&\psi^{n+1}=\psi_1^{n+1} +\sum\limits_{m=1}^N\eta_m^{n+\frac 12}\psi_{2,m}^{n+1},
%&\nu^{n+1} =\nu_{1}^{n+1} +\sum\limits_{m=1}^N\eta_m^{n+\frac 12}\nu_{2,m}^{n+1},
\end{split}
\end{equation}
in \eqref{linear:scheme:3}, we find that $\psi_1^{n+1}$ and
$ (\psi_{2,m}^{n+1}, \,m=1,2,\cdots,N)$ are also  determined from \eqref{part:2}- \eqref{part:2:2}. 
Once we have obtained $(\phi_{i,m}^{n+1}, \,i=1,2,3)$ and $ \psi_{i,m}^{n+1}, \,i=1,2)$,  we  plug  \eqref{sol:update:1b}-\eqref{sol:update:2b} into \eqref{linear:scheme:5}-\eqref{linear:scheme:6} to obtain a $2\times 2$ linear algebraic system for $(\lambda_m^{n+1},\eta_m^{n+1})$ that can be solved explicitly.
In summary,  the scheme \eqref{linear:scheme:1}-\eqref{linear:scheme:6} can be efficiently implemented as follows.

\begin{itemize}
\item Solve $\psi_1^{n+1}$ from  \eqref{part:2}.
 \item For $m=1,\cdots, N$:
\begin{itemize}
\item   solve $(\phi_{i,m}^{n+1}, \,i=1,2,3)$  from \eqref{part:1}-\eqref{part:1:2} and $\psi_{2,m}^{n+1}$   from \eqref{part:2:2};
  \item  determine the Lagrange multipliers  $(\lambda_m^{n+1},\eta_m^{n+1})$ from the coupled  linear algebraic system \eqref{linear:scheme:5} and \eqref{linear:scheme:6};
\item  update $\phi_m^{n+1}$  using \eqref{sol:update:1}.
\end{itemize}
\item Update $\psi^{n+1}$ using  \eqref{sol:update:2}.
\end{itemize}
Note that  the scheme \eqref{linear:scheme:1}-\eqref{linear:scheme:6}  is well posed for any time step.

\subsection{A weakly nonlinear energy stable scheme}
Note that the schemes   \eqref{scheme:1}-\eqref{scheme:5} and \eqref{linear:scheme:1}-\eqref{linear:scheme:6} are not guaranteed to be energy dissipative. 
 Below we modify the scheme  \eqref{scheme:1}-\eqref{scheme:5} slightly to construct a weakly nonlinear but  energy stable scheme with essentially  the same computational cost for $t\ge t_0$% 
  when volumes of each chromosome $V_m(t)$ and heterochromatin $v_m(t)$ become constants.
 
The idea is to introduce another Lagrange multiplier to enforce the energy dissipation. To this end,  we introduce another Lagrange multiplier $R(t)$ and  expand the system \eqref{new:eq:1}-\eqref{new:eq:6} for $t\ge t_0$ as
\begin{eqnarray}
&&\frac{\partial \phi_m}{\partial t}=-M\big(-\eps^2_\phi\Delta \phi_m+R(t)\frac{\delta \tilde{E}}{\delta \phi_m}-\lambda_m h'(\phi_m)-\eta_m h'(\phi_m)h(\psi)\big),\;m=1,\cdots,N,\label{stab:eq:1}\\
&&\frac{\partial \psi}{\partial t}=-M\big(-\eps^2_\psi\Delta \psi +R(t)\frac{\delta \tilde{E}}{\delta \psi}-\sum\limits_{m=1}^N\eta_mh(\phi_m)h'(\psi)\big),\label{stab:eq:3}\\
&&\int_{\Omega}h(\phi^0_m)d\bx=\int_{\Omega}h(\phi_m)d\bx,\;m=1,\cdots,N,\label{stab:eq:5}\\
&& \int_{\Omega}h(\phi_m^0)h(\psi^0) d\bx =\int_{\Omega}h(\phi_m)h(\psi) d\bx, \;m=1,\cdots,N,\label{stab:eq:6}\\
&& \frac{d}{dt}\tilde{E}=R(t)\sum\limits_{m=1}^N(\frac{\delta \tilde{E}}{\delta \phi_m},\partial_t\phi_m)+R(t)(\frac{\delta \tilde{E}}{\delta \psi},\partial_t\psi)\label{stab:eq:7}\\
&&\hskip 2cm +\sum\limits_{m=1}^N\{(h(\phi_m)h'(\psi), \partial_t\psi)+ (h'(\phi_m)h(\psi),\partial_t\phi_m)\}.\nonumber
\end{eqnarray}

\begin{remark}
Since volumes of each chromosome $V_m=\int_{\Omega}h(\phi^0_m)d\bx$ and heterochromatin $v_m=\int_{\Omega}h(\phi_m^0)h(\psi^0) d\bx$ are constants for $t\ge t_0$, we have $\sum\limits_{m=1}^N\{(h(\phi_m)h'(\psi), \partial_t\psi)+ (h'(\phi_m)h(\psi),\partial_t\phi_m)\} =0$ for $t\ge t_0$.   This zero term is critical for  constructing  energy stable schemes.
\end{remark}

Then,  a second-order energy stable scheme based on system \eqref{stab:eq:1}-\eqref{stab:eq:7} can be constructed as follows. 

%Given $(\phi_m^k, \psi^k, \lambda_m^k, \eta_m^k, R^k)$ for $k=0,\1,\cdots,n$, we update $(\phi_m^{n+1}, \psi^{n+1}, \lambda_m^{n+1}, \eta_m^{n+1}, R^{n+1})$  by  
\begin{align}
\frac{ \phi^{n+1}_m-\phi_m^n}{\delta t}&=-M\big(-\eps^2_\phi\Delta \phi^{n+\frac 12}_m\nonumber\\
& +R^{n+\frac 12}(\frac{\delta \tilde{E}}{\delta \phi_m})^{n,\star}
 -\lambda^{n+\frac 12}_m h'(\phi^{n,\star}_m)-\eta^{n+\frac 12}_m h'(\phi^{n,\star}_m)h(\psi^{n,\star})\big),\;m=1,\cdots,N, \label{stab:linear:scheme:1}\\
\frac{ \psi^{n+1}-\psi^n}{\delta  t}&=-M\big(-\eps^2_\psi\Delta \psi^{n+\frac 12} +R^{n+\frac 12}(\frac{\delta \tilde{E}}{\delta \psi})^{n,\star} -\sum\limits_{m=1}^N\eta_m^{n+\frac 12}h(\phi^{n,\star}_m)h'(\psi^{n,\star})\big),\label{stab:linear:scheme:2}\\
\int_{\Omega}h(\phi^0_m)d\bx&=\int_{\Omega}h(\phi^{n+1}_m)d\bx,\;m=1,\cdots,N,\label{stab:linear:scheme:3}\\
\int_{\Omega}h(\phi_m^0)h(\psi^0) d\bx &=\int_{\Omega}h(\phi^{n+1}_m)h(\psi^{n+1})d\bx,\;m=1,\cdots,N,\label{stab:linear:scheme:4}\\
\tilde{E}^{n+1}(\phi_m^{n+1},\psi^{n+1},\phi_0)&-\tilde{E}^n(\phi_m^n,\psi^n,\phi_0) = R^{n+\frac 12}\sum\limits_{m=1}^N((\frac{\delta \tilde{E}}{\delta \phi_m})^{n,\star},\phi_m^{n+1}-\phi^n_m)\nonumber\\ &+R^{n+\frac 12}((\frac{\delta \tilde{E}}{\delta \psi})^{n,\star},\psi^{n+1}-\psi^n)+ (h'(\phi^{n,\star}_m),\phi^{n+1}_m-\phi_m^n)\nonumber\\&
+\sum\limits_{m=1}^N\{(h(\phi_m^{n,\star})h'(\psi^{n,\star}), \psi^{n+1}-\psi^n)+ (h'(\phi^{n,\star}_m)h(\psi^{n,\star}),\phi^{n+1}_m-\phi_m^n)\}.\label{stab:linear:scheme:5}
\end{align}
The above scheme is coupled and weakly nonlinear as \eqref{stab:linear:scheme:3}--\eqref{stab:linear:scheme:5} lead to a system of nonlinear algebraic equations for the Lagrange multipliers. The scheme can be efficiently solved as follows: 

For $m=1,2,\cdots,N$,  setting  
\begin{equation}\label{update:1}
\phi_m^{n+1}  =  \phi_{1,m}^{n+1} + \lambda_m^{n+\frac 12} \phi_{2,m}^{n+1} + \eta_m^{n+\frac 12}\phi_{3,m}^{n+1} + R^{n+\frac 12}\phi_{4,m}^{n+1},
\end{equation}
in \eqref{stab:linear:scheme:1}-\eqref{stab:linear:scheme:2}, we find that $\phi_{2,m}^{n+1}$ and $\phi_{3,m}^{n+1}$ are determined again by \eqref{part:1:1}-\eqref{part:1:2}, while
 $\phi_{1,m}^{n+1}$ and $\phi_{4,m}^{n+1}$ can be determined by
\begin{equation}\label{part:10}
\begin{split}
\frac{ \phi^{n+1}_{1,m}-\phi_m^n}{\delta t}=-M(-\eps^2_\phi\Delta \phi^{n+\frac 12}_{1,m});
\end{split}
\end{equation}
and 
\begin{equation}\label{part:11}
\begin{split}
\frac{ \phi^{n+1}_{4,m}}{\delta t}=-M(-\eps^2_\phi\Delta \phi^{n+\frac 12}_{4,m}+(\frac{\delta \tilde{E}}{\delta \phi_m})^{n,\star}).
\end{split}
\end{equation}
On the other hand, setting
\begin{equation}\label{update:2}
\psi^{n+1} =  \psi_1^{n+1}  +   \sum\limits_{m=1}^N\eta_m^{n+\frac 12}\psi^{n+1}_{2,m} +  R^{n+\frac 12}\psi_3^{n+1},  
\end{equation}
in  \eqref{stab:linear:scheme:1}-\eqref{stab:linear:scheme:2},  we find that $\psi^{n+1}_{2,m}$ is still determined by \eqref{part:2:2},  while 
$\psi_1^{n+1}$ and $\psi_3^{n+1}$ can be determined by 
\begin{equation}\label{part:20}
\begin{split}
\frac{ \psi^{n+1}_1-\psi^n}{\delta  t}=-M(-\eps^2_\psi\Delta \psi_1^{n+\frac 12});
\end{split}
\end{equation}
and 
\begin{equation}\label{part:21}
\begin{split}
\frac{ \psi^{n+1}_3}{\delta  t}=-M(-\eps^2_\psi\Delta \psi_3^{n+\frac 12} +(\frac{\delta \tilde{E}}{\delta \psi})^{n,\star}).
\end{split}
\end{equation}
Finally, we plug   \eqref{update:1} and \eqref{update:2} into  \eqref{stab:linear:scheme:3}-\eqref{stab:linear:scheme:5} to obtained a coupled  nonlinear algebraic system of $2N+1$ equations for  $(\lambda_m^{n+\frac 12},\eta_m^{n+\frac 12}, \,m=1,2,\cdots,N)$ and $R^{n+\frac 12}$. 
Hence,  compared with the scheme  the scheme \eqref{scheme:1}-\eqref{scheme:5}, \eqref{stab:linear:scheme:1}-\eqref{stab:linear:scheme:5} involves a slightly more  complicated nonlinear algebraic system which may require small time steps to have suitable solutions. 

In summary,  we can determine $\phi_m^{n+1}$ and $\psi^{n+1}$ as follows:
\begin{itemize}
\item  Solve  $(\phi_{1,m}^{n+1}, \phi_{2,m}^{n+1},\phi_{3,m}^{n+1},\phi_{4,m}^{n+1})$ and $\psi_{2,m}^{n+1}$ for $m=1,2,\cdots,N$ from \eqref{part:1:1}-\eqref{part:1:2}, \eqref{part:2:2} and \eqref{part:10}-\eqref{part:11}, and solve  $(\psi_1^{n+1},\psi_3^{n+1})$ from \eqref{part:20}-\eqref{part:21}.
\item  Solve  $(\lambda_m^{n+\frac 12},\eta_m^{n+\frac 12}, \,m=1,2,\cdots,N)$ and $R^{n+\frac 12}$  from the coupled  nonlinear system \eqref{stab:linear:scheme:3}-\eqref{stab:linear:scheme:5}.
\item  Update $\phi_m^{n+1}$ ($m=1,2,\cdots,N$) and $\psi^{n+1}$ from equations \eqref{update:1} and  \eqref{update:2}.
\end{itemize}

\begin{theorem}
Let $(\phi^{n+1}_m,\psi^{n+1},\lambda^{n+1}_m,\eta^{n+1}_m,R^{n+1})$ be the solution of   \eqref{stab:linear:scheme:1}-\eqref{stab:linear:scheme:5} with \eqref{rate2}. 
Then the following energy dissipation law  is satisfied unconditionally:
\begin{equation}
\frac{E^{n+1}-E^n}{\delta t}   \leq- M(\sum\limits_{m=1}^N\|\frac{ \phi^{n+1}_m-\phi_m^n}{\delta t} \|^2+\|\frac{ \psi^{n+1}_m-\psi_m^n}{\delta t}\|^2),\quad \forall n\ge t_0/{\delta t}
\end{equation}
where  the energy $E^{n+1}$ is defined as 
\begin{equation}
E^{n+1}=\sum\limits_{m=1}^N\frac{\eps_{\phi}^2}{2}\|\Grad\phi^{n+1}_m\|^2 +\frac{\eps_{\psi}^2}{2}\|\Grad\psi^{n+1}\|^2+\tilde{E}^{n+1}.
\end{equation}
\end{theorem}
\begin{proof}
Note that for $ n\ge t_0/{\delta t}$, we have $V_m(t^n)=\bar V_m$ and $v_m(t^n)=\bar v_m$.

Taking inner product of equation \eqref{stab:linear:scheme:1} with $-\frac{1}{M}\frac{ \phi^{n+1}_m-\phi_m^n}{\delta t}$, we obtain 
\begin{equation}\label{stab:1}
\begin{split}
&-\frac 1M\|\frac{ \phi^{n+1}_m-\phi_m^n}{\delta t}\|^2=-\eps^2_\phi(\Delta \phi^{n+\frac 12}_m,\frac{ \phi^{n+1}_m-\phi_m^n}{\delta t})\\&+(R^{n+\frac 12}(\frac{\delta \tilde{E}}{\delta \phi_m})^{n,\star},\frac{ \phi^{n+1}_m-\phi_m^n}{\delta t})
 -\lambda^{n+\frac 12}_m (h'(\phi^{n,\star}_m),\frac{ \phi^{n+1}_m-\phi_m^n}{\delta t})\\&-\eta^{n+\frac 12}_m (h'(\phi^{n,\star}_m)h(\psi^{n,\star}),\frac{ \phi^{n+1}_m-\phi_m^n}{\delta t}).
 \end{split}
\end{equation}
Taking inner product of equation \eqref{stab:linear:scheme:2} with $-\frac 1M\frac{ \psi^{n+1}-\psi^n}{\delta  t}$, we derive
\begin{equation}\label{stab:2}
\begin{split}
&-\frac 1M\|\frac{ \psi^{n+1}-\psi^n}{\delta t}\|^2=-\eps^2_\psi(\Delta \psi^{n+\frac 12},\frac{ \psi^{n+1}-\psi^n}{\delta  t}) \\&+R^{n+\frac 12}((\frac{\delta \tilde{E}}{\delta \psi})^{n,\star},\frac{ \psi^{n+1}-\psi^n}{\delta  t})-\sum\limits_{m=1}^N\eta^{n+\frac 12}_m(h(\phi_m^{n,\star})h'(\psi^{n,\star}),\frac{ \psi^{n+1}-\psi^n}{\delta  t}).
\end{split}
\end{equation}
On the other hand, we have 
\begin{equation}\label{stab:3}
(\Delta \phi^{n+\frac 12}_m,\frac{ \phi^{n+1}_m-\phi_m^n}{\delta t})=-\frac{1}{2\delta t}(\|\Grad\phi_m^{n+1}\|^2-\|\Grad\phi_m^n\|^2),
\end{equation}
and
\begin{equation}\label{stab:4}
(\Delta \psi^{n+\frac 12},\frac{ \psi^{n+1}-\psi^n}{\delta  t}) =-\frac{1}{2\delta t}(\|\Grad\psi^{n+1}\|^2-\|\Grad\psi^n\|^2).
\end{equation}
Summing up equations \eqref{stab:1} for $m=1,2,\cdots,N$ and combined with equation \eqref{stab:3}, we obtain
\begin{equation}\label{stab:5}
\begin{split}
 &-\frac 1M\sum\limits_{m=1}^N\|\frac{ \phi^{n+1}_m-\phi_m^n}{\delta t}\|^2-\frac 1M\|\frac{ \psi^{n+1}-\psi^n}{\delta t}\|^2=\sum\limits_{m=1}^N\Big\{-\eps^2_\phi(\Delta \phi^{n+\frac 12}_m,\frac{ \phi^{n+1}_m-\phi_m^n}{\delta t})\\&+(R^{n+\frac 12}(\frac{\delta \tilde{E}}{\delta \phi_m})^{n,\star},\frac{ \phi^{n+1}_m-\phi_m^n}{\delta t})
 -\lambda^{n+\frac 12}_m (h'(\phi^{n,\star}_m),\frac{ \phi^{n+1}_m-\phi_m^n}{\delta t})\\&-\eta^{n+\frac 12}_m (h'(\phi^{n,\star}_m)h(\psi^{n,\star}),\frac{ \phi^{n+1}_m-\phi_m^n}{\delta t})\Big\}-\eps^2_\psi(\Delta \psi^{n+\frac 12},\frac{ \psi^{n+1}-\psi^n}{\delta  t}) \\&+R^{n+\frac 12}((\frac{\delta \tilde{E}}{\delta \psi})^{n,\star},\frac{ \psi^{n+1}-\psi^n}{\delta  t})-\sum\limits_{m=1}^N\eta^{n+\frac 12}_m(h(\phi_m^{n,\star})h'(\psi^{n,\star}),\frac{ \psi^{n+1}-\psi^n}{\delta  t}).
\end{split}
\end{equation}
Using \eqref{stab:linear:scheme:2},  \eqref{stab:linear:scheme:4} and combing \eqref{stab:1}--\eqref{stab:4},   equation \eqref{stab:5} reduces to 
\begin{equation}\label{stab:6}
\begin{split}
 &-\frac 1M\sum\limits_{m=1}^N\|\frac{ \phi^{n+1}_m-\phi_m^n}{\delta t}\|^2-\frac 1M\|\frac{ \psi^{n+1}-\psi^n}{\delta t}\|^2=\sum\limits_{m=1}^N
 \frac{\eps_\phi^2}{2\delta t}(\|\Grad\phi_m^{n+1}\|^2-\|\Grad\phi_m^n\|^2)\\&+\frac{\eps_\psi^2}{2\delta t}(\|\Grad\psi^{n+1}\|^2-\|\Grad\psi^n\|^2)+\tilde{E}^{n+1}(\phi_m^{n+1},\psi^{n+1},\phi_0)-\tilde{E}^n(\phi_m^n,\psi^n,\phi_0).
\end{split}
\end{equation}
Finally, from \eqref{stab:6}  we arrive at the desired result.
\end{proof}

\section{Numerical simulations}
In this section, we consider the application of nuclear architecture reorganization system \eqref{nuclear:1}-\eqref{nuclear:5} to model drosophila nucleus with $8$ chromosomes and human nucleus with $46$ chromosomes.  We present  numerical simulations to explore the mechanisms underlying the reorganization process. The default computational domain is $\Omega=[-\pi,\pi)^2$  and  $(2,2.9)$ is chosen to be  the  $x$-diameter and  $y$-diameter  of an  elliptic nucleus which is located  in the center of domain $\Omega$.  We use a Fourier spectral method in space with $256^2$ modes, coupled with the three time discretization schemes presented in the last section.    When
presenting the simulations results, nucleus is depicted in white, chromosome territories in green, and heterochromatin in red ( see Fig.\;\ref{ini}).  

First, we test the convergence rate for proposed  linear  scheme and weakly nonlinear schemes. Then we study  the conventional architectures with affinity and without affinity. Finally, we explore the mechanisms underlying the reorganization process and the pattern formation of chromatin, e.g., the effect of  nucleus size and shape and different phase parameters.

\subsection{Convergence test}
We first test the convergence rate for the linear scheme \eqref{linear:scheme:1}-\eqref{linear:scheme:6} and the weakly nonlinear scheme \eqref{stab:linear:scheme:1}-\eqref{stab:linear:scheme:5} with fixed nucleus. The phase  parameters are set to be $\eps^2_{\phi}=0.01$, $\eps^2_{\psi}=0.05$, $\beta_0=\frac 53$, $\beta_\phi=1$, $\beta_\psi=\frac 23$ and $\gamma=0$.  The initial condition is chosen as the case of (b) in Fig.\;\ref{ini} with  Affinity $>0$.  The reference solutions are obtained with a very small time step $\delta t=10^{-6}$ using the linear scheme  \eqref{linear:scheme:1}-\eqref{linear:scheme:6}.  We plot
$\max_{m=1}^N\|\phi_m-\phi_{m,Ref}\|_{L^{\infty}}$ and $\|\psi-\psi_{Ref}\|_{L^{\infty}}$ in Fig.\;\ref{Order_test}. Second order convergence rate is observed for both  schemes.

\begin{figure}[htbp]
\centering
\subfigure{
\includegraphics[width=0.45\textwidth,clip==]{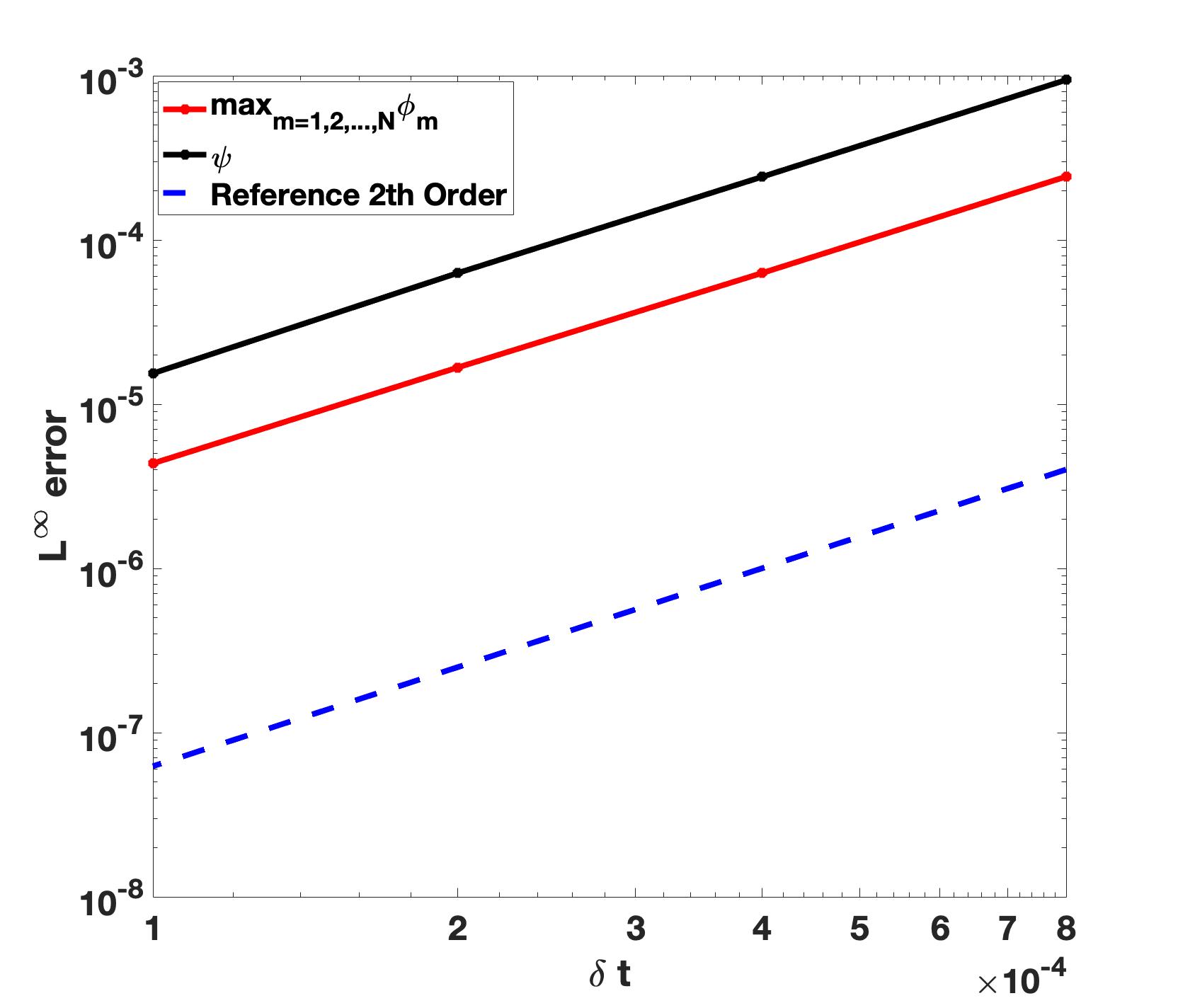}}
\subfigure{
\includegraphics[width=0.45\textwidth,clip==]{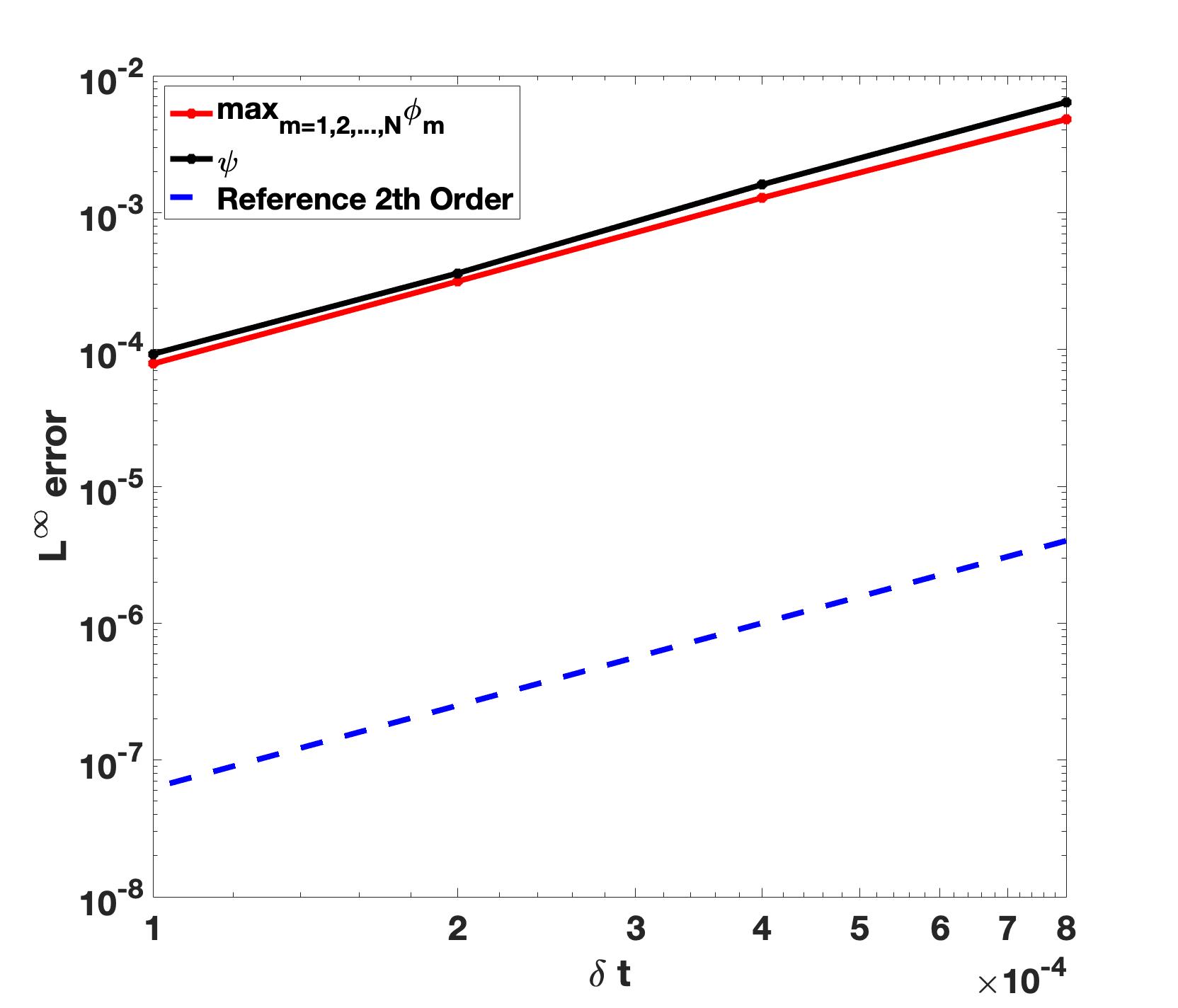}}
\caption{Convergence rate of linear scheme \eqref{linear:scheme:1}-\eqref{linear:scheme:6} and weakly nonlinear scheme \eqref{stab:linear:scheme:1}-\eqref{stab:linear:scheme:5} with fixed nucleus.}\label{Order_test}
\end{figure}

\subsection{Affinity test and conventional architecture with fixed nucleus}

We now demonstrate the conventional architecture for drosophila nucleus with $8$ chromosomes. The initial condition is given in Fig.\;\ref{ini} where an elliptic nucleus are generated by function $\phi_0(x,y)=\frac 12(1-\tanh(\frac{\sqrt{x^2+y^2/1.45^2}-2}{\sqrt{2}\eps}))$.  The $8$ chromosomes  are initialized by tanh-like functions :$\frac 12[1-\tanh(\frac{r}{\sqrt{2}\eps})]$ with centers at  $(0,2.5), (-1,1.4), (-0.3,-0.5), (1,-1), (0,0.6), (1,1.3), (0,-2.5), (-1,-0.8)$.  The $x$-diameter and $y$-diameter are $(0.2,0.4)$ for each elliptic nucleus.  A smaller ellipse with $x$-diameter and $y$-diameter as $(0.05,0.1)$ in each chromosome  is set to be heterochromatin territory. The affinity between heterochromatin and the nuclear envelope is controlled by the parameter $\gamma$. A positive affinity value indicates a tethering of heterochromatin to LBR or lamin A/C on the nuclear envelope.   To demonstrate that  the conventional architecture is obtained with the positive affinity, we choose $\gamma=0.02$ and $\gamma=0$  and plot in Fig.\;\ref{ini}   numerical results by using the  weakly nonlinear scheme \eqref{scheme:1}-\eqref{scheme:5}. We observe from Fig.\;\ref{ini} that heterochromatin domains are fused with adjacent heterochromatin.  When affinity $=0$ heterochromatin  accumulates at the territories between chromosomes. But there is no interaction with the region of the nuclear envelope.  With a  positive affinity, heterochromatin is observed to be distributed almost homogeneously along the nuclear envelope,  indicating the formation of the conventional architecture.  Our numerical simulations indicate that the affinity plays important roles in forming the conventional architecture, and  that  the expression of LBR and lamin A/C is essential to generate the conventional architecture. These numerical results  are consistent with the experiment results in \cite{solovei2013lbr}.

In Fig.\;\ref{volume}, we plot the dynamics of mean volume of chromosome $\bar{V}(t)=\frac{\sum_{i=1}^NV_m}{N}$ and heterochromatin $\bar{v}(t)=\frac{\sum_{i=1}^Nv_m}{N}$. From Fig.\;\ref{volume}, the volumes of chromosome and heterochromain are well preserved  by using our weakly nonlinear schemes \eqref{scheme:1}-\eqref{scheme:5}.

\begin{figure}
\centering
\subfigure[Initial condition.]{
\includegraphics[width=0.3\textwidth,clip==]
{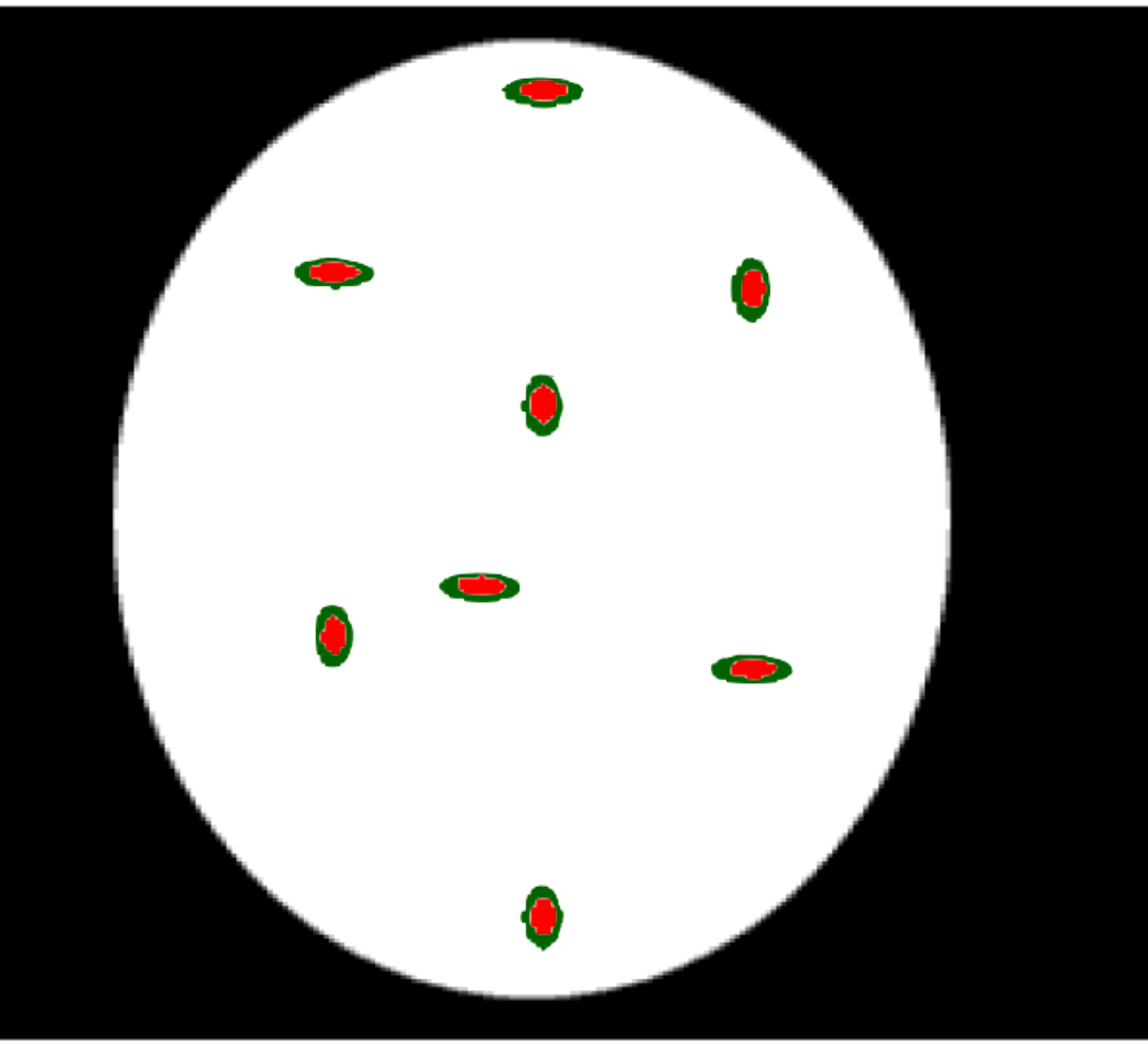}}\hskip 0cm
\subfigure[$\mbox{Affinity}>0$.]{
\includegraphics[width=0.3\textwidth,clip==]
{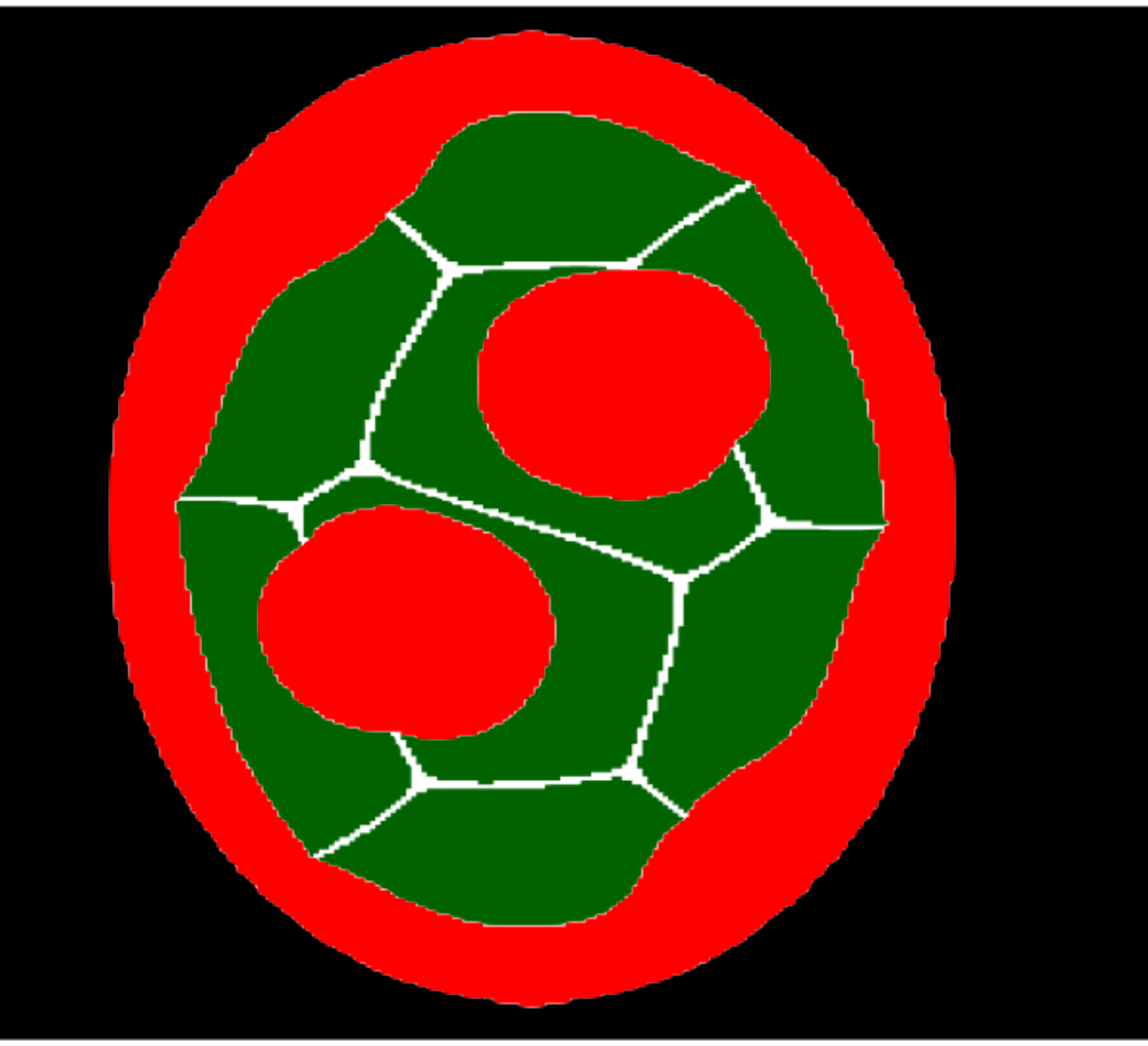}}\hskip 0cm
\subfigure[$\mbox{Affinity}=0$.]{
\includegraphics[width=0.3\textwidth,clip==]
{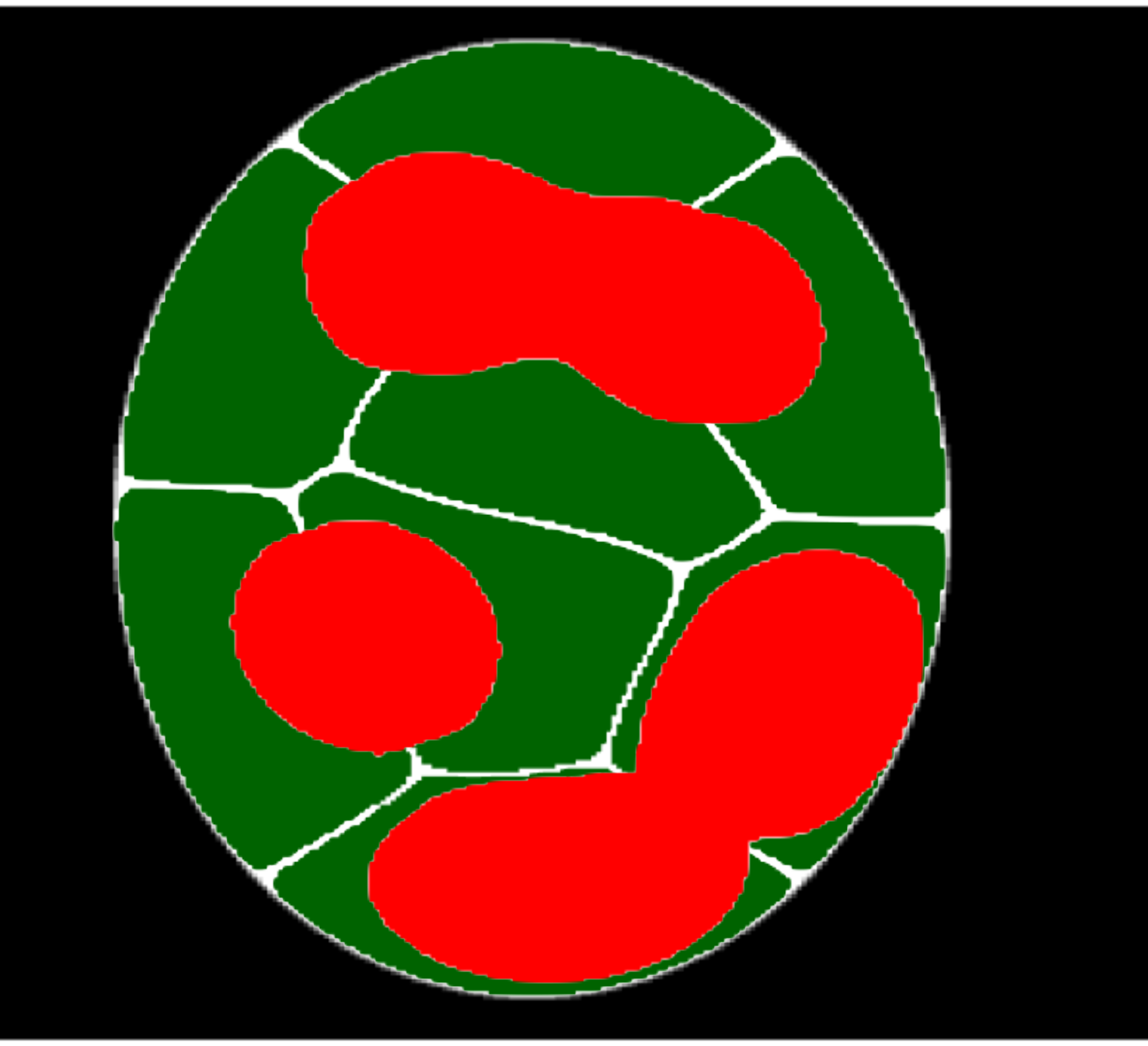}}\hskip 0cm
\subfigure[$\mbox{Affinity}=0$.]{
\includegraphics[width=0.3\textwidth,clip==]
{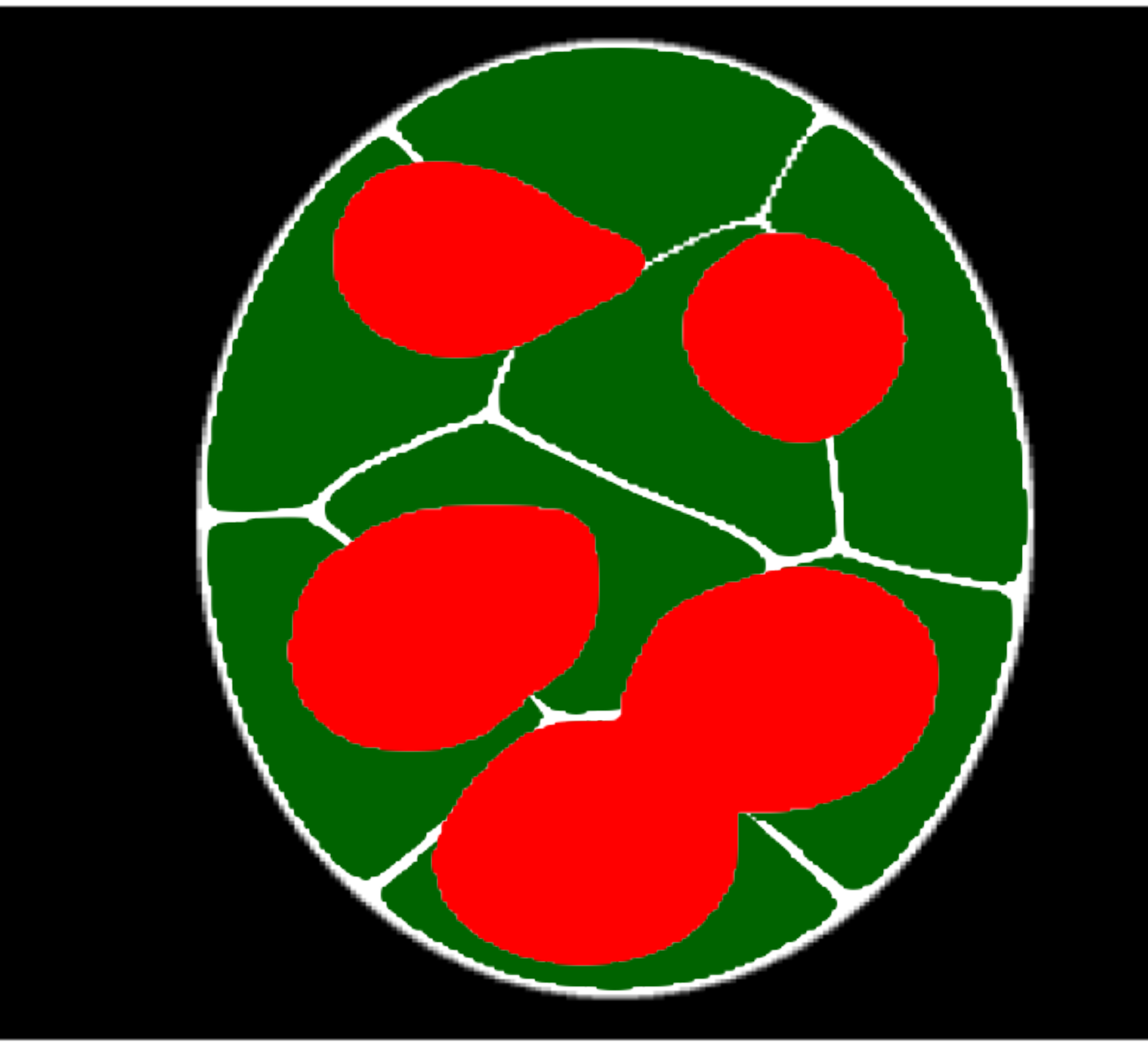}}\hskip 0cm
\caption{The parameters for nuclear reorganization process: $\beta_0=\frac 53$, $\beta_\phi=\frac 83$, $\beta_\psi=\frac 83$ with  $\gamma=0.02$ for positive affinity and $\gamma=0$ for zero affinity.  Interface parameters are $\eps_{\phi}^2=0.01$,  $\eps_{\psi}^2=0.05$.  $\bar{V}_m=\frac{\mbox{Nuclear Volume}}{N}$ and $\bar{v}_m=V_m\times 0.23$ where $m=8$ for drosophila nucleus. Volume growth rate parameters $\alpha_1=1$, $\alpha_2=10$ in \eqref{rate2}. }\label{ini}
\end{figure}

\begin{figure}[htbp]
\centering
%\includegraphics[width=0.75\textwidth,clip==]{volume.jpg}
%\subfigure[With affinity.]{
%\includegraphics[width=0.45\textwidth,clip==]{volume_dynamic.pdf}}
\includegraphics[width=0.45\textwidth,clip==]{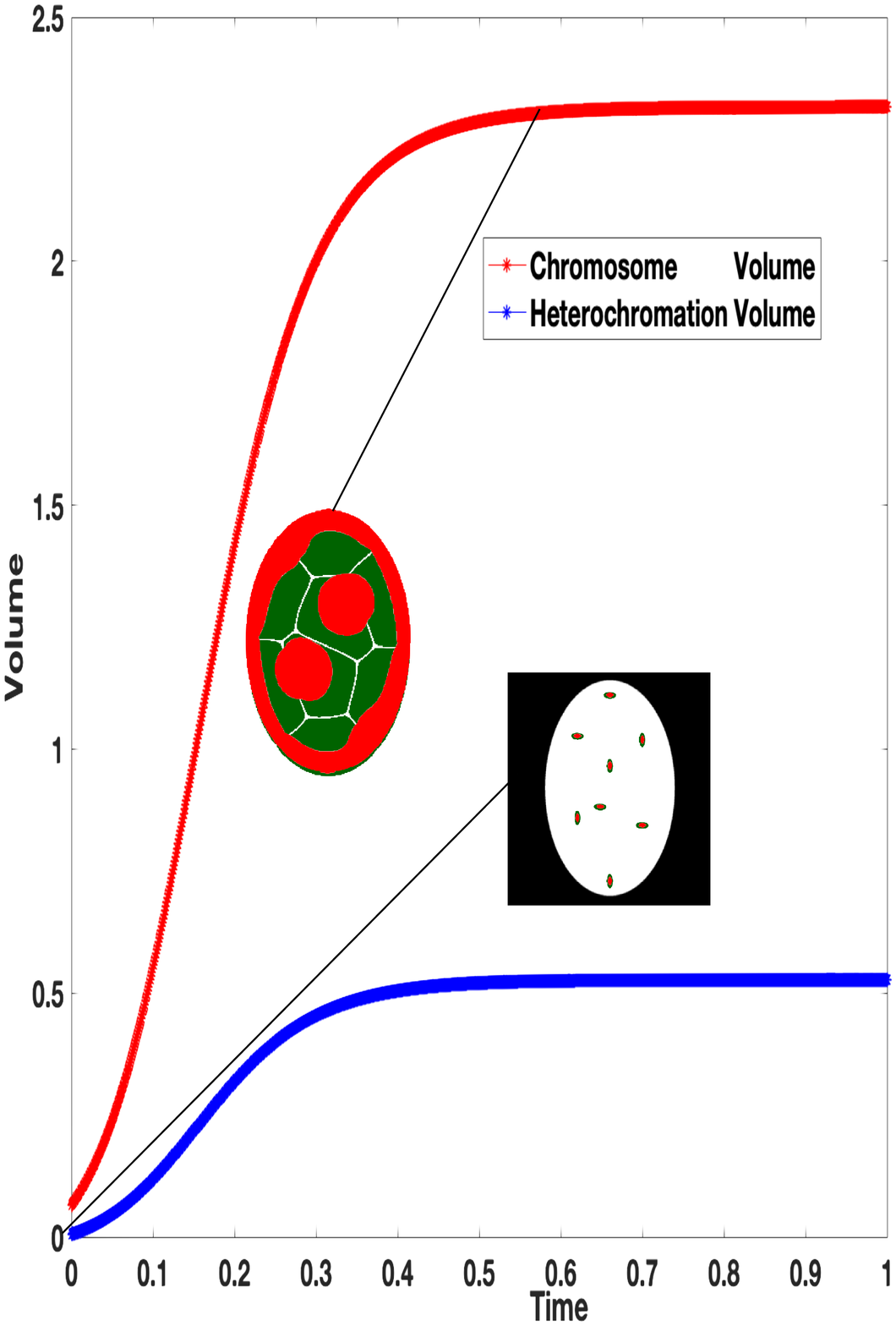}
\caption{Evolutions of  mean  volume of chromosome and heterochromation $\bar{V}(t)=\frac{\sum_{i=1}^NV_m}{N}$,  $\bar{v}(t)=\frac{\sum_{i=1}^Nv_m}{N}$ with respect to time for nuclear reorganization process with affinity $\gamma=0.02$.}\label{volume}
\end{figure}

\subsection{Inverted architecture and reorganization process}
In this subsection, we  study the  architecture reorganization process with fixed nucleus. First, we examine whether the increase of  heterochromatin conversion rate and the absence of  affinity  between the nuclear envelope and heterochromatin  are necessary for the induction of the single hetero-cluster in the inverted architecture.  We fix the heterochromatin conversion rate $\rho_m=\frac{v_m}{V_m}$ for $m=1,2,\cdots, N$,  and  set $\gamma=0$.  From the first row of Fig.\;\ref{process}, it is observed that affinity between  heterochromatin and nuclear envelope vanishes gradually. Finally, four clusters of heterochromatin are formed at $t=50$ when the conversion rates are  fixed for all $m$.  We then examine the case with  an increasing conversion rate $\rho_m(t)$  described by 
 \begin{equation}
 \rho_{m}(t) =\rho_m(0) +\frac{\bar{\rho}_mt}{t + \alpha_1 e^{-\alpha_2 t}},
 \end{equation}
 where $\alpha_1=150$ and $\alpha_2=0.3$.  In our simulations, we set the increased conversion rate to be $\bar{\rho}_m=\{(0.35,0.4,0.4,0.35,0.15,0.15,0.35,0.35)\}$ and $\bar{\rho}_m=\{(0.35,0.4,0.4,0.45,0.15,0.15,0.35,0.45)\}$   for the second and third rows in Fig.\;\ref{process},   and set the affinity parameter to  be $\gamma=0$. We observe from the second and third rows of Fig.\;\ref{process} that a single cluster of heterochromatin is formed  which  implies  the inverted architecture.  Next we keep  the affinity between the nuclear and the nuclear envelope unchanged  at $\gamma=0.02$, and  increase the conversion rate $\rho_m$ for $m=1,2,\cdots, N$.    We observe from the fourth row of  Fig.\;\ref{process}  that the affinity between nuclear envelope and heterochromatin are present all the time, and   the heterochromatin  grows on each  chromosome territory gradually during architecture reorganization process.

The numerical simulations from Fig.\;\ref{process} indicate that  increase of 
heterochromatin  conversion rate and the absence of affinity between nuclear envelope and heterochromatin are sufficient for the formation of the inverted architecture during the nuclear architecture reorganization process, which are   with the previous results in \cite{lee2017new}.

\begin{figure}
\centering
\subfigure[$t=0$, $\gamma=0$.]{
\includegraphics[width=0.20\textwidth,clip==]
{lag_nuclear_affinity.pdf}}\hskip 0cm
\subfigure[$t=10$, $\gamma=0$.]{
\includegraphics[width=0.20\textwidth,clip==]{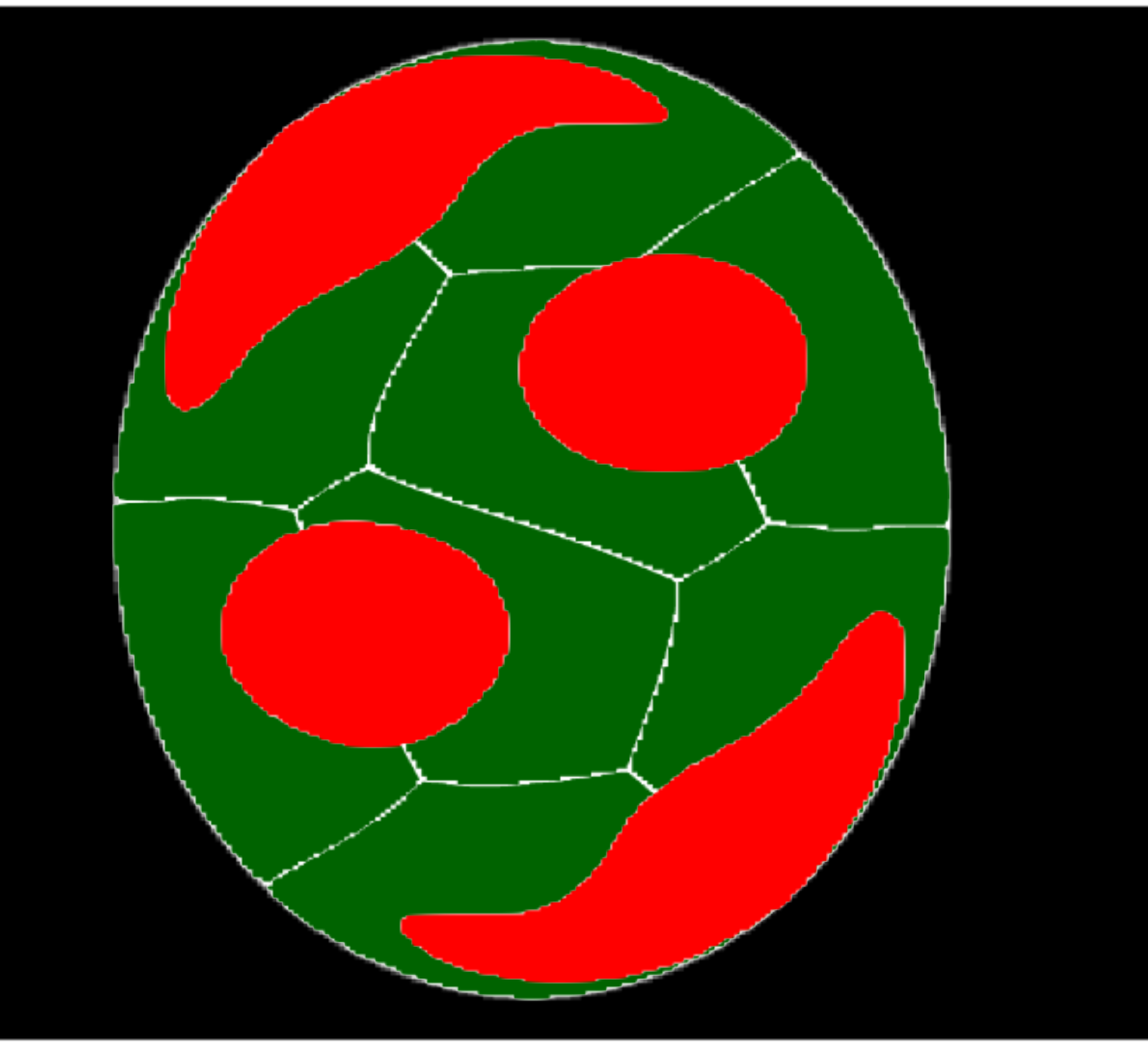}}
\subfigure[$t=20$, $\gamma=0$.]{
\includegraphics[width=0.20\textwidth,clip==]{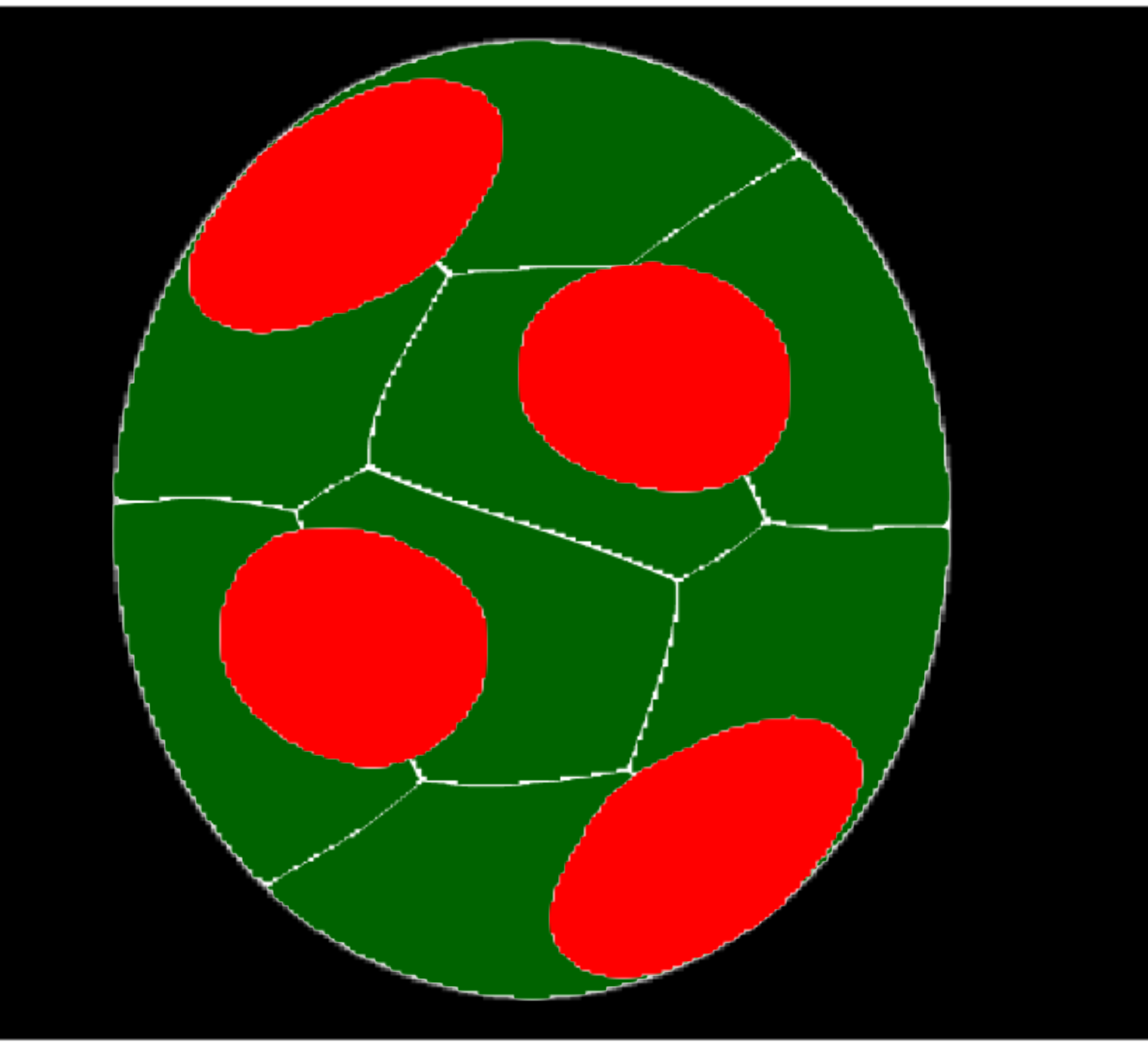}}\hskip 0cm
\subfigure[$t=50$, $\gamma=0$.]{
\includegraphics[width=0.20\textwidth,clip==]{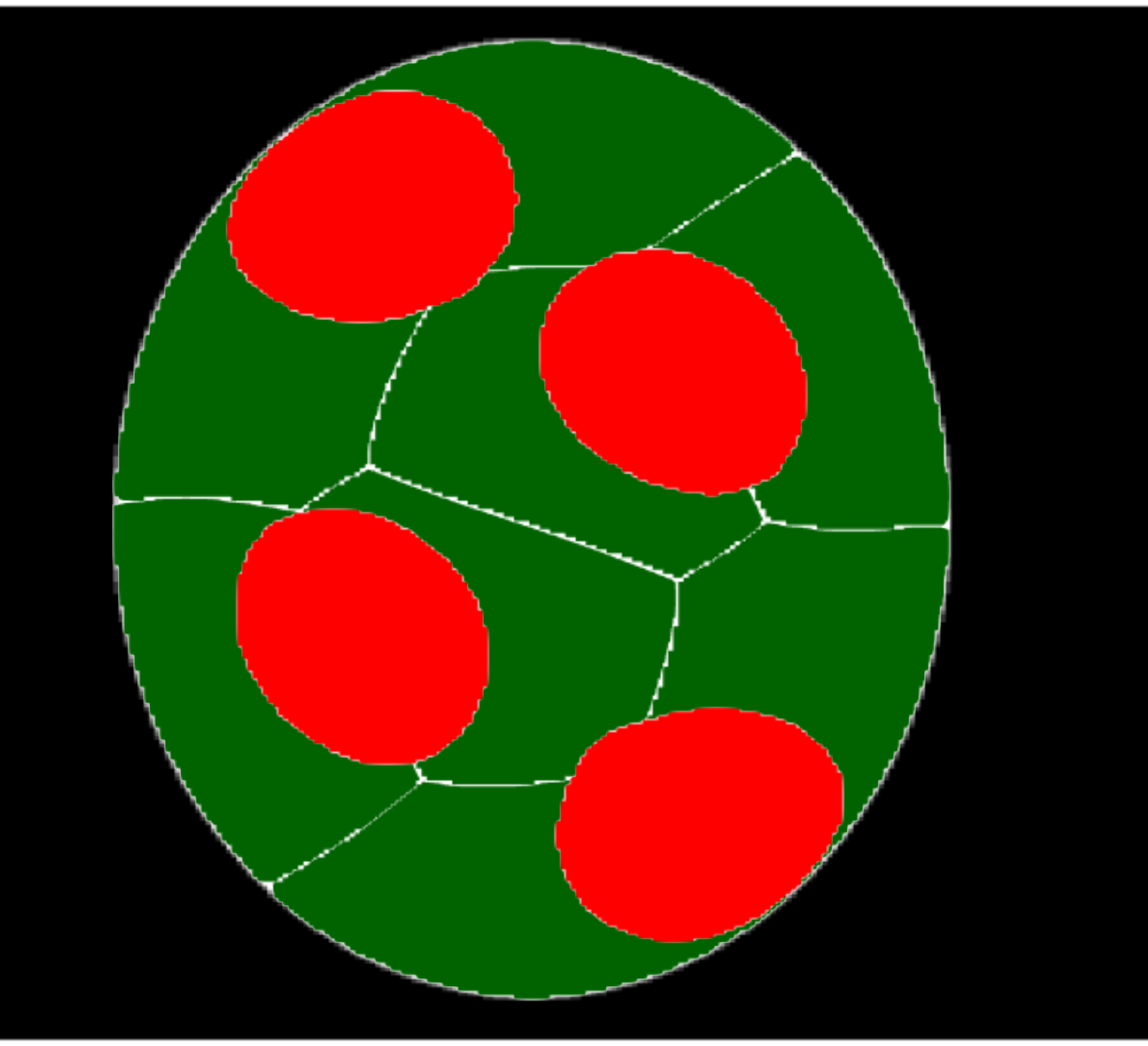}}
\subfigure[$t=0$, $\gamma=0$.]{
\includegraphics[width=0.20\textwidth,clip==]
{lag_nuclear_affinity.pdf}}\hskip 0cm
\subfigure[$t=10$, $\gamma=0$.]{
\includegraphics[width=0.20\textwidth,clip==]{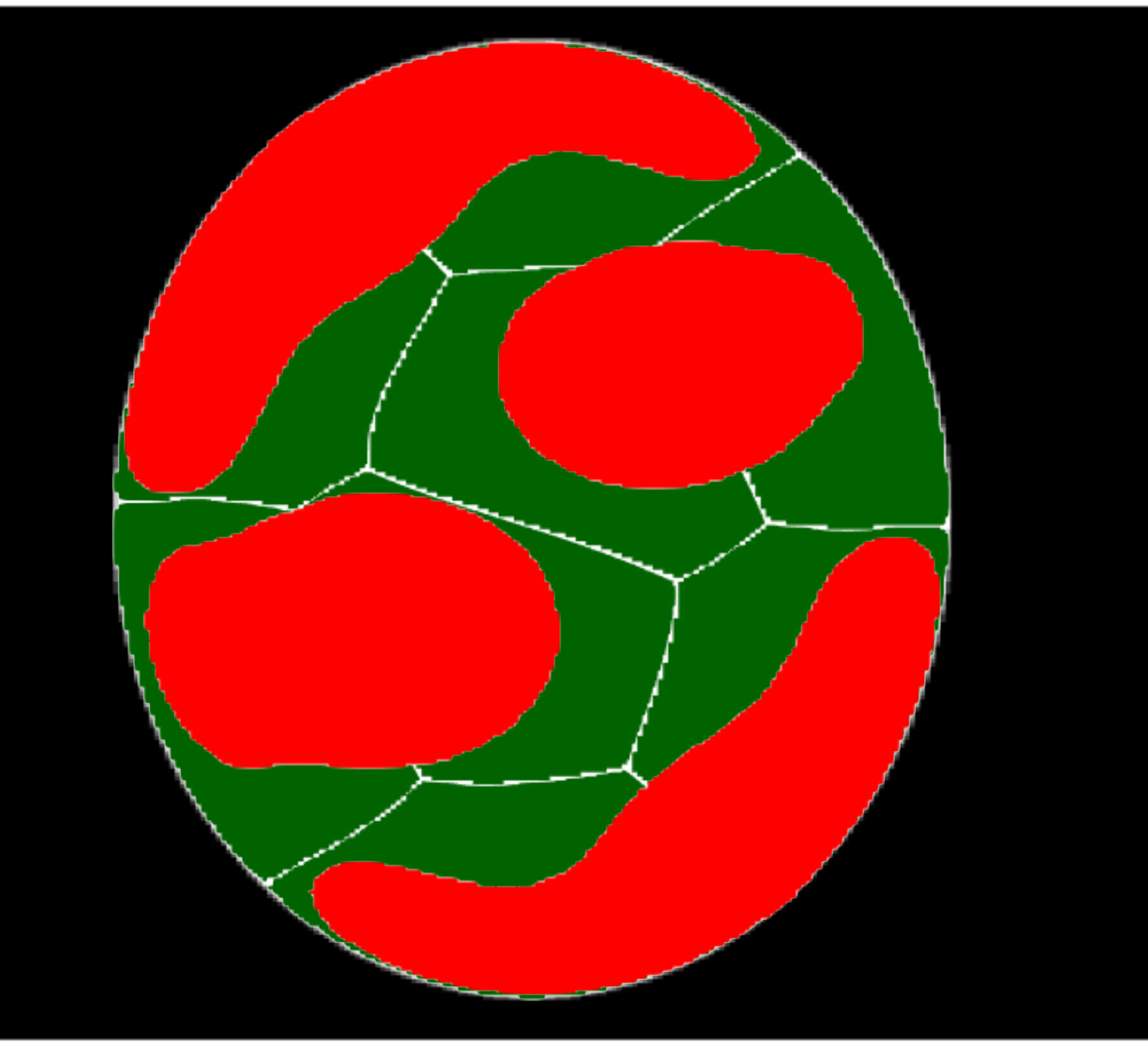}}
\subfigure[$t=20$, $\gamma=0$.]{
\includegraphics[width=0.20\textwidth,clip==]{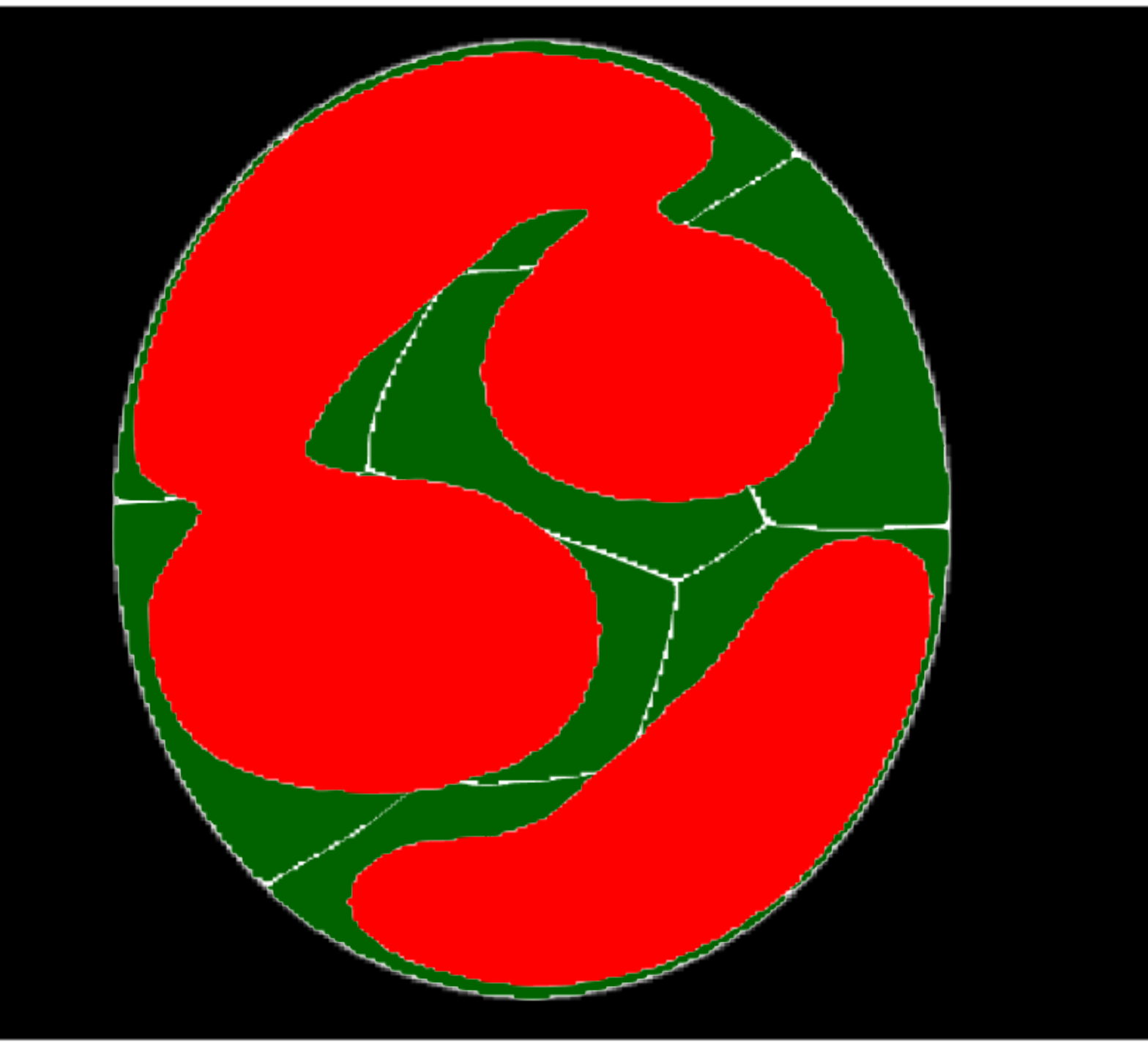}}\hskip 0cm
\subfigure[$t=50$, $\gamma=0$.]{
\includegraphics[width=0.20\textwidth,clip==]{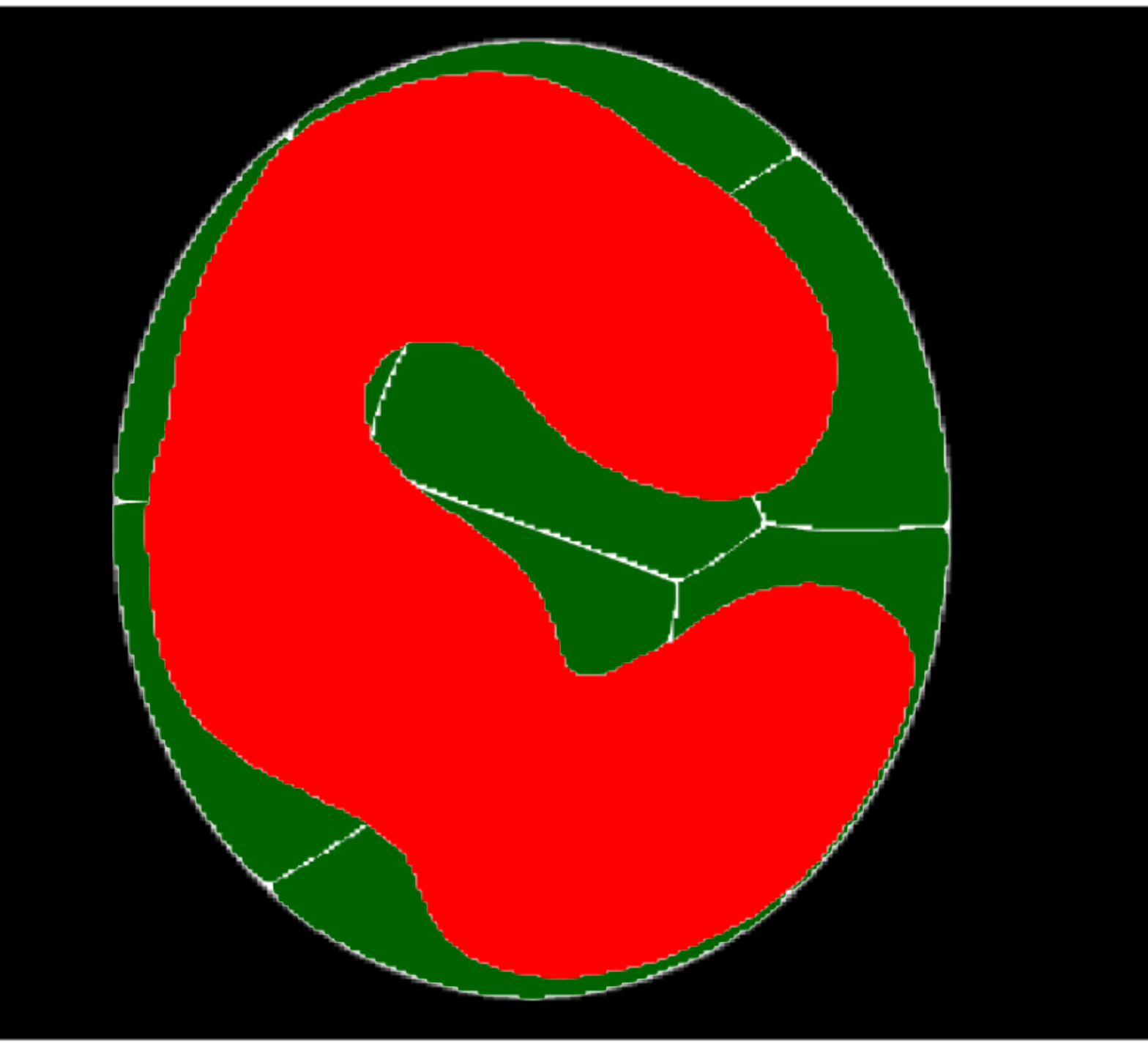}}
\subfigure[$t=0$, $\gamma=0$.]{
\includegraphics[width=0.20\textwidth,clip==]
{lag_nuclear_affinity.pdf}}\hskip 0cm
\subfigure[$t=5$, $\gamma=0$.]{
\includegraphics[width=0.20\textwidth,clip==]{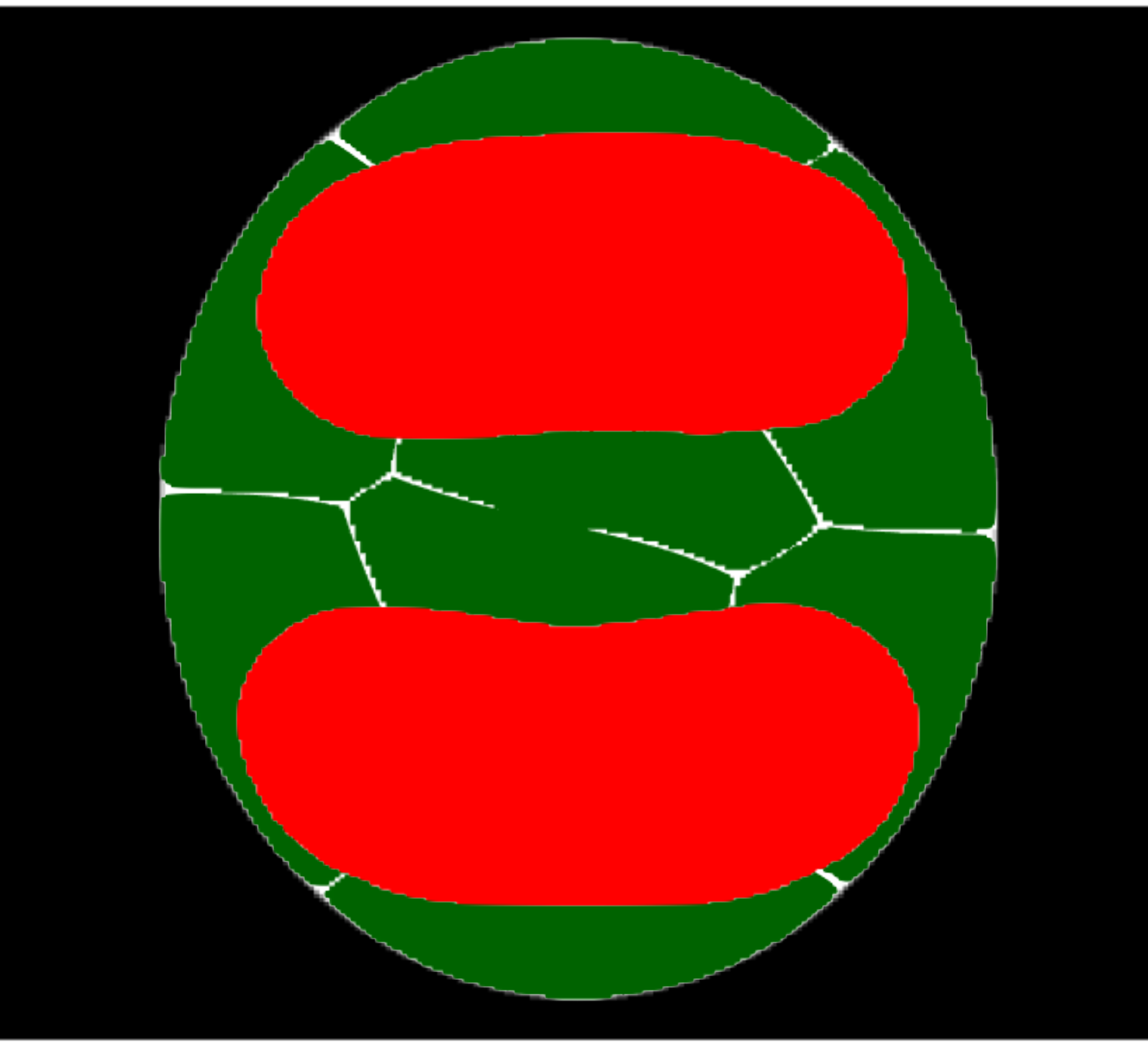}}
\subfigure[$t=10$, $\gamma=0$.]{
\includegraphics[width=0.20\textwidth,clip==]{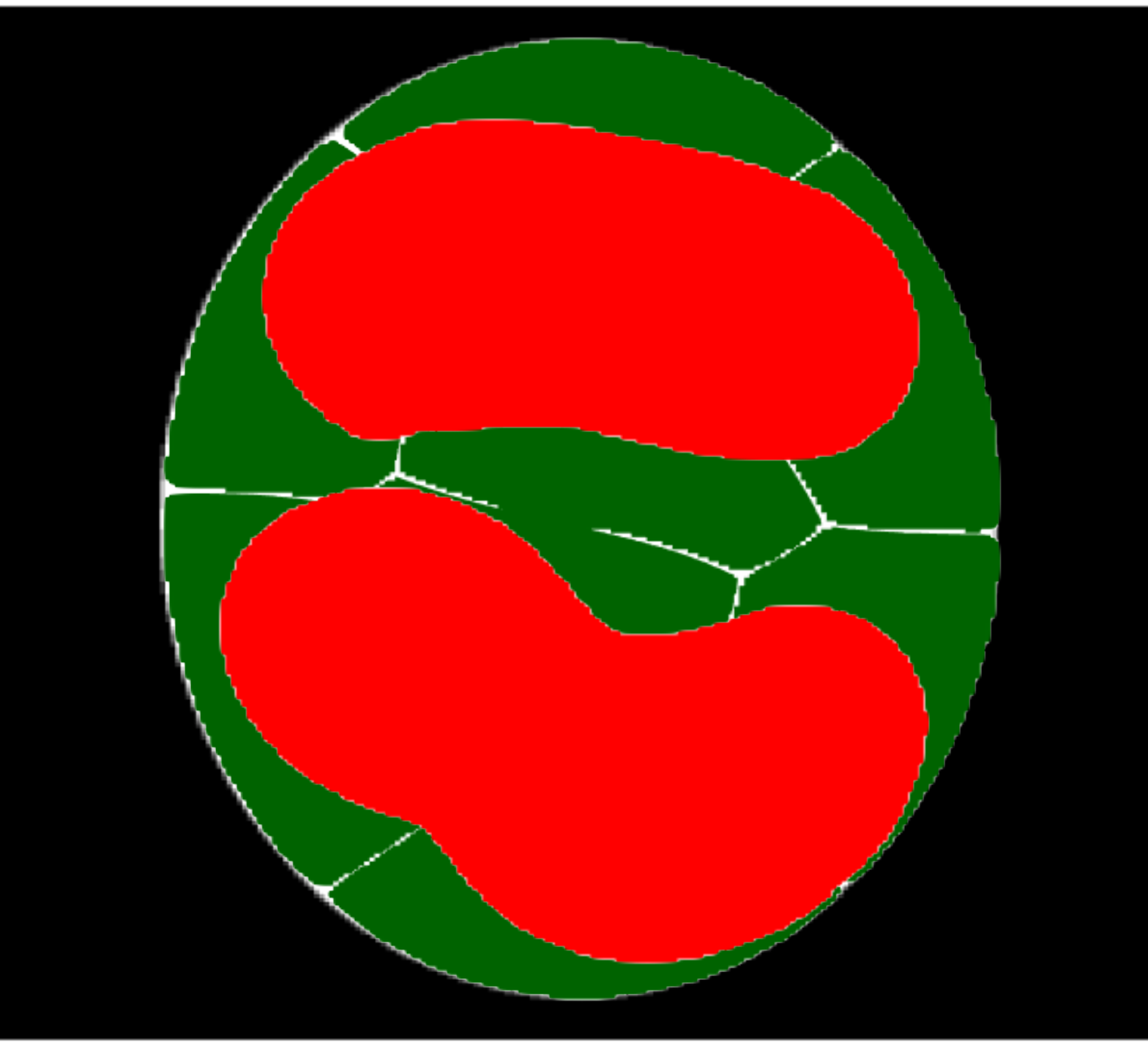}}\hskip 0cm
\subfigure[$t=50$, $\gamma=0$.]{
\includegraphics[width=0.20\textwidth,clip==]{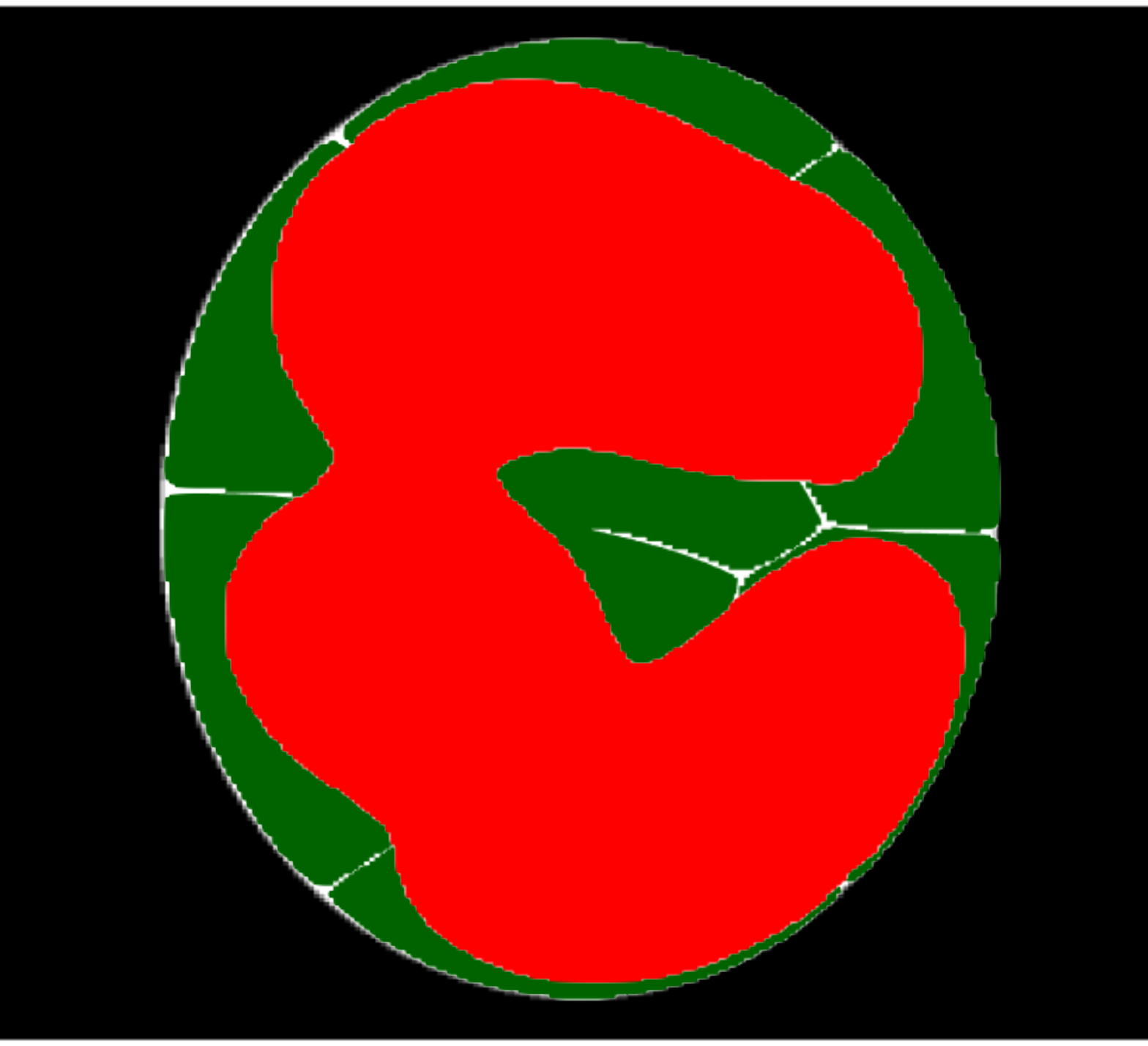}}
\subfigure[$t=0$, $\gamma=0.02$.]{
\includegraphics[width=0.20\textwidth,clip==]
{lag_nuclear_affinity.pdf}}\hskip 0cm
\subfigure[$t=10$, $\gamma=0.02$.]{
\includegraphics[width=0.20\textwidth,clip==]{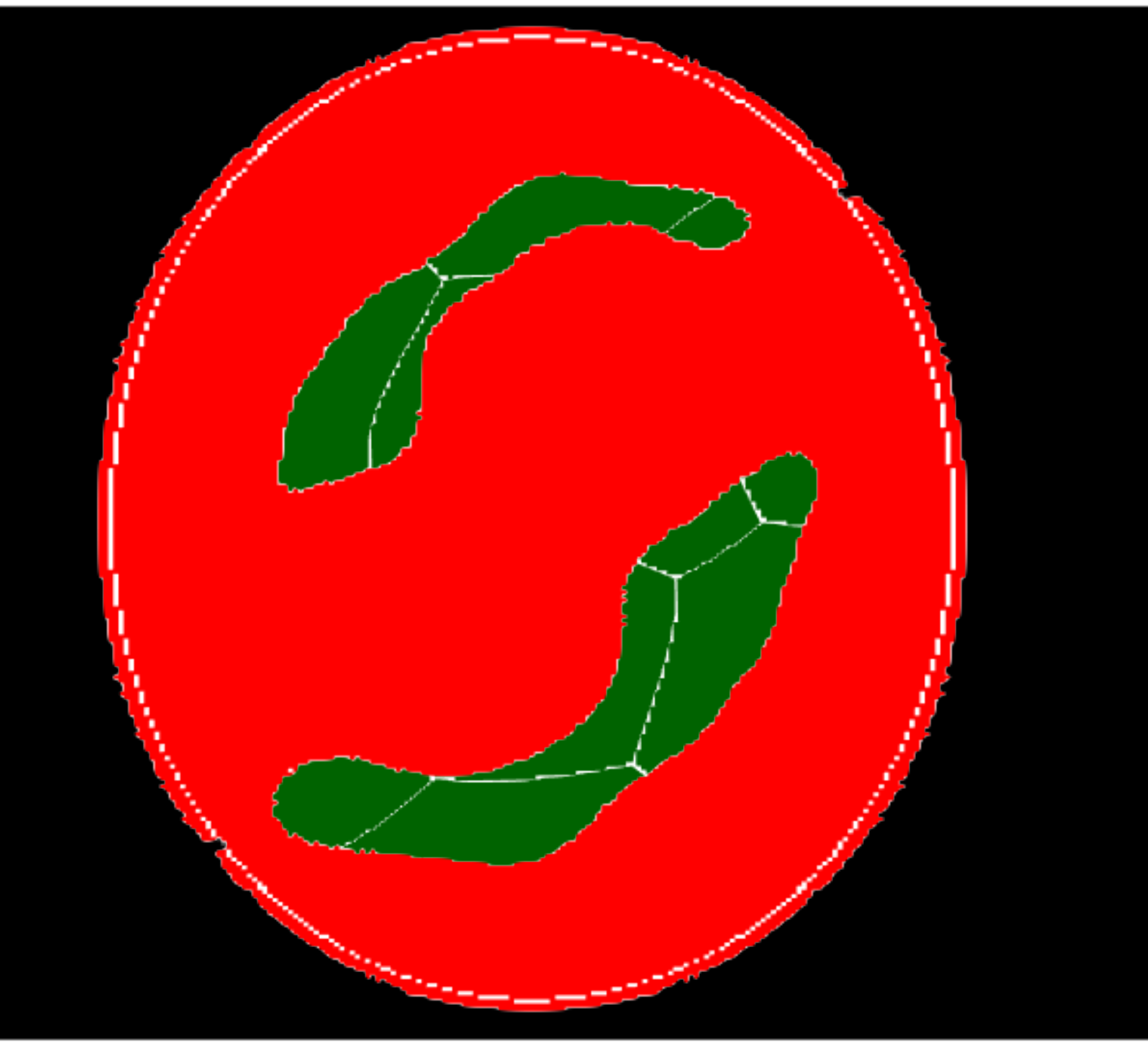}}
\caption{The parameters for nuclear reorganization process: $\beta_0=\frac 53$, $\beta_\phi=1$, $\beta_\psi=\frac 23$ with $\gamma=0$ for zero affinity.  Interface parameters are $\eps_\phi^2=0.001$ and  $\eps_{\psi}^2=0.005$.  $\bar{V}_m=\frac{\mbox{Nuclear Volume}}{N}$ and $\bar{v}_m=V_m\times 0.23$ where $m=8$  and time step $\delta t=10^{-2}$ with fixed nucleus. }\label{process}
\end{figure}

\subsection{Reduced nuclear size and the reorganization process}
In this subsection, we focus on the architecture reorganization process with reduced nuclear shape, and assess whether the nuclear shape is an  indispensable condition for the induction of a single cluster inverted architecture. 

We  introduce two sigmoid functions to describe the x-radius  and y-radius  of nuclear shape.
\begin{equation}
 r_x(t) =r_x(0) +\frac{\bar{r}_xt}{t + \alpha_1 e^{-\alpha_2 t}};\;  r_y(t) =r_y(0) +\frac{\bar{r}_yt}{t + \alpha_3 e^{-\alpha_4 t}},
 \end{equation}
where $\bar{r}_x$ and $\bar{r}_y$ are the decreasing rate of nuclear size and $\alpha_1$, $\alpha_2$, $\alpha_3$ and $\alpha_4$ are four positive constants which controls the decreasing scale with respect to time.  We consider that  the nucleus shape will decrease to be a circular or an  elliptical shape  with time evolution, and investigate how the nuclear size and shape influence the nuclear architecture reorganization process.  
The parameters of decreasing scale are $\alpha_1=\alpha_3=1$ and $\alpha_2=\alpha_4=0.01$. Numerical results  are computed by the linear scheme  
\eqref{stab:linear:scheme:1}-\eqref{stab:linear:scheme:5} with $\delta t=10^{-4}$.  We also increase the volume of each chromosome $V_m$ and volume of heterochromatin in each chromosome $v_m$ with time.  Snapshots  at $t=0, 0.05, 1, 5$ are depicted for different nuclear pattern in Fig.\;\ref{decrease_1}. It is observed from Fig.\;\ref{decrease_1} that both circular or elliptical shape will eventually achieve the one cluster inverted architecture. The second row of  Fig.\;\ref{decrease_1} displays  chromosome territories of the first row. So the nuclear size or shape are not indispensable condition for the nuclear architecture reorganization process.

\begin{figure}
\centering
\subfigure[$t=0$.]{
\includegraphics[width=0.22\textwidth,clip==]
{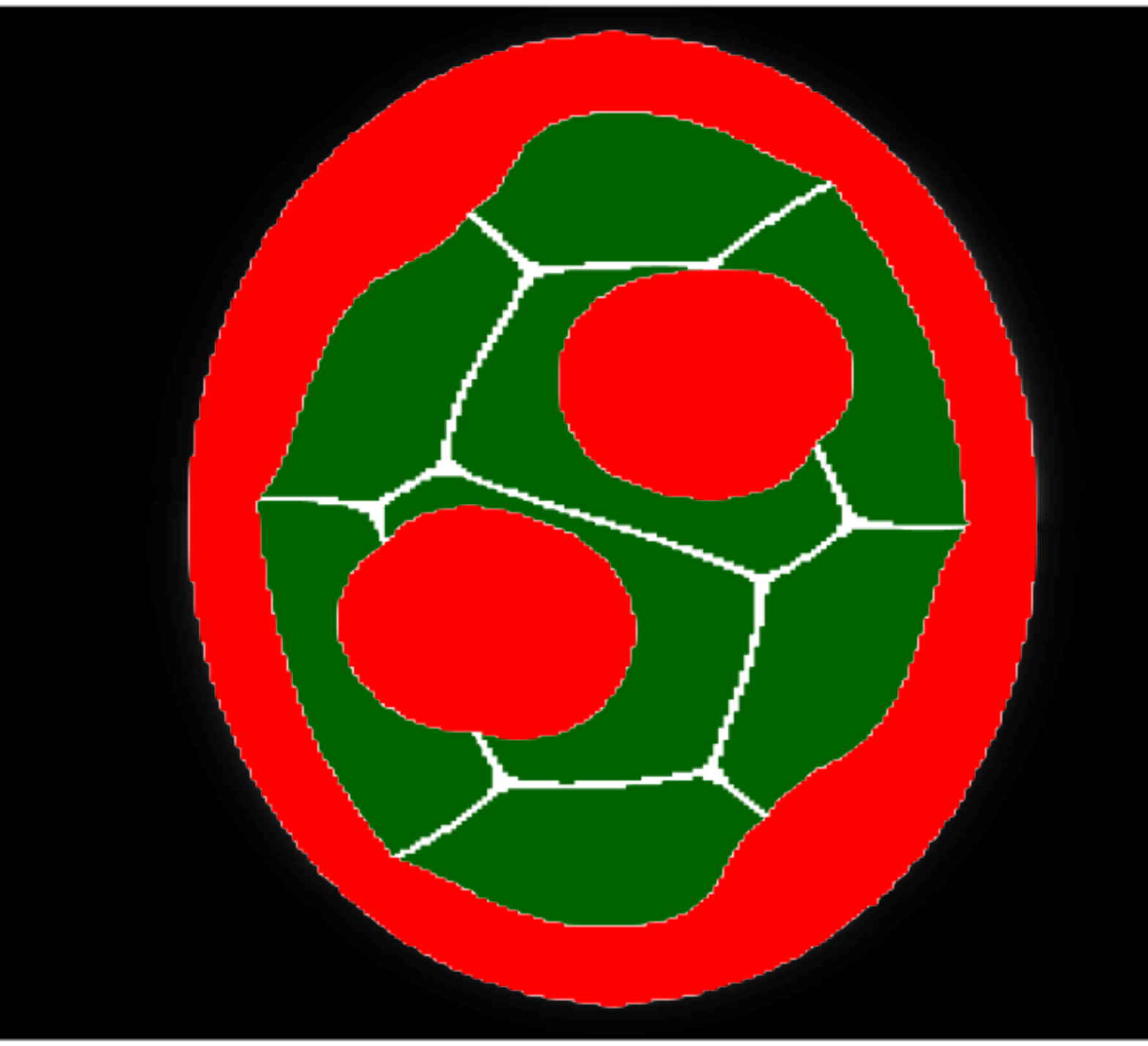}}\hskip 0cm
\subfigure[$t=0.05$.]{
\includegraphics[width=0.22\textwidth,clip==]{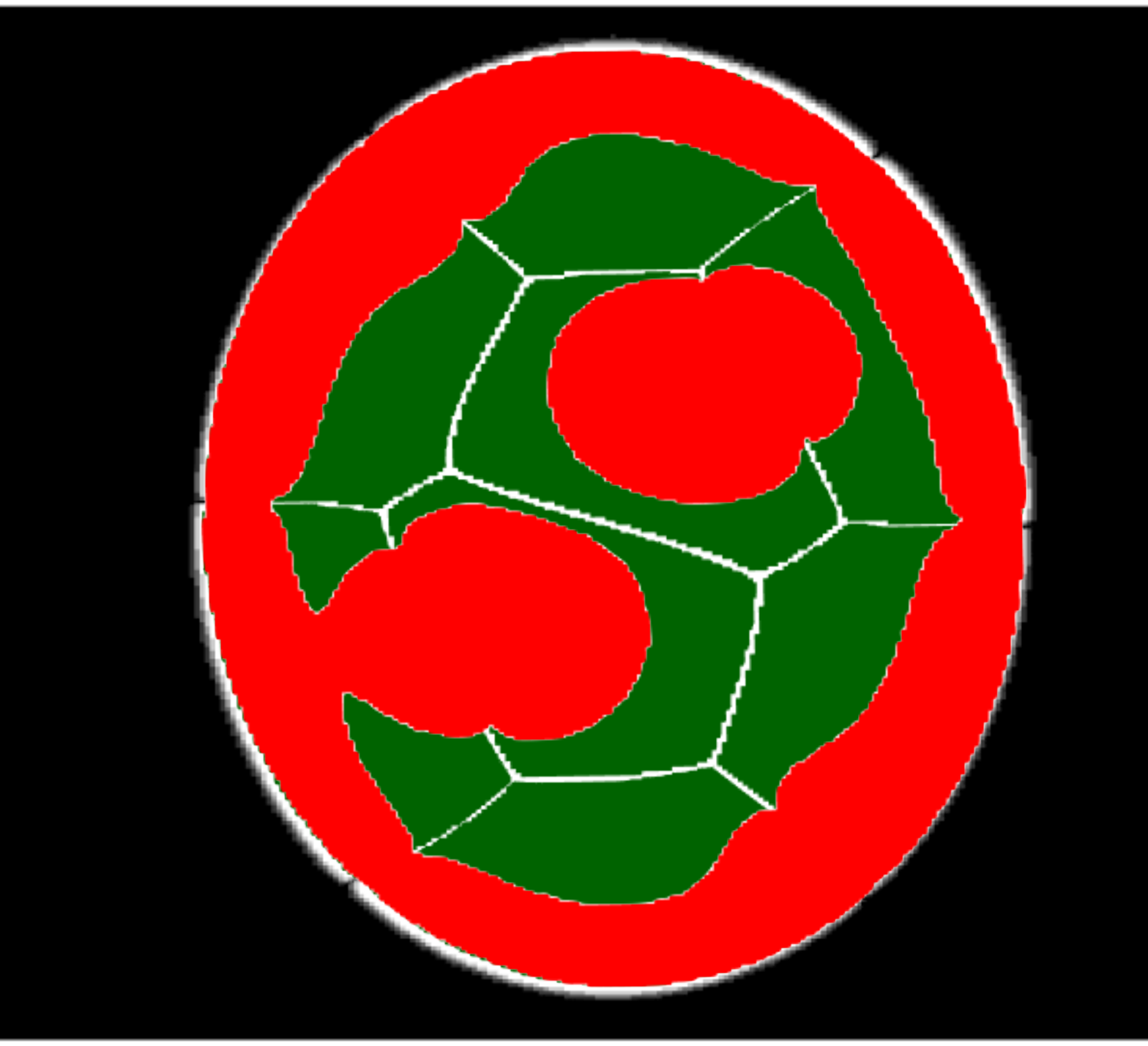}}
\subfigure[$t=1$.]{
\includegraphics[width=0.22\textwidth,clip==]{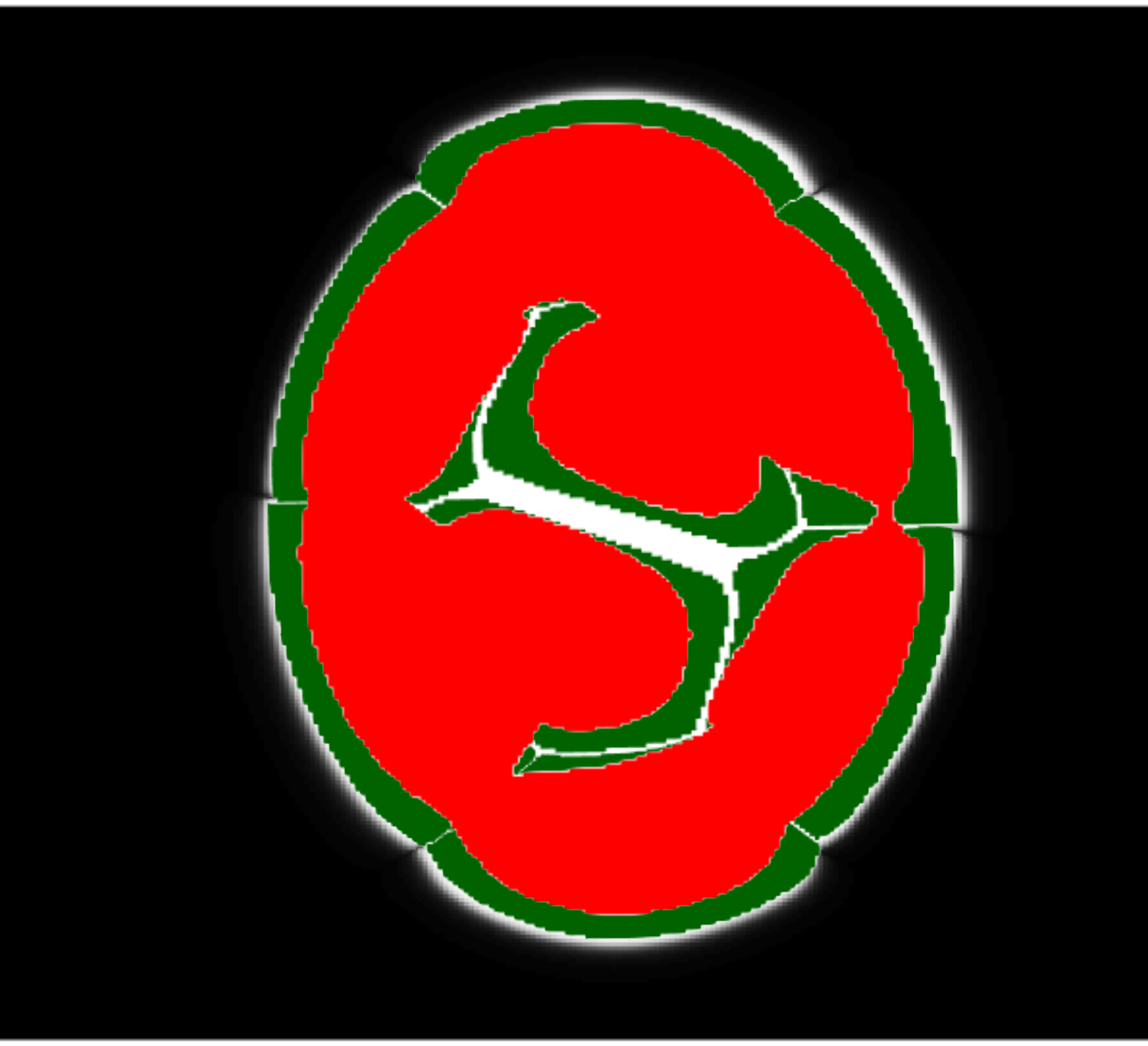}}
\subfigure[$t=5$.]{
\includegraphics[width=0.22\textwidth,clip==]{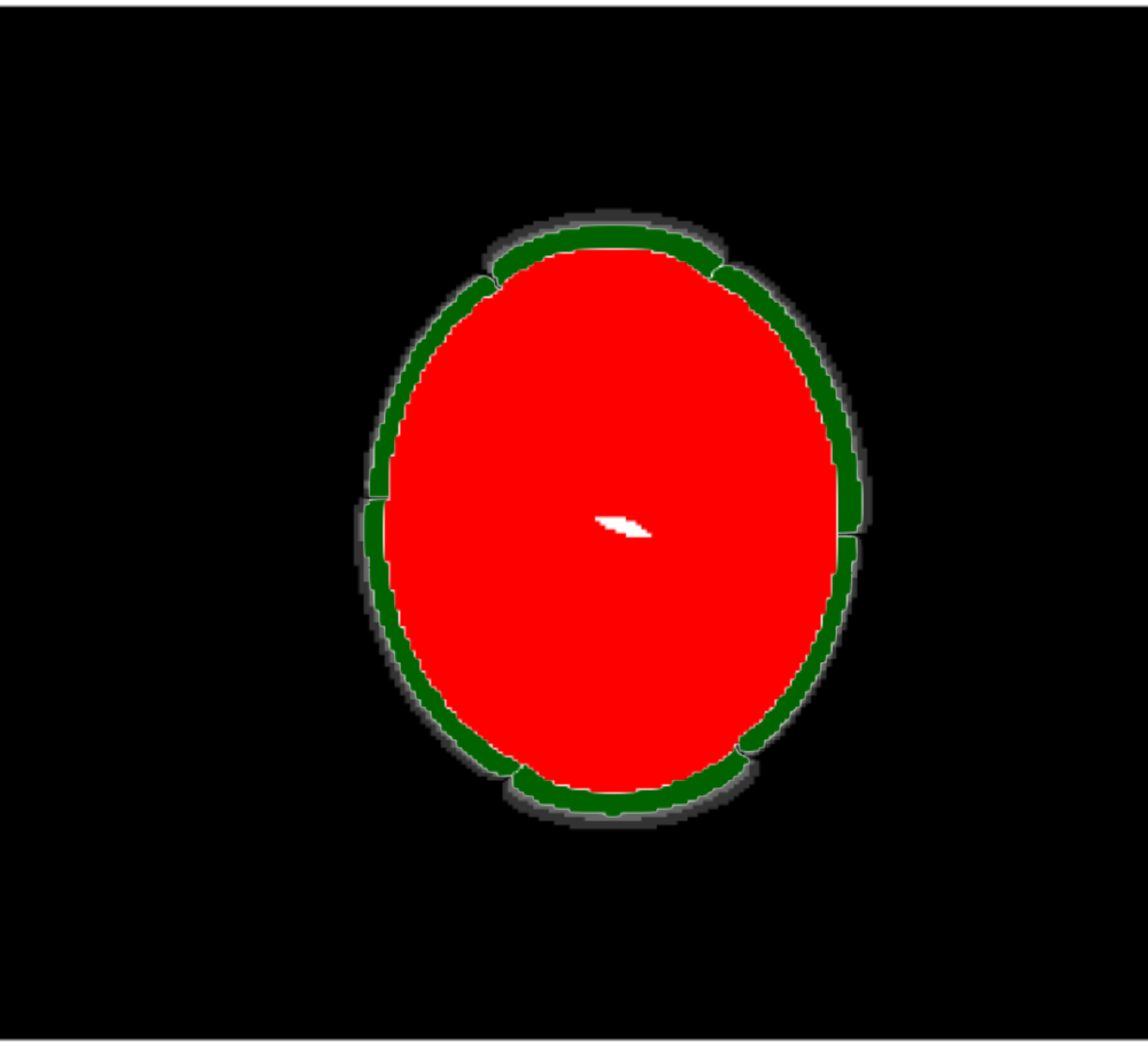}}\hskip 0cm
\subfigure[$t=0$.]{
\includegraphics[width=0.22\textwidth,clip==]
{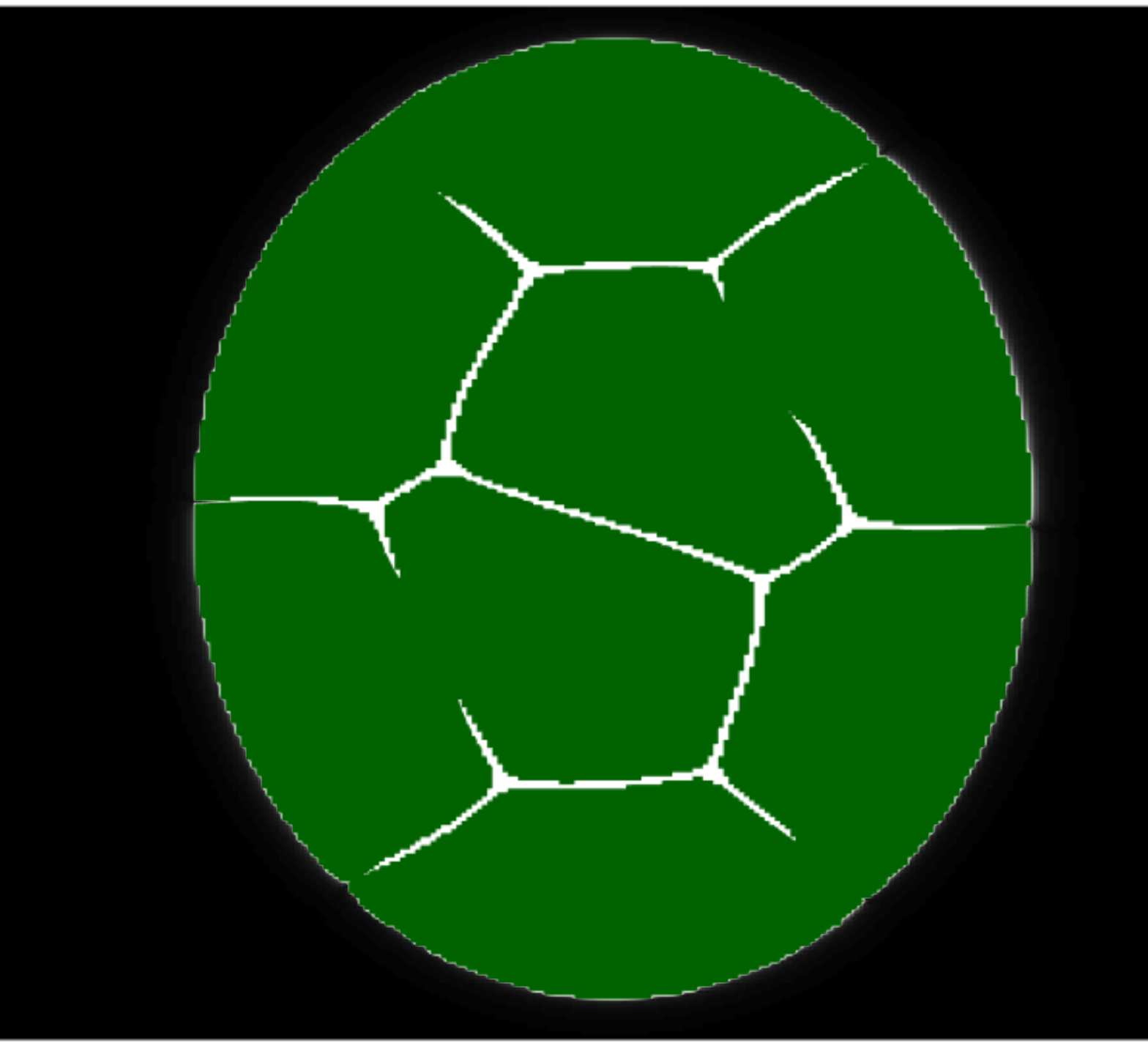}}\hskip 0cm
\subfigure[$t=0.05$.]{
\includegraphics[width=0.22\textwidth,clip==]{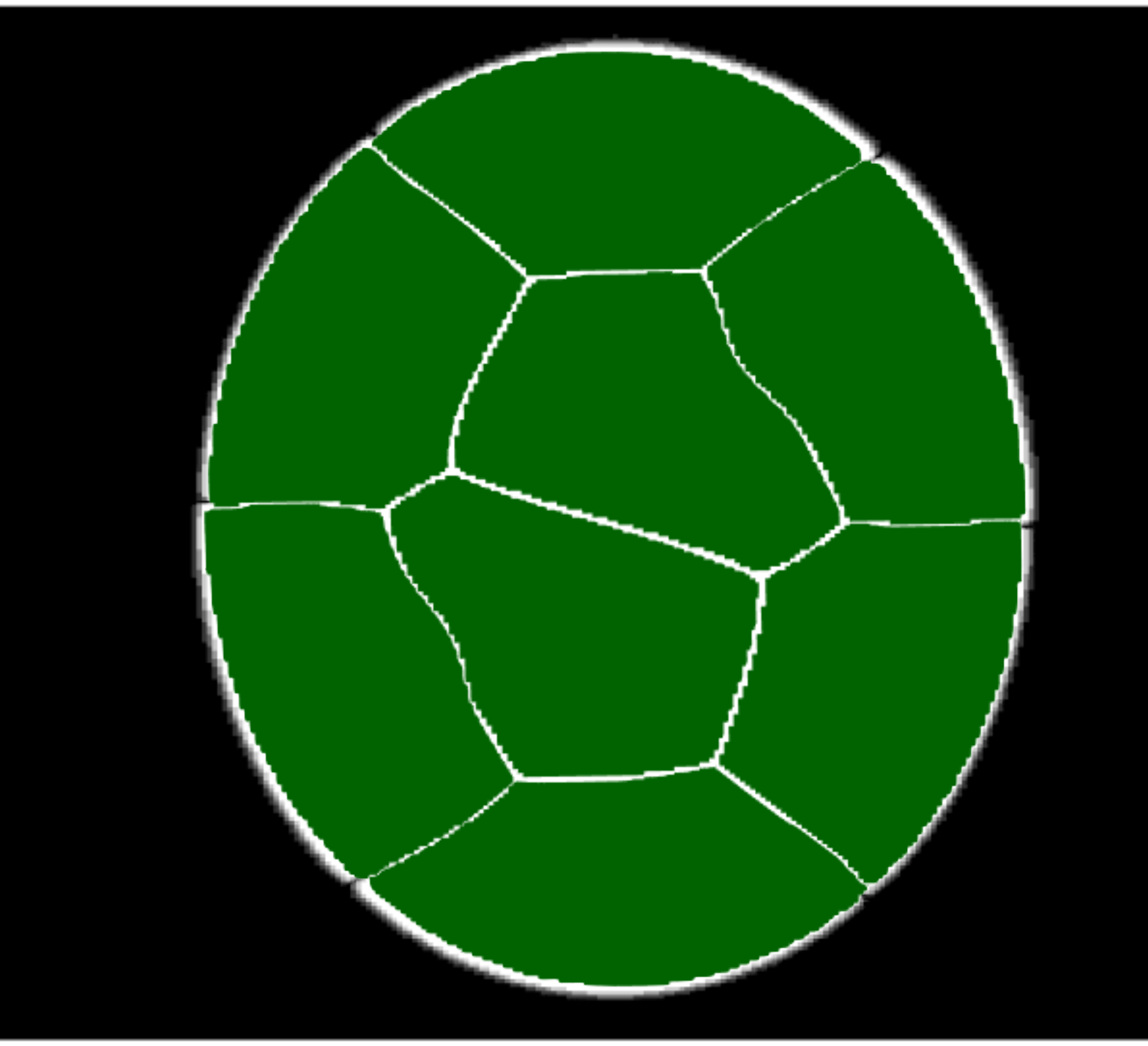}}
\subfigure[$t=1$.]{
\includegraphics[width=0.22\textwidth,clip==]{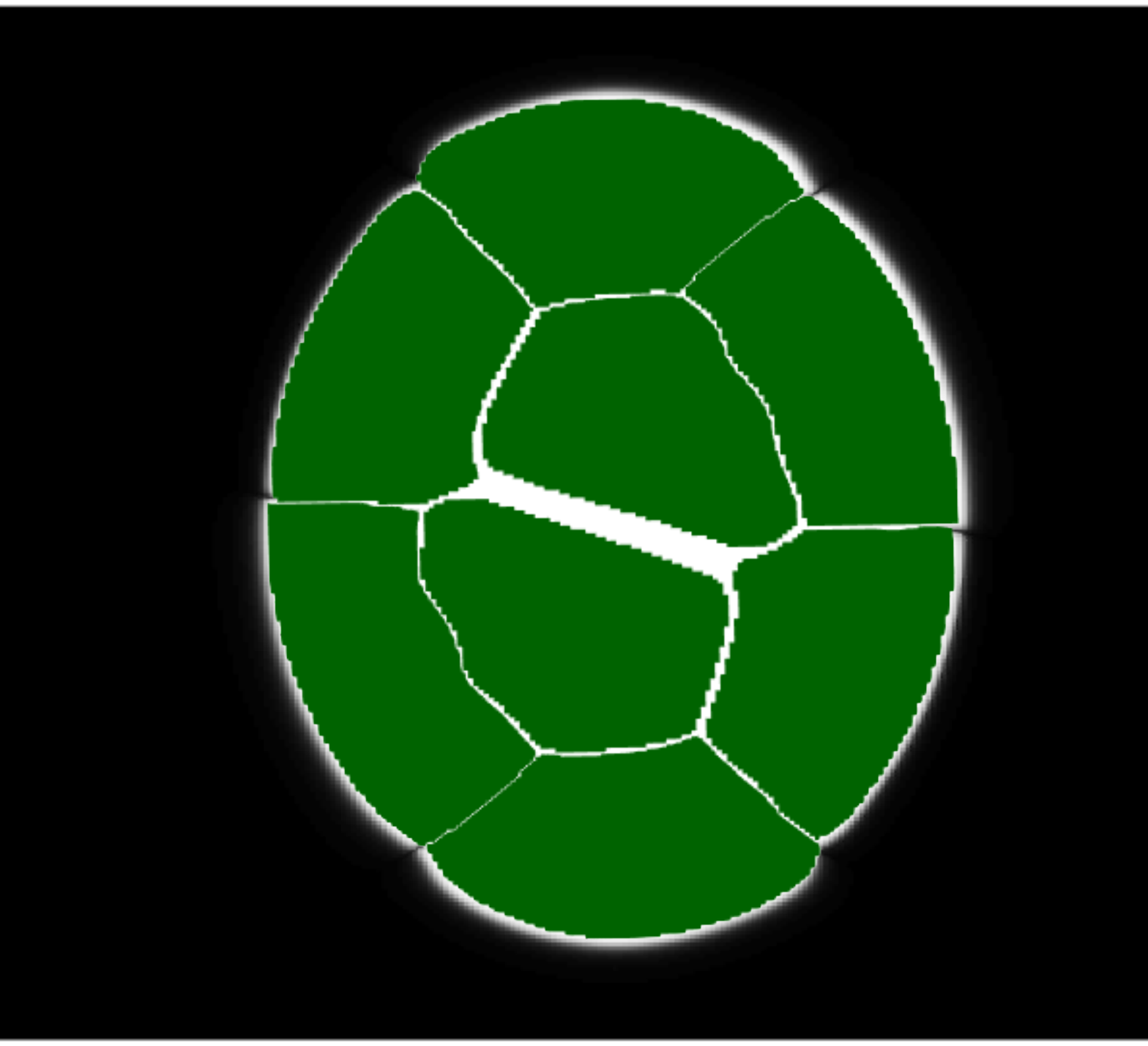}}
\subfigure[$t=5$.]{
\includegraphics[width=0.22\textwidth,clip==]{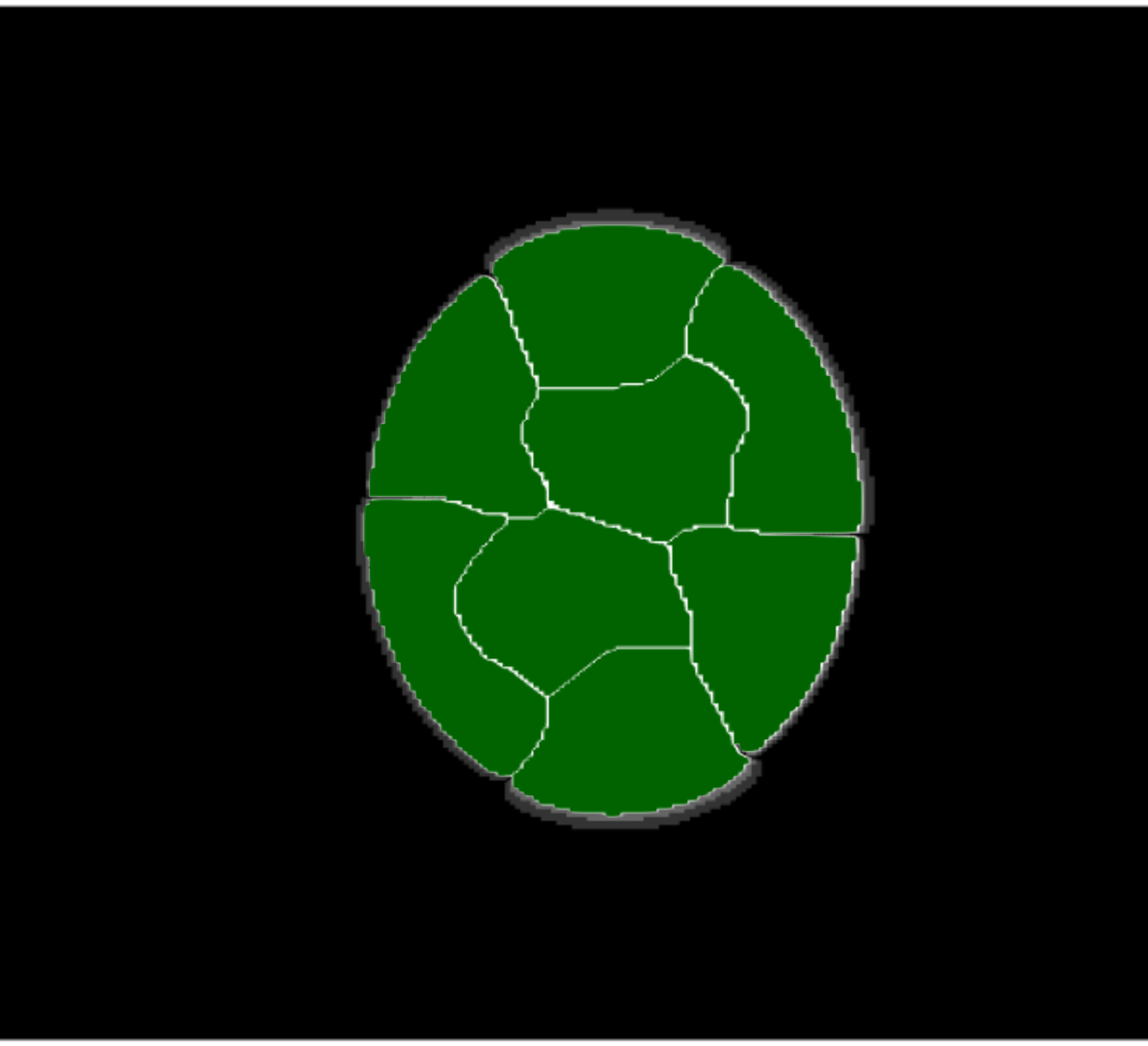}}\hskip 0cm
\subfigure[$t=0$.]{
\includegraphics[width=0.235\textwidth,clip==]
{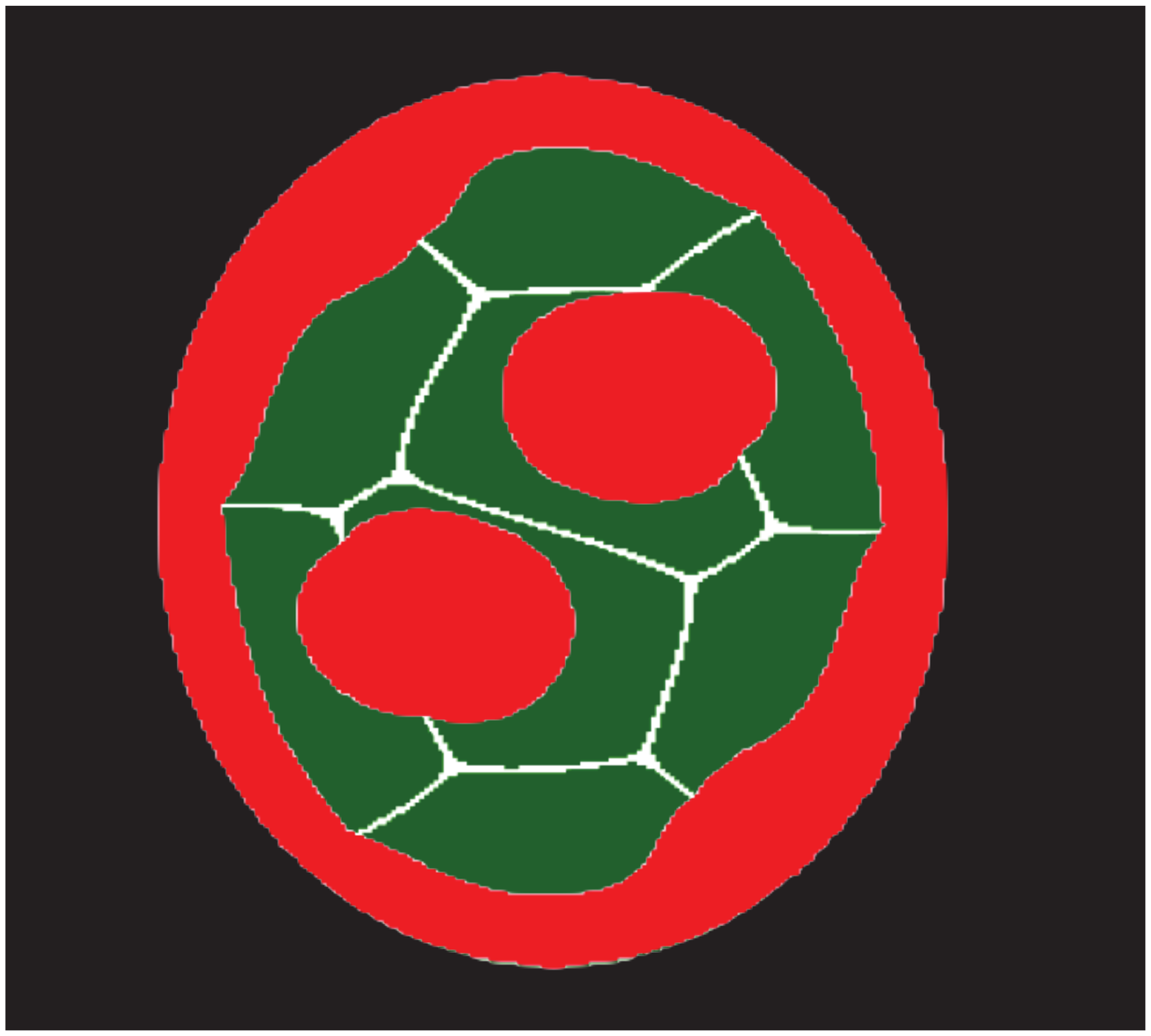}}\hskip 0cm
\subfigure[$t=0.05$.]{
\includegraphics[width=0.235\textwidth,clip==]{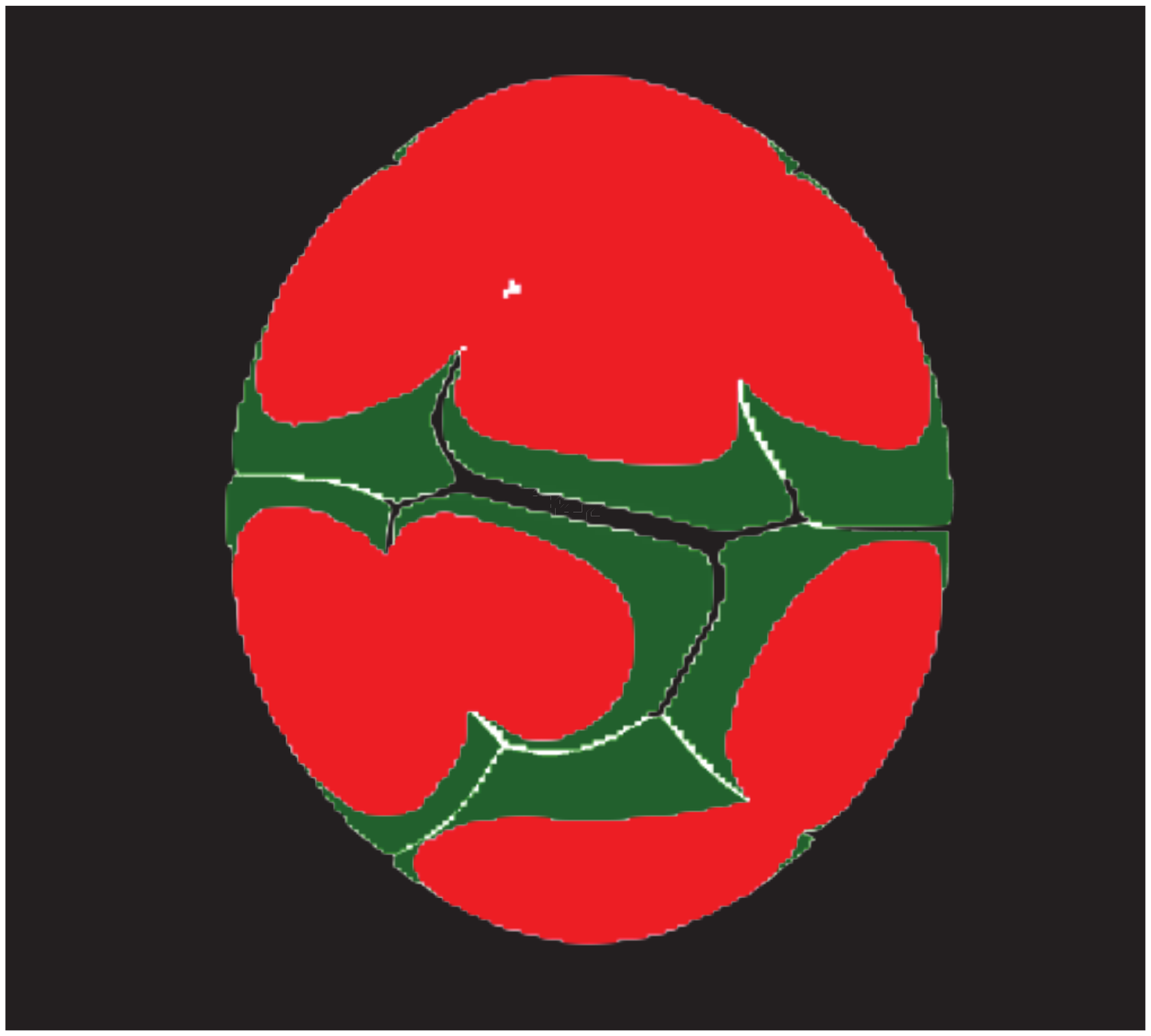}}
\subfigure[$t=1$.]{
\includegraphics[width=0.235\textwidth,clip==]{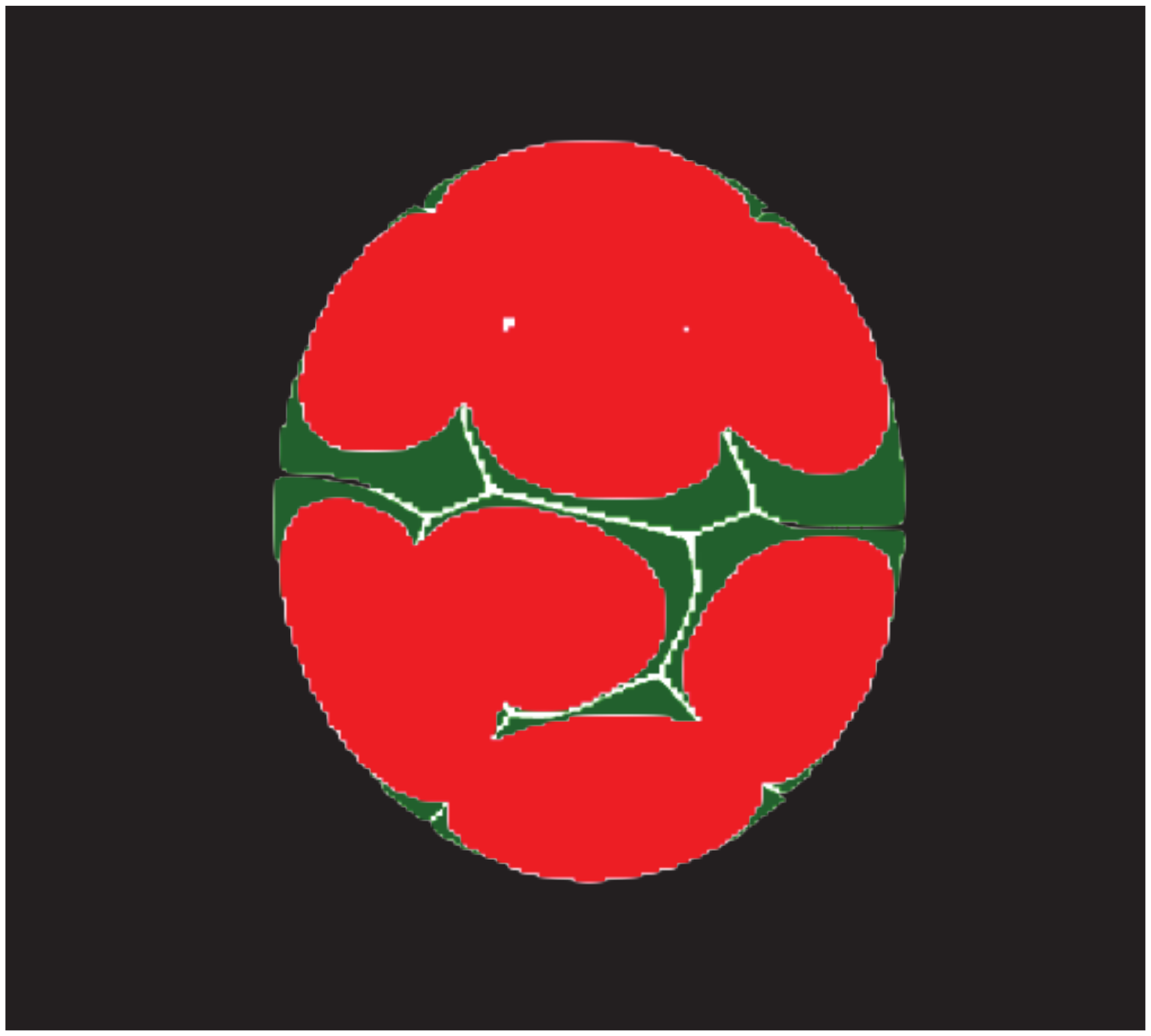}}
\subfigure[$t=5$.]{
\includegraphics[width=0.235\textwidth,clip==]{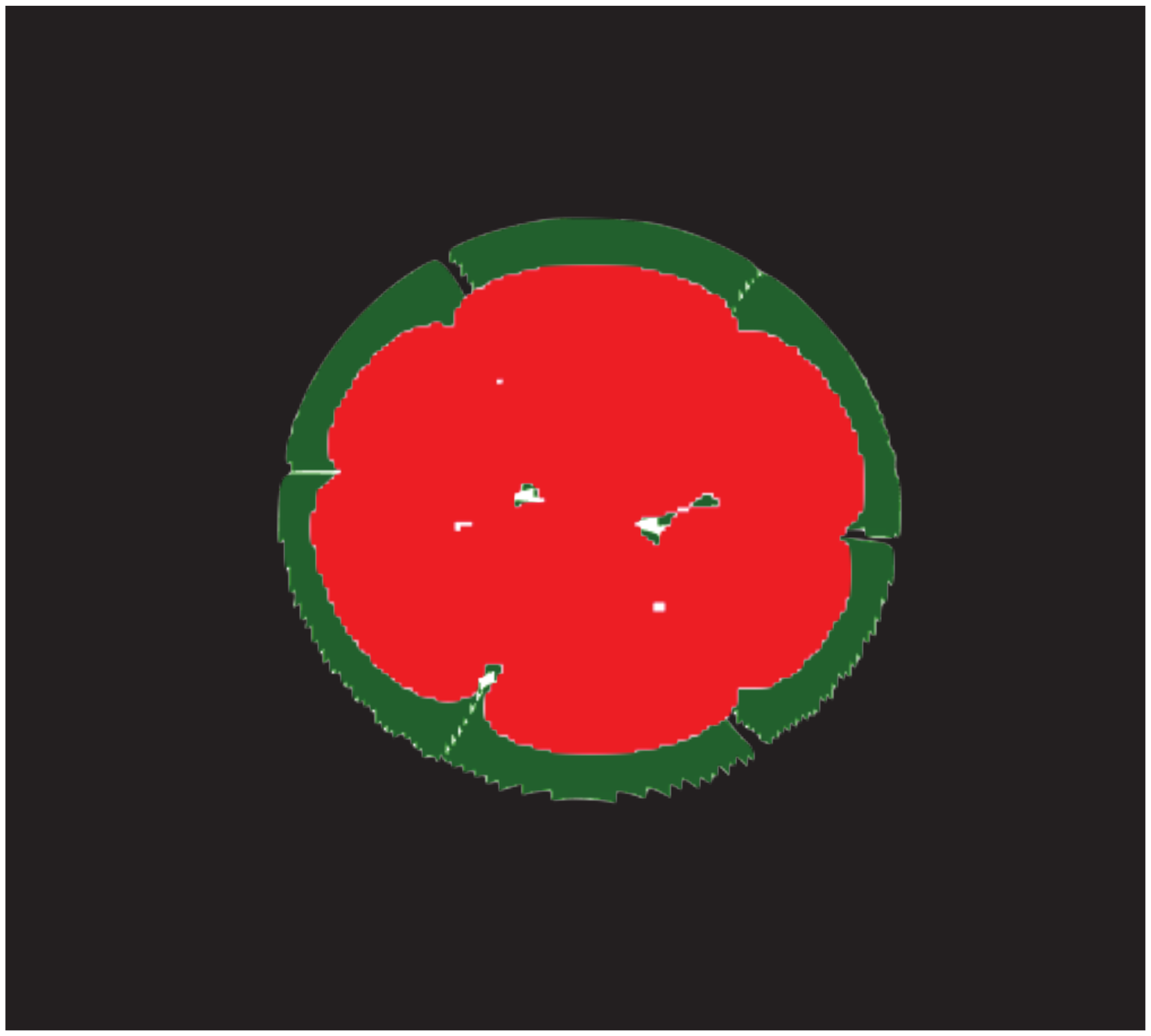}}\hskip 0cm
\caption{Decreased  Nucleus with parameters $(\beta_0,\beta_\phi,\beta_\psi)=(\frac{50}{3},50,\frac{50}{3})$. Interface width are $\eps^2_\phi=0.01$ and $\eps^2_\psi=0.05$.  $\bar{V}_m=$ Nuclear volume/$m$ and $\bar{v}_m=V_m\times [0.23]$, where $m=8$.}\label{decrease_1}
\end{figure}

\subsection{Inverted architecture reorganization for human beings}
In the previous subsections, we only considered $8$ chromosomes for drosophila,  and find that the deformation of nuclear size and shape are not sufficient conditions  for the nuclear architecture conversion. The absence of both LBR and lamin A/C expression $\gamma=0$  and the increase of heterochromatin rate are indispensable for inverted nuclear architecture.  Now we explore  the nuclear architecture with $46$  chromosomes for human beings. We also compute the numerical results by using the linear scheme \eqref{linear:scheme:1}-\eqref{linear:scheme:6} with $\delta t=10^{-3}$ and  examine the  effect of affinity in Fig.\;\ref{nuclear46_1}. It is observed from Fig.\;\ref{nuclear46_1} that  heterochromatin is shown to be distributed along the nuclear envelope with $\gamma=0.02$ with $46$ chromosome. However the heterochromatin accumulates at the territories between chromosome instead of in the region of nuclear envelope with $\gamma=0$. In Fig.\;\ref{nuclear46_2},  we decrease nuclear shape  and eliminate the  affinity of nuclear envelope, while increasing the heterochromatin conversion rate, and we observe that  one cluster  inverted architecture is formed at $t=2$.

\begin{figure}
\centering
\subfigure[$t=0$.]{
\includegraphics[width=0.3\textwidth,clip==]
{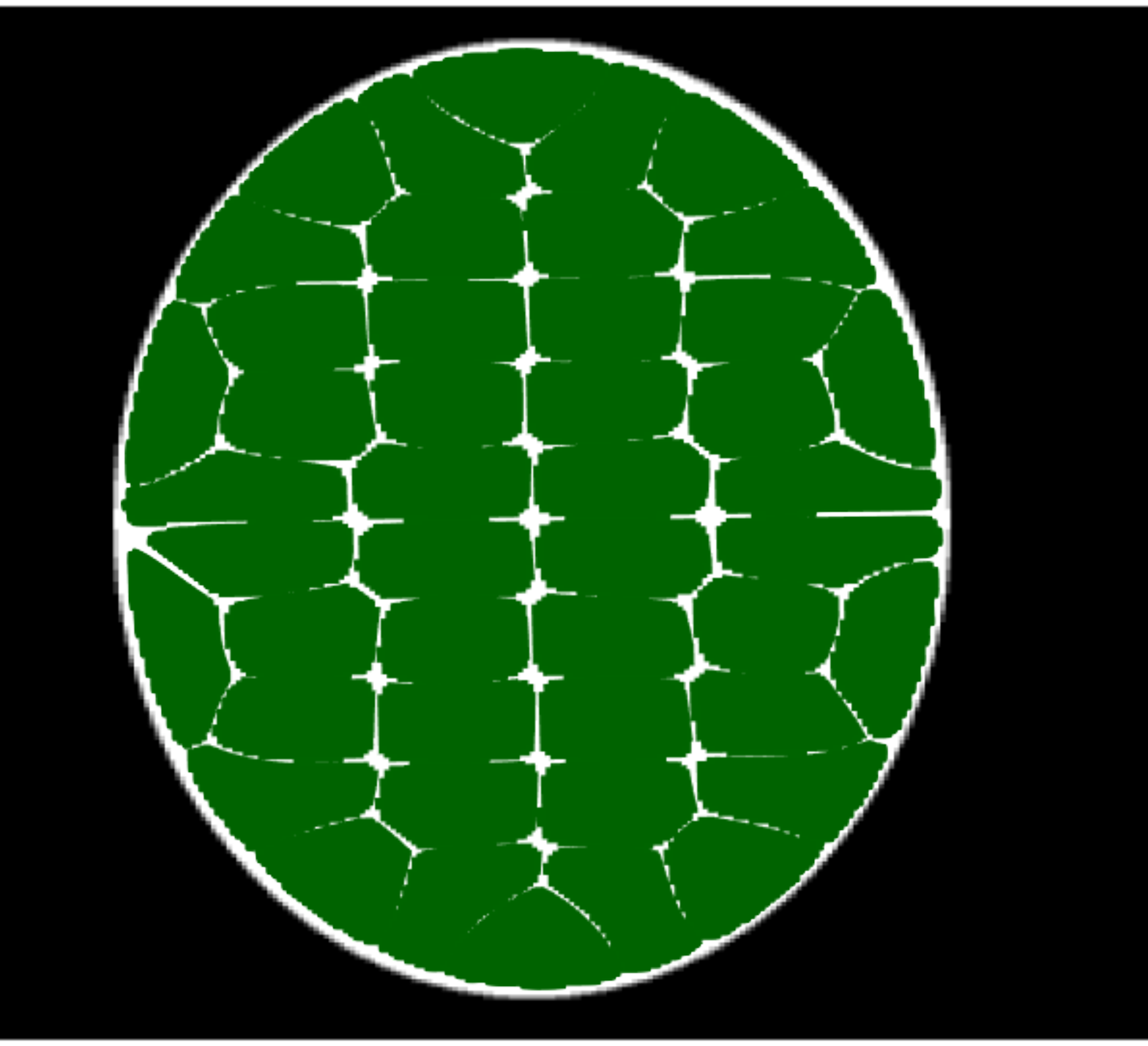}}\hskip 0cm
\subfigure[$t=5$.]{
\includegraphics[width=0.3\textwidth,clip==]{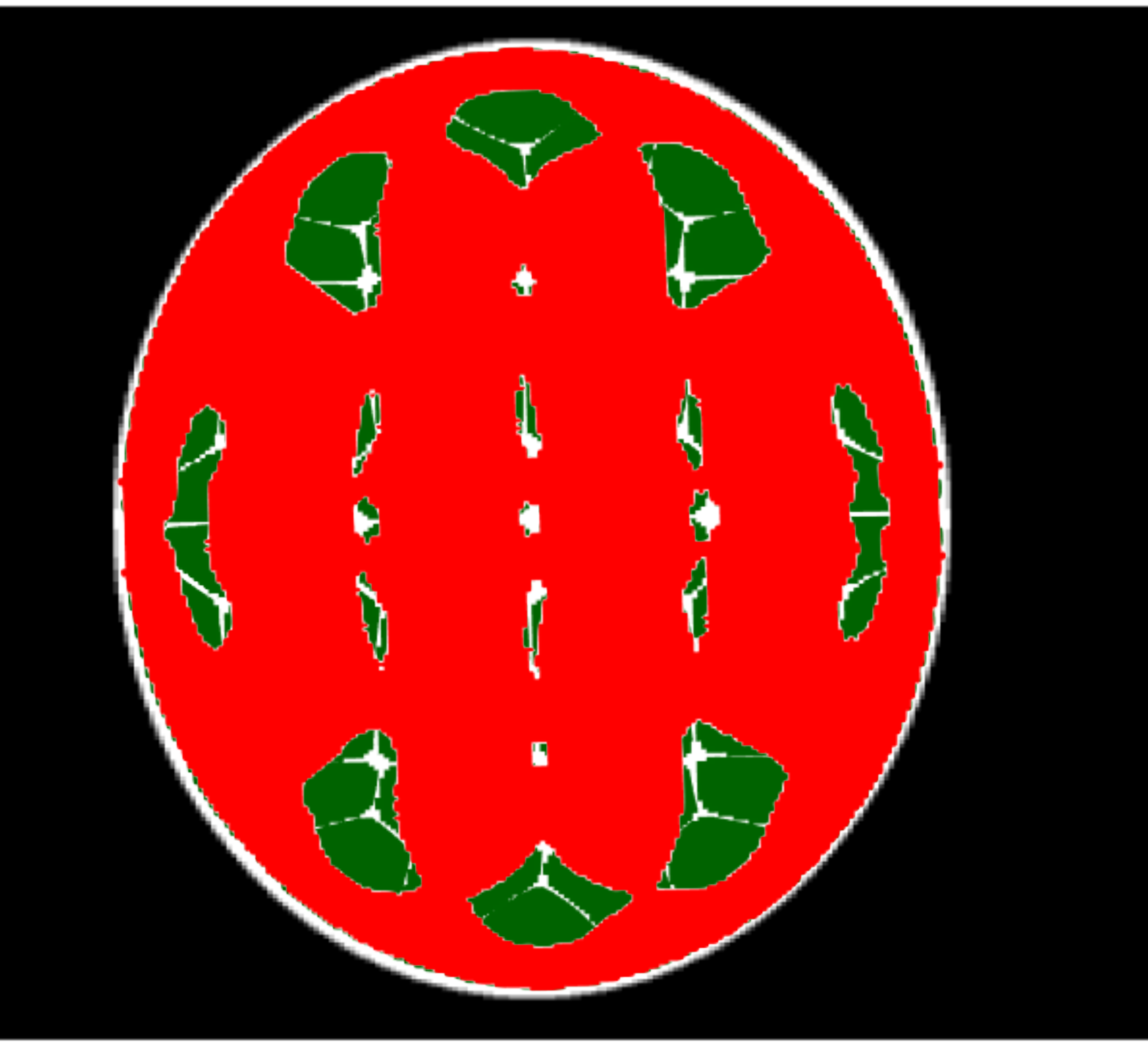}}
%\subfigure[$t=0$.]{
%\includegraphics[width=0.23\textwidth,clip==]
%{shape46_2.pdf}}\hskip 0cm
\subfigure[$t=5$.]{
\includegraphics[width=0.3\textwidth,clip==]{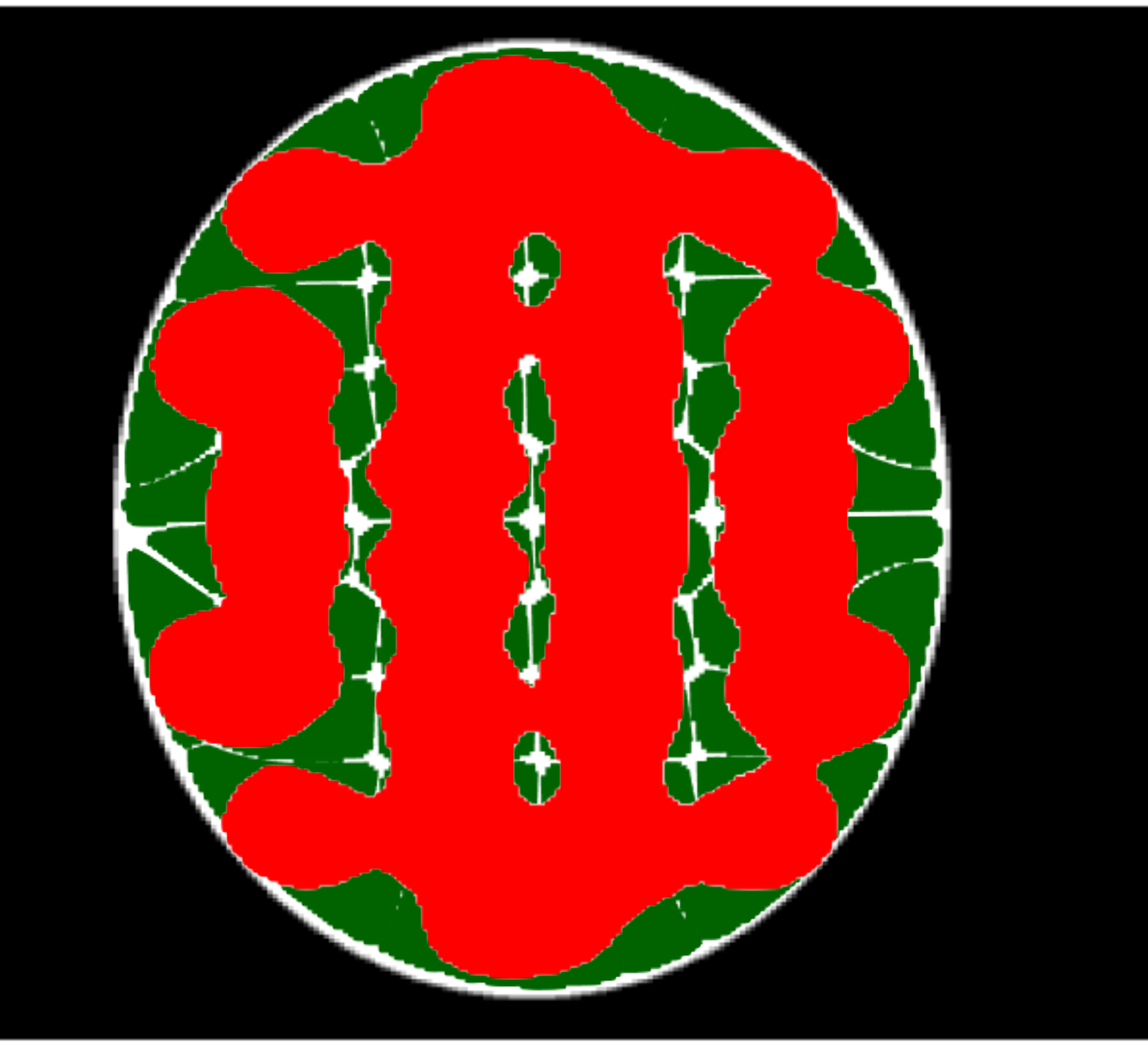}}
\caption{Fixed  Nucleus with parameters $(\beta_0,\beta_\phi,\beta_\psi)=(\frac{5}{3},1,\frac 23)$ and $\gamma=(0,0.02)$ where $m=46$. Interface width $\eps_\phi^2=0.01$ and $\eps^2_\psi=0.01$.}\label{nuclear46_1}
\end{figure}

\begin{figure}
\centering
\subfigure[$t=0$.]{
\includegraphics[width=0.25\textwidth,clip==]
{bc46.pdf}}\hskip 0cm
%\subfigure[$t=2$.]{
%\includegraphics[width=0.2\textwidth,clip==]
%{de46_1.pdf}}\hskip 0cm
\subfigure[$t=2$.]{
\includegraphics[width=0.30\textwidth,clip==]{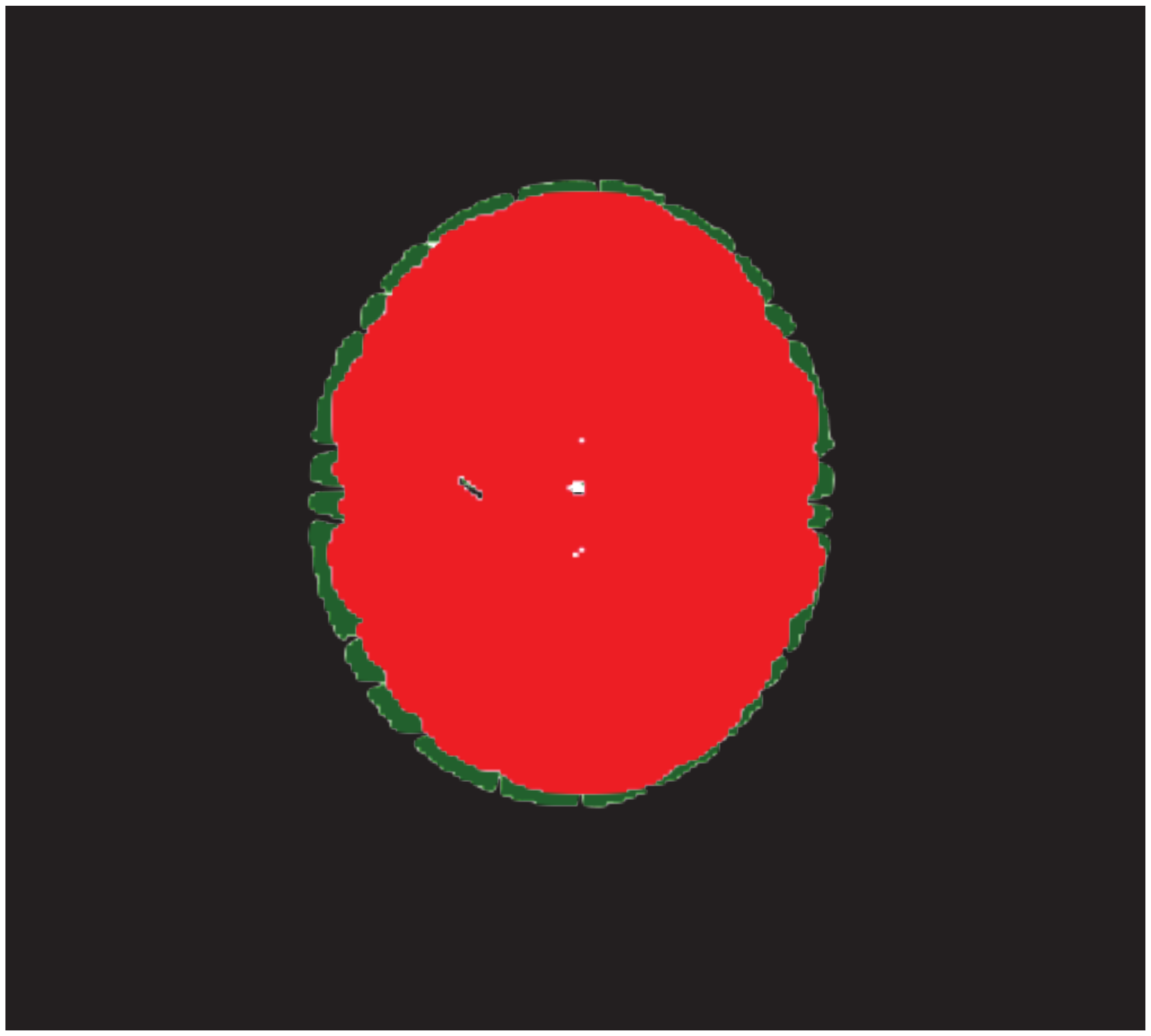}}
\subfigure[$t=2$.]{
\includegraphics[width=0.30\textwidth,clip==]
{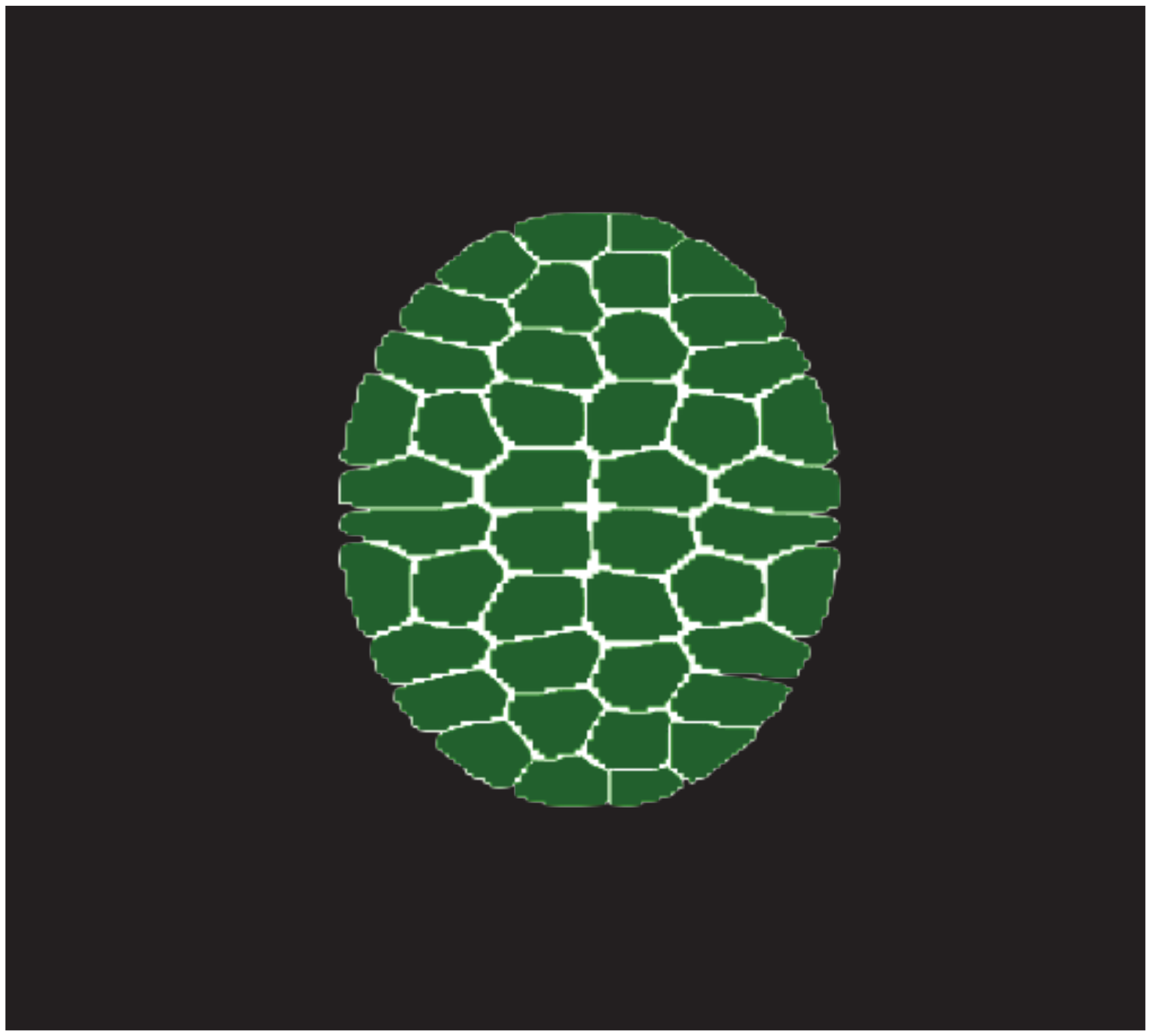}}\hskip 0cm
%\subfigure[$t=2$.]{
%\includegraphics[width=0.30\textwidth,clip==]{decrease3.pdf}}
\caption{Decreased  nucleus with parameters $(\beta_0,\beta_\phi,\beta_\psi)=(\frac{5}{3},1,\frac 23)$ and $\gamma=0$ where $m=46$. Interface width $\eps_\phi^2=0.01$ and $\eps^2_\psi=0.01$.}\label{nuclear46_2}
\end{figure}

\begin{comment}
\begin{figure}
\centering
\subfigure[$t=0$.]{
\includegraphics[width=0.23\textwidth,clip==]{bc46.pdf}}
\subfigure[$t=30$.]{
\includegraphics[width=0.23\textwidth,clip==]{f46_t30.pdf}}
\subfigure[$t=40$.]{
\includegraphics[width=0.23\textwidth,clip==]{f46_t40.pdf}}
\subfigure[$t=60$.]{
\includegraphics[width=0.23\textwidth,clip==]{f46_t60.pdf}}
\caption{Fixed  Nucleus with parameters $(\beta_0,\beta_\phi,\beta_\psi)=(\frac{5}{3},1,\frac 23)$ and $\gamma=(0,0.02)$. $\bar{V}_m=$ Nuclear volume/$m$ and $\bar{v}_m=V_m\times [0.23]$, where $m=8$.}\label{nuclear46_3}
\end{figure}
\end{comment}

\section{Concluding remarks}
 Specific features of nuclear architecture  are closely related to  the functional organization of the nucleus. Within nucleus, 
 chromatin consists of two forms, heterochromatin and euchromatin. The conventional nuclear architecture is observed 
 when heterochromatin is enriched at nuclear periphery, and it represents the primary structure in the majority of eukaryotic cells, 
 including the rod cells of diurnal mammals. In contrast to this, the inverted nuclear architecture is observed when the heterochromatin is 
 distributed at the center of the nucleus, which occurs in the rod cells of nocturnal mammals. 
 The conventional architecture can transform into  the inverted architecture during nuclear reorganization process.

We developed in this paper a new phase field model  with Lagrange multipliers to simulate the nuclear architecture reorganization process. Introducing Lagrange multipliers enables us to preserve the specific physical and geometrical constraints for the biological events.
We developed several efficient time discretization schemes for the constrained gradient system. One is  a  full linear scheme which can only  preserve volume constrains with second order accuracy, but it is very easy to solve. The other two are weakly nonlinear scheme which can exactly preserve non-local constraints, and one of them is also  unconditionally energy stable. The price we pay for the exact preservation of geometric constraints is that we need to solve a nonlinear algebraic system for the Lagrange multipliers, which can be solved at negligible cost but may require the time step to be sufficiently small. These time discretization schemes can be used with any consistent Galerkin type discretization in space.

We presented several simulations using our proposed schemes for drosophila and human beings with $8$ chromosomes and $46$ chromosomes. Our results  indicate that the increase of  heterochromatin   conversion rate and the absence of affinity between nuclear envelope and heterochromatin are sufficient for the formation of the inverted architecture during the nuclear architecture reorganization process, while nuclear size and shape are not indispensable for the formation of the single hetero-cluster inverted architecture.

\bibliography{references}
\end{document}